\documentclass[10pt]{amsart}

\usepackage[T1]{fontenc}
\usepackage[utf8]{inputenc} 

\usepackage{amsmath}
\usepackage{amssymb}
\usepackage{amsthm}
\usepackage[shortlabels]{enumitem}
\usepackage{moreenum}
\usepackage[dvipsnames]{xcolor}
\usepackage[top=2.7cm, bottom=2.7cm, left=2.4cm, right=2.4cm]{geometry}
\usepackage{tikz}
\usepackage{tikz-cd}
\usepackage{comment}
\usepackage[colorlinks=true, linkcolor=blue, citecolor=magenta]{hyperref}
\usepackage[normalem]{ulem}

\theoremstyle{definition}
\newtheorem{defn}{Definition}[section]
\newtheorem{exm}[defn]{Example}
\newtheorem{rmk}[defn]{Remark}
\newtheorem{con}[defn]{Construction}

\newtheorem*{exm*}{Example}
\newtheorem*{rmk*}{Remark}
\newtheorem*{prop*}{Proposition}
\newtheorem{thmx}{Theorem}

\theoremstyle{plain}
\newtheorem{thm}[defn]{Theorem}
\newtheorem{prop}[defn]{Proposition}
\newtheorem{lem}[defn]{Lemma}
\newtheorem{cor}[defn]{Corollary}

\newcommand{\kk}{\Bbbk}
\newcommand{\kmod}{\kk\mathsf{Mod}}

\newcommand{\ZZ}{\mathbb{Z}}
\newcommand{\NN}{\mathbb{N}}
\newcommand{\RP}{\mathbb{R}P}

\newcommand{\ev}{\mathrm{ev}}
\newcommand{\id}{\mathrm{id}}

\newcommand{\End}{\mathrm{End}}
\newcommand{\Sk}{\mathrm{Sk}}
\DeclareMathOperator{\Hom}{Hom}
\renewcommand{\sl}{\mathfrak{sl}}
\newcommand{\gl}{\mathfrak{gl}}
\newcommand{\op}{{\mathrm{op}}}
\newcommand{\C}{\mathsf{C}}
\newcommand{\cob}{\mathsf{Cob}}
\newcommand{\ucob}{\mathsf{UCob}_2}

\newcommand{\dTL}{\mathsf{dTL}}
\newcommand{\skcat}{\mathsf{SkCat}}
\newcommand{\BN}{\mathsf{BN}}
\newcommand{\A}{\mathsf{A}}
\newcommand{\B}{\mathsf{B}}
\newcommand{\bim}{\mathsf{Bim}}
\newcommand{\sktft}{\mathsf{AFK}}

\newcommand{\kcat}{\kk\mathsf{Cat}}
\newcommand{\mor}{\mathsf{Mor}}
\newcommand{\Tr}{\mathsf{Tr}}
\newcommand{\Mat}{\mathsf{Mat}}
\newcommand{\Kar}{\mathsf{Kar}}

\newcommand{\eval}{\mathrm{eval}}

\newcommand{\arcs}{\mathrm{arcs}}

\definecolor{amaranth}{rgb}{0.9, 0.17, 0.31}

\newcommand{\cutdup}{	\begin{tikzpicture}[baseline={([yshift=-.5ex]current bounding box.center)}, scale=0.6]
		\draw[amaranth,fill =amaranth,fill opacity=0.3, line cap=round, semithick] (-0.5,0) to[out=90, in=180] (0,0.75) to[out=0,in=90] (0.5,0) to[out=90, in=90, looseness=0.4] (-0.5,0);
		\draw[amaranth,fill=amaranth, fill opacity=0.5, line cap=round, semithick] (-0.5,0) to[out=-90,in=-90,looseness=0.4] (0.5,0) to[out=90, in=90, looseness=0.4] (-0.5,0);
		
		\draw[amaranth,fill =amaranth,fill opacity=0.3, line cap=round, semithick, shift={(0,2)},rotate=180] (-0.5,0) to[out=90, in=180] (0,0.75) to[out=0,in=90] (0.5,0) to[out=90, in=90, looseness=0.4] (-0.5,0);
		\draw[amaranth, fill] (0,1.55) circle (0.06);
		\draw[amaranth,fill=amaranth, fill opacity=0.5, line cap=round, semithick, shift={(0,2)}, rotate=180] (-0.5,0) to[out=-90,in=-90,looseness=0.4] (0.5,0) to[out=90, in=90, looseness=0.4] (-0.5,0);	
\end{tikzpicture}}

\newcommand{\cutddown}{	\begin{tikzpicture}[baseline={([yshift=-.5ex]current bounding box.center)}, scale=0.6]
		\draw[amaranth,fill =amaranth,fill opacity=0.3, line cap=round, semithick] (-0.5,0) to[out=90, in=180] (0,0.75) to[out=0,in=90] (0.5,0) to[out=90, in=90, looseness=0.4] (-0.5,0);
		\draw[amaranth, fill] (0,0.45) circle (0.06);
		\draw[amaranth,fill=amaranth, fill opacity=0.5, line cap=round, semithick] (-0.5,0) to[out=-90,in=-90,looseness=0.4] (0.5,0) to[out=90, in=90, looseness=0.4] (-0.5,0);
		
		\draw[amaranth,fill =amaranth,fill opacity=0.3, line cap=round, semithick, shift={(0,2)},rotate=180] (-0.5,0) to[out=90, in=180] (0,0.75) to[out=0,in=90] (0.5,0) to[out=90, in=90, looseness=0.4] (-0.5,0);
		\draw[amaranth,fill=amaranth, fill opacity=0.5, line cap=round, semithick, shift={(0,2)}, rotate=180] (-0.5,0) to[out=-90,in=-90,looseness=0.4] (0.5,0) to[out=90, in=90, looseness=0.4] (-0.5,0);	
\end{tikzpicture}}

\newcommand{\neck}{	\begin{tikzpicture}[baseline={([yshift=-.5ex]current bounding box.center)}, scale=0.6]
		\draw[amaranth,fill=amaranth, fill opacity=0.5, line cap=round, semithick] (-0.5,0) to[out=-90,in=-90,looseness=0.4] (0.5,0) to[out=90, in=90, looseness=0.4] (-0.5,0);
		\draw[amaranth,fill=amaranth, fill opacity=0.3, line cap=round, semithick] (-0.5,0) to (-0.5,2) to[out=-90,in=-90,looseness=0.4] (0.5,2) to (0.5,0) to[out=90,in=90,looseness=0.4] (-0.5,0);
		
		\draw[amaranth,fill=amaranth, fill opacity=0.5, line cap=round, semithick, shift={(0,2)}, rotate=180] (-0.5,0) to[out=-90,in=-90,looseness=0.4] (0.5,0) to[out=90, in=90, looseness=0.4] (-0.5,0);	
\end{tikzpicture}}

\newcommand{\sheet}{\begin{tikzpicture}[baseline={([yshift=-.7ex]current bounding box.center)}, scale=0.6]
		\draw[amaranth, fill=amaranth, fill opacity=0.3,  line cap=round, semithick]
		(0,0) rectangle (1.2,1.2);
\end{tikzpicture}}

\newcommand{\ddsheet}{\begin{tikzpicture}[baseline={([yshift=-.7ex]current bounding box.center)}, scale=0.6]
		\draw[amaranth, fill=amaranth, fill opacity=0.3,  line cap=round, semithick]
		(0,0) rectangle (1.2,1.2);
		\draw[amaranth, fill] (0.45,0.6) circle (0.06);
		\draw[amaranth, fill] (0.75,0.6) circle (0.06);
\end{tikzpicture}}

\newcommand{\sphere}{\begin{tikzpicture}[baseline={([yshift=-.7ex]current bounding box.center)}, scale=0.6]
		\draw[amaranth, fill=amaranth, fill opacity=0.3,  line cap=round, semithick] (0,0) circle (0.7);
		\draw[amaranth, line cap=round, semithick] (-0.7,0) to[out=-60, in=-120, looseness=0.35] (0.7,0);
		\draw[amaranth, line cap=round, densely dotted, semithick] (-0.7,0) to[out=60, in=120, looseness=0.35] (0.7,0);
\end{tikzpicture}}

\newcommand{\dsphere}{{\begin{tikzpicture}[baseline={([yshift=-.7ex]current bounding box.center)}, scale=0.6]
			\draw[amaranth, fill=amaranth, fill opacity=0.3,  line cap=round, semithick] (0,0) circle (0.7);
			\draw[amaranth, line cap=round, semithick] (-0.7,0) to[out=-60, in=-120, looseness=0.35] (0.7,0);
			\draw[amaranth, line cap=round, densely dotted, semithick] (-0.7,0) to[out=60, in=120, looseness=0.35] (0.7,0);
			\draw[amaranth, fill] (0,0.4) circle (0.06);	
\end{tikzpicture}}}

\newcommand{\asphere}{{\begin{tikzpicture}[baseline={([yshift=-.7ex]current bounding box.center)}, scale=0.6]
			\draw[amaranth, fill=amaranth, fill opacity=0.3,  line cap=round, semithick] (0,0) circle (0.7);
			\draw[amaranth, line cap=round, semithick] (-0.7,0) to[out=-60, in=-120, looseness=0.5] (0.7,0);
			\draw[amaranth, line cap=round, densely dotted, semithick] (-0.7,0) to[out=60, in=120, looseness=0.35] (0.7,0);
			\node at (0,0.4) {\scriptsize $a$};
\end{tikzpicture}}}

\newcommand{\aneck}{	\begin{tikzpicture}[baseline={([yshift=-.5ex]current bounding box.center)}, scale=0.6]
		\draw[amaranth,fill=amaranth, fill opacity=0.5, line cap=round, semithick] (-0.5,0) to[out=-90,in=-90,looseness=0.4] (0.5,0) to[out=90, in=90, looseness=0.4] (-0.5,0);
		\draw[amaranth,fill=amaranth, fill opacity=0.3, line cap=round, semithick] (-0.5,0) to (-0.5,2) to[out=-90,in=-90,looseness=0.4] (0.5,2) to (0.5,0) to[out=90,in=90,looseness=0.4] (-0.5,0);
		
		\draw[amaranth,fill=amaranth, fill opacity=0.5, line cap=round, semithick, shift={(0,2)}, rotate=180] (-0.5,0) to[out=-90,in=-90,looseness=0.4] (0.5,0) to[out=90, in=90, looseness=0.4] (-0.5,0);
		\node at (0,1) {\scriptsize $a$};	
\end{tikzpicture}}

\newcommand{\cutupdown}[2]{	\begin{tikzpicture}[baseline={([yshift=-.5ex]current bounding box.center)}, scale=0.6]
		\draw[amaranth,fill =amaranth,fill opacity=0.3, line cap=round, semithick] (-0.5,0) to[out=90, in=180] (0,0.75) to[out=0,in=90] (0.5,0) to[out=90, in=90, looseness=0.4] (-0.5,0);
		
		\draw[amaranth,fill=amaranth, fill opacity=0.5, line cap=round, semithick] (-0.5,0) to[out=-90,in=-90,looseness=0.4] (0.5,0) to[out=90, in=90, looseness=0.4] (-0.5,0);
		
		\draw[amaranth,fill =amaranth,fill opacity=0.3, line cap=round, semithick, shift={(0,2)},rotate=180] (-0.5,0) to[out=90, in=180] (0,0.75) to[out=0,in=90] (0.5,0) to[out=90, in=90, looseness=0.4] (-0.5,0);
		\draw[amaranth,fill=amaranth, fill opacity=0.5, line cap=round, semithick, shift={(0,2)}, rotate=180] (-0.5,0) to[out=-90,in=-90,looseness=0.4] (0.5,0) to[out=90, in=90, looseness=0.4] (-0.5,0);
		\node at (0,0.35) {\scriptsize $#1$};	
		\node at (0,1.55) {\scriptsize $#2$};
\end{tikzpicture}}

\newcommand{\asheet}[1]{\begin{tikzpicture}[baseline={([yshift=-.7ex]current bounding box.center)}, scale=0.6]
		\draw[amaranth, fill=amaranth, fill opacity=0.3,  line cap=round, semithick]
		(0,0) rectangle (1.8,1.2);
		\node at (0.9,0.6) {\scriptsize $#1$};
\end{tikzpicture}}

\newcommand{\distk}{\begin{tikzpicture}[baseline={([yshift=-.6ex]current bounding box.center)}, scale=1.6]
		\draw[thick] (-0.3,0) ellipse (0.1 and 0.03);
		\draw[thick] (-0.3,0.4) ellipse (0.1 and 0.03);
		\draw[thick, line cap=round] (-0.2,0) -- (-0.2,0.4);
		\draw[thick, line cap=round] (-0.4,0) -- (-0.4,0.4);
		
		\draw[thick, line cap=round] (-0.4,0) to[out=-90,in=-90,looseness=1.5] (0.4,0);
		\draw[thick, line cap=round] (-0.2,0) to[out=-90,in=-90,looseness=1.5] (0.2,0);
		
		\draw[thick, teal] (0.3,0) ellipse (0.1 and 0.03);
		\draw[thick, teal] (0.3,0.4) ellipse (0.1 and 0.03);
		\draw[thick, line cap=round, teal, dashed] (0.2,0) -- (0.2,0.4);
		\draw[thick, line cap=round, teal, dashed] (0.4,0) -- (0.4,0.4);
		
		\node at (0,-0.24) {\tiny $H^{-1}$};
\end{tikzpicture}}

\newcommand{\distksq}{\begin{tikzpicture}[baseline={([yshift=-.6ex]current bounding box.center)}, scale=1.6]
		\draw[thick] (-0.5,0.7) ellipse (0.1 and 0.03);
		\draw[thick, line cap = round] (-0.6,0.4) to[out=90,in=90, looseness=1.5] (-0.4,0.4);
		\draw[thick, line cap = round] (-0.8,0.4) to[out=90, in=-90] (-0.6,0.7);  
		\draw[thick, line cap = round] (-0.2,0.4) to[out=90, in=-90] (-0.4,0.7);  
		
		\draw[thick] (0.5,0.7) ellipse (0.1 and 0.03);
		\draw[thick, line cap = round] (0.6,0.4) to[out=90,in=90, looseness=1.5] (0.4,0.4);
		\draw[thick, line cap = round] (0.8,0.4) to[out=90, in=-90] (0.6,0.7);  
		\draw[thick, line cap = round] (0.2,0.4) to[out=90, in=-90] (0.4,0.7);  	
		
		\draw[thick] (-0.3,0) ellipse (0.1 and 0.03);
		\draw[thick] (-0.3,0.4) ellipse (0.1 and 0.03);
		\draw[thick, line cap=round] (-0.2,0) -- (-0.2,0.4);
		\draw[thick, line cap=round] (-0.4,0) -- (-0.4,0.4);
		
		\draw[thick, line cap=round] (-0.4,0) to[out=-90,in=-90,looseness=1.5] (0.4,0);
		\draw[thick, line cap=round] (-0.2,0) to[out=-90,in=-90,looseness=1.5] (0.2,0);
		
		\node at (0,-0.24) {\tiny $H^{-1}$};
		\draw[thick, teal] (0.3,0) ellipse (0.1 and 0.03);
		\draw[thick, teal] (0.3,0.4) ellipse (0.1 and 0.03);
		\draw[thick, line cap=round, teal, dashed] (0.2,0) -- (0.2,0.4);
		\draw[thick, line cap=round, teal, dashed] (0.4,0) -- (0.4,0.4);
		
		\draw[thick] (-0.7,0) ellipse (0.1 and 0.03);
		\draw[thick] (-0.7,0.4) ellipse (0.1 and 0.03);
		\draw[thick, line cap=round] (-0.6,0) -- (-0.6,0.4);
		\draw[thick, line cap=round] (-0.8,0) -- (-0.8,0.4);
		
		\draw[thick, line cap=round] (-0.8,0) to[out=-90,in=-90,looseness=1.5] (0.8,0);
		\draw[thick, line cap=round] (-0.6,0) to[out=-90,in=-90,looseness=1.5] (0.6,0);
		
		\draw[thick, teal] (0.7,0) ellipse (0.1 and 0.03);
		\draw[thick, teal] (0.7,0.4) ellipse (0.1 and 0.03);
		\draw[thick, line cap=round, teal, dashed] (0.6,0) -- (0.6,0.4);
		\draw[thick, line cap=round, teal, dashed] (0.8,0) -- (0.8,0.4);
		
		\node at (0,-0.61) {\tiny $H^{-1}$};
\end{tikzpicture}}

\newcommand{\distksqmod}{\begin{tikzpicture}[baseline={([yshift=-.6ex]current bounding box.center)}, scale=1.6]
		\draw[thick] (-0.5,0.3) ellipse (0.1 and 0.03);
		\draw[thick, line cap = round] (-0.6,0) to[out=90,in=90, looseness=1.5] (-0.4,0);
		\draw[thick, line cap = round] (-0.8,0) to[out=90, in=-90] (-0.6,0.3);  
		\draw[thick, line cap = round] (-0.2,0) to[out=90, in=-90] (-0.4,0.3);  
		
		\draw[thick, line cap = round] (0.6,0) to[out=90,in=90, looseness=1.5] (0.4,0);
		\draw[thick, line cap = round] (0.8,0) to[out=90, in=-90] (0.6,0.3);  
		\draw[thick, line cap = round] (0.2,0) to[out=90, in=-90] (0.4,0.3);  
		
		\draw[thick] (-0.5,0.7) ellipse (0.1 and 0.03);
		\draw[thick, line cap=round] (-0.4,0.3) -- (-0.4,0.7);
		\draw[thick, line cap=round] (-0.6,0.3) -- (-0.6,0.7);		
		
		\draw[thick] (-0.3,0) ellipse (0.1 and 0.03);
		\draw[thick] (-0.7,0) ellipse (0.1 and 0.03);
		
		\draw[thick, line cap=round] (-0.4,0) to[out=-90,in=-90,looseness=1.5] (0.4,0);
		\draw[thick, line cap=round] (-0.2,0) to[out=-90,in=-90,looseness=1.5] (0.2,0);
		
		\node at (0,-0.24) {\tiny $H^{-1}$};
		
		\draw[thick] (0.7,0) ellipse (0.1 and 0.03);
		\draw[thick] (0.3,0) ellipse (0.1 and 0.03);
		
		\draw[thick, line cap=round] (-0.8,0) to[out=-90,in=-90,looseness=1.5] (0.8,0);
		\draw[thick, line cap=round] (-0.6,0) to[out=-90,in=-90,looseness=1.5] (0.6,0);
		
		\draw[thick, teal] (0.5,0.3) ellipse (0.1 and 0.03);
		\draw[thick, teal] (0.5,0.7) ellipse (0.1 and 0.03);
		\draw[thick, line cap=round, teal, dashed] (0.4,0.3) -- (0.4,0.7);
		\draw[thick, line cap=round, teal, dashed] (0.6,0.3) -- (0.6,0.7);	
		
		\node at (0,-0.61) {\tiny $H^{-1}$};
\end{tikzpicture}}

\newcommand{\cross}{\begin{tikzpicture}[baseline={([yshift=-.7ex]current bounding box.center)}, scale=1]
		\draw[very thick,line cap=round] (0.2,0) -- (-0.2,0.6);
		\draw[very thick,line cap=round] (-0.2,0) -- (0.2,0.6);
\end{tikzpicture}}

\newcommand{\crossSimpleTransp}{\begin{tikzpicture}[baseline={([yshift=.1ex]current bounding box.center)}, scale=1]
        \draw[very thick,line cap=round] (0.6,0) -- (0.6,0.6);
        \draw[very thick,line cap=round] (1.2,0) -- (1.2,0.6);
        \draw[very thick,line cap=round] (-0.6,0) -- (-0.6,0.6);
        \draw[very thick,line cap=round] (-1.2,0) -- (-1.2,0.6);
        \draw[very thick, line cap = round, dash pattern=on 0pt off 3.5\pgflinewidth] (0.75,0.3) -- (1.05,0.3);
        \draw[very thick, line cap = round, dash pattern=on 0pt off 3.5\pgflinewidth] (-0.75,0.3) -- (-1.05,0.3);
        
		\draw[very thick,line cap=round] (0.2,0) -- (-0.2,0.6);
		\draw[very thick,line cap=round] (-0.2,0) -- (0.2,0.6);

        \node[above] at (-1.2,-0.4) {\tiny $1$};
        \node[below] at (-0.3,-0.02) {\tiny $i$};
        \node[below] at (0.3,-0.02) {\tiny $i+1$};
        \node[above] at (1.2,-0.4) {\tiny $m$};
\end{tikzpicture}}

\newcommand{\urdcross}{\begin{tikzpicture}[baseline={([yshift=-.7ex]current bounding box.center)}, scale=1]
		\draw[fill] (0.1,0.45) circle (0.06);
		\draw[very thick,line cap=round] (0.2,0) -- (-0.2,0.6);
		\draw[very thick,line cap=round] (-0.2,0) -- (0.2,0.6);
\end{tikzpicture}}

\newcommand{\uldcross}{\begin{tikzpicture}[baseline={([yshift=-.7ex]current bounding box.center)}, scale=1]
		\draw[fill] (-0.1,0.45) circle (0.06);
		\draw[very thick,line cap=round] (0.2,0) -- (-0.2,0.6);
		\draw[very thick,line cap=round] (-0.2,0) -- (0.2,0.6);
\end{tikzpicture}}

\newcommand{\drdcross}{\begin{tikzpicture}[baseline={([yshift=-.7ex]current bounding box.center)}, scale=1]
		\draw[fill] (0.1,0.15) circle (0.06);
		\draw[very thick,line cap=round] (0.2,0) -- (-0.2,0.6);
		\draw[very thick,line cap=round] (-0.2,0) -- (0.2,0.6);
\end{tikzpicture}}

\newcommand{\dldcross}{\begin{tikzpicture}[baseline={([yshift=-.7ex]current bounding box.center)}, scale=1]
		\draw[fill] (-0.1,0.15) circle (0.06);
		\draw[very thick,line cap=round] (0.2,0) -- (-0.2,0.6);
		\draw[very thick,line cap=round] (-0.2,0) -- (0.2,0.6);
\end{tikzpicture}}

\newcommand{\cupcap}{\begin{tikzpicture}[baseline={([yshift=-.7ex]current bounding box.center)}, scale=1]
		\draw[fill,white] (0.2,0.3) circle (0.06);
		\draw[very thick,line cap=round] (-0.2,0.6) to[out=-90,in=180] (0,0.4) to[out=0, in=-90] (0.2,0.6);
		\draw[very thick,line cap=round] (-0.2,0) to[out=90,in=180] (0,0.2) to[out=0, in=90] (0.2,0);
\end{tikzpicture}}

\newcommand{\udid}{\begin{tikzpicture}[baseline={([yshift=-.7ex]current bounding box.center)}, scale=1]
		\draw[fill, white] (0.2,0.3) circle (0.06);
		\draw[very thick,line cap=round] (0.2,0) -- (0.2,0.6);
		\draw[very thick,line cap=round] (-0.2,0) -- (-0.2,0.6);
\end{tikzpicture}}

\newcommand{\udcup}{\begin{tikzpicture}[baseline={([yshift=-.7ex]current bounding box.center)}, scale=0.6]
		\draw[fill, white] (0,0) circle (0.06); 
		\draw[thick,line cap=round] (-0.2,0.2) to[out=-90,in=180] (0,0) to[out=0, in=-90] (0.2,0.2);
\end{tikzpicture}}

\newcommand{\dcup}{\begin{tikzpicture}[baseline={([yshift=-.7ex]current bounding box.center)}, scale=0.6]
		\draw[fill] (0,0) circle (0.06);
		\draw[thick,line cap=round] (-0.2,0.2) to[out=-90,in=180] (0,0) to[out=0, in=-90] (0.2,0.2);
\end{tikzpicture}}

\newcommand{\nesteddcups}{\begin{tikzpicture}[baseline={([yshift=-.7ex]current bounding box.center)}, scale=0.6]
		\draw[fill] (0,0) circle (0.06);
		\draw[thick,line cap=round] (-0.2,0.2) to[out=-90,in=180] (0,0) to[out=0, in=-90] (0.2,0.2);
		
		\draw[fill] (0,-0.2) circle (0.06);
		\draw[thick,line cap=round] (-0.45,0.2) to[out=-90,in=180] (0,-0.2) to[out=0, in=-90] (0.45,0.2);
		
\end{tikzpicture}}

\newcommand{\nesteddcupudcup}{\begin{tikzpicture}[baseline={([yshift=-.7ex]current bounding box.center)}, scale=0.6]

		\draw[thick,line cap=round] (-0.2,0.2) to[out=-90,in=180] (0,0) to[out=0, in=-90] (0.2,0.2);
		
		\draw[fill] (0,-0.2) circle (0.06);
		\draw[thick,line cap=round] (-0.45,0.2) to[out=-90,in=180] (0,-0.2) to[out=0, in=-90] (0.45,0.2);
		
\end{tikzpicture}}

\newcommand{\nestedudcupdcup}{\begin{tikzpicture}[baseline={([yshift=-.7ex]current bounding box.center)}, scale=0.6]
		\draw[fill] (0,0) circle (0.06);
		\draw[thick,line cap=round] (-0.2,0.2) to[out=-90,in=180] (0,0) to[out=0, in=-90] (0.2,0.2);

		\draw[thick,line cap=round] (-0.45,0.2) to[out=-90,in=180] (0,-0.2) to[out=0, in=-90] (0.45,0.2);
		
\end{tikzpicture}}

\newcommand{\nestedudcups}{\begin{tikzpicture}[baseline={([yshift=-.7ex]current bounding box.center)}, scale=0.6]
		\draw[thick,line cap=round] (-0.2,0.2) to[out=-90,in=180] (0,0) to[out=0, in=-90] (0.2,0.2);

		\draw[thick,line cap=round] (-0.45,0.2) to[out=-90,in=180] (0,-0.2) to[out=0, in=-90] (0.45,0.2);
		
\end{tikzpicture}}

\newcommand{\dTLCirclea}{\begin{tikzpicture}[baseline={([yshift=-.6ex]current bounding box.center)}, scale=0.9]
		
		\draw[thin,line cap=round, dashed] (0,0) circle (0.6);
		\draw[very thick, line cap=round] (0,0) circle (0.35);
		\draw[very thick, fill=white] (0,-0.35) circle (0.19);
		\node at (0,-0.35) {$a$}; 
\end{tikzpicture}}

\newcommand{\dTLemptyCircle}{\begin{tikzpicture}[baseline={([yshift=-.6ex]current bounding box.center)}, scale=0.9]
		\draw[thin,line cap=round, dashed] (0,0) circle (0.6);
\end{tikzpicture}}

\newcommand{\LocaldTLLeftRight}[2]{\begin{tikzpicture}[baseline={([yshift=-.6ex]current bounding box.center)}, scale=0.9]
		
		\draw[thin,line cap=round, dashed] (0,0) circle (0.6);
		\draw[very thick, line cap=round] (0.4243,0.4243) to[out = 180+45, in=180-45, looseness=1] (0.4243,-0.4243);	
		\draw[very thick, line cap=round] (-0.4243,0.4243) to[out = -45, in=+45, looseness=1] (-0.4243,-0.4243);	
		\draw[very thick, fill=white] (0.25,0) circle (0.2);
		\node at (0.25,0) {\scriptsize$#2$}; 
		\draw[very thick, fill=white] (-0.25,0) circle (0.2);
		\node at (-0.25,0) {\scriptsize$#1$}; 
		
\end{tikzpicture}}

\newcommand{\LocaldTLUpDown}[2]{\begin{tikzpicture}[baseline={([yshift=-.6ex]current bounding box.center)}, scale=0.9]
		
		\draw[thin,line cap=round, dashed] (0,0) circle (0.6);
		\draw[very thick, line cap=round] (0.4243,0.4243) to[out = 180+45, in=-45, looseness=1] (-0.4243,0.4243);	
		\draw[very thick, line cap=round] (0.4243,-0.4243) to[out = 180-45, in=+45, looseness=1] (-0.4243,-0.4243);		
		\draw[very thick, fill=white] (0,0.25) circle (0.2);
		\node at (0,0.25) {\scriptsize$#1$}; 
		\draw[very thick, fill=white] (0,-0.25) circle (0.2);
		\node at (0,-0.25) {\scriptsize$#2$}; 
\end{tikzpicture}}

\newcommand{\dTLCircle}{\begin{tikzpicture}[baseline={([yshift=-.6ex]current bounding box.center)}, scale=0.9]
		
		\draw[thin,line cap=round, dashed] (0,0) circle (0.6);
		\draw[very thick, line cap=round] (0,0) circle (0.3);
\end{tikzpicture}}

\newcommand{\dTLCircleDot}{\begin{tikzpicture}[baseline={([yshift=-.6ex]current bounding box.center)}, scale=0.9]
		
		\draw[thin,line cap=round, dashed] (0,0) circle (0.6);
		\draw[very thick, line cap=round] (0,0) circle (0.3);
		\draw[color=black, fill]  (0,-0.3) circle (0.06);
\end{tikzpicture}}

\newcommand{\LocaldTLudud}{\begin{tikzpicture}[baseline={([yshift=-.6ex]current bounding box.center)}, scale=0.9]
		
		\draw[thin,line cap=round, dashed] (0,0) circle (0.6);
		\draw[very thick, line cap=round] (0.4243,0.4243) to[out = 180+45, in=180-45, looseness=1] (0.4243,-0.4243);	
		\draw[very thick, line cap=round] (-0.4243,0.4243) to[out = -45, in=+45, looseness=1] (-0.4243,-0.4243);	
		
\end{tikzpicture}}

\newcommand{\LocaldTLdd}{\begin{tikzpicture}[baseline={([yshift=-.6ex]current bounding box.center)}, scale=0.9]
		
		\draw[thin,line cap=round, dashed] (0,0) circle (0.6);
		\draw[very thick, line cap=round] (0.4243,0.4243) to[out = 180+45, in=180-45, looseness=1] (0.4243,-0.4243);	
		\draw[very thick, line cap=round] (-0.4243,0.4243) to[out = -45, in=+45, looseness=1] (-0.4243,-0.4243);	
		\draw[color=black, fill]  (-0.25,0) circle (0.06);		
		\draw[color=black, fill]  (0.25,0) circle (0.06);	
		
\end{tikzpicture}}

\newcommand{\LocaldTLudd}{\begin{tikzpicture}[baseline={([yshift=-.6ex]current bounding box.center)}, scale=0.9]
		
		\draw[thin,line cap=round, dashed] (0,0) circle (0.6);
		\draw[very thick, line cap=round] (0.4243,0.4243) to[out = 180+45, in=180-45, looseness=1] (0.4243,-0.4243);	
		\draw[very thick, line cap=round] (-0.4243,0.4243) to[out = -45, in=+45, looseness=1] (-0.4243,-0.4243);	
		\draw[color=black, fill]  (0.25,0) circle (0.06);		
		
\end{tikzpicture}}

\newcommand{\LocaldTLdud}{\begin{tikzpicture}[baseline={([yshift=-.6ex]current bounding box.center)}, scale=0.9]
		
		\draw[thin,line cap=round, dashed] (0,0) circle (0.6);
		\draw[very thick, line cap=round] (0.4243,0.4243) to[out = 180+45, in=180-45, looseness=1] (0.4243,-0.4243);	
		\draw[very thick, line cap=round] (-0.4243,0.4243) to[out = -45, in=+45, looseness=1] (-0.4243,-0.4243);	
		\draw[color=black, fill]  (-0.25,0) circle (0.06);		
		
\end{tikzpicture}}

\newcommand{\LocaldTLudupperdlower}{\begin{tikzpicture}[baseline={([yshift=-.6ex]current bounding box.center)}, scale=0.9]
		
		\draw[thin,line cap=round, dashed] (0,0) circle (0.6);
		\draw[very thick, line cap=round] (0.4243,0.4243) to[out = 180+45, in=-45, looseness=1] (-0.4243,0.4243);	
		\draw[very thick, line cap=round] (0.4243,-0.4243) to[out = 180-45, in=+45, looseness=1] (-0.4243,-0.4243);	
		\draw[color=black, fill]  (0,-0.25) circle (0.06);		
		
\end{tikzpicture}}

\newcommand{\LocaldTLdupperudlower}{\begin{tikzpicture}[baseline={([yshift=-.6ex]current bounding box.center)}, scale=0.9]
		
		\draw[thin,line cap=round, dashed] (0,0) circle (0.6);
		\draw[very thick, line cap=round] (0.4243,0.4243) to[out = 180+45, in=-45, looseness=1] (-0.4243,0.4243);	
		\draw[very thick, line cap=round] (0.4243,-0.4243) to[out = 180-45, in=+45, looseness=1] (-0.4243,-0.4243);	
		\draw[color=black, fill]  (0,0.25) circle (0.06);		
		
\end{tikzpicture}}

\newcommand{\LocaldTLdupperdlower}{\begin{tikzpicture}[baseline={([yshift=-.6ex]current bounding box.center)}, scale=0.9]
		
		\draw[thin,line cap=round, dashed] (0,0) circle (0.6);
		\draw[very thick, line cap=round] (0.4243,0.4243) to[out = 180+45, in=-45, looseness=1] (-0.4243,0.4243);	
		\draw[very thick, line cap=round] (0.4243,-0.4243) to[out = 180-45, in=+45, looseness=1] (-0.4243,-0.4243);	
		\draw[color=black, fill]  (0,0.25) circle (0.06);		
		\draw[color=black, fill]  (0,-0.25) circle (0.06);				
\end{tikzpicture}}

\newcommand{\LocaldTLududKL}{\begin{tikzpicture}[baseline={([yshift=1ex]current bounding box.center)}, scale=0.9]
		
		\draw[thin,line cap=round, dashed] (0,0) circle (0.6);
		\draw[very thick, line cap=round] (0.4243,0.4243) to[out = 180+45, in=180-45, looseness=1] (0.4243,-0.4243);	
		\draw[very thick, line cap=round] (-0.4243,0.4243) to[out = -45, in=+45, looseness=1] (-0.4243,-0.4243);	
		
		\node at (-0.4,-0.9) {$k$};
		\node at (0.4,-0.9) {$l$};
\end{tikzpicture}}

\newcommand{\LocaldTLddKL}{\begin{tikzpicture}[baseline={([yshift=1ex]current bounding box.center)}, scale=0.9]
		
		\draw[thin,line cap=round, dashed] (0,0) circle (0.6);
		\draw[very thick, line cap=round] (0.4243,0.4243) to[out = 180+45, in=180-45, looseness=1] (0.4243,-0.4243);	
		\draw[very thick, line cap=round] (-0.4243,0.4243) to[out = -45, in=+45, looseness=1] (-0.4243,-0.4243);	
		\draw[color=black, fill]  (0.25,0) circle (0.06);		
		\draw[color=black, fill]  (-0.25,0) circle (0.06);	
		
		\node at (-0.4,-0.9) {$k$};
		\node at (0.4,-0.9) {$l$};
		
\end{tikzpicture}}

\newcommand{\LocaldTLDoubleddKL}{\begin{tikzpicture}[baseline={([yshift=1ex]current bounding box.center)}, scale=0.9]
		
		\draw[thin,line cap=round, dashed] (0,0) circle (0.6);
		\draw[very thick, line cap=round] (0.4243,0.4243) to[out = 180+45, in=180-45, looseness=1] (0.4243,-0.4243);	
		\draw[very thick, line cap=round] (-0.4243,0.4243) to[out = -45, in=+45, looseness=1] (-0.4243,-0.4243);	
		\draw[color=black, fill]  (0.27,0.15) circle (0.06);		
		\draw[color=black, fill]  (0.27,-0.15) circle (0.06);
		\draw[color=black, fill]  (-0.27,0.15) circle (0.06);		
		\draw[color=black, fill]  (-0.27,-0.15) circle (0.06);	
		
		\node at (-0.4,-0.9) {$k$};
		\node at (0.4,-0.9) {$l$};
\end{tikzpicture}}

\newcommand{\frobmult}{\begin{tikzpicture}[baseline={([yshift=-.6ex]current bounding box.center)}, scale=1.6]
		\draw[thick, line cap = round] (0.2,0) ellipse (0.1 and 0.03);
		\draw[thick, line cap = round] (-0.2,0) ellipse (0.1 and 0.03);
		\draw[thick, line cap = round] (0,0.3) ellipse (0.1 and 0.03);
		\draw[thick, line cap = round] (0.1,0) to[out=90,in=90, looseness=1.5] (-0.1,0);
		\draw[thick, line cap = round] (0.3,0) to[out=90, in=-90] (0.1,0.3);  
		\draw[thick, line cap = round] (-0.3,0) to[out=90, in=-90] (-0.1,0.3);  
\end{tikzpicture}}

\newcommand{\frobcomult}{\begin{tikzpicture}[baseline={([yshift=-.6ex]current bounding box.center)}, scale=1.6]
		\draw[thick, line cap = round] (0.2,0) ellipse (0.1 and 0.03);
		\draw[thick, line cap = round] (-0.2,0) ellipse (0.1 and 0.03);
		\draw[thick, line cap = round] (0,-0.3) ellipse (0.1 and 0.03);
		\draw[thick, line cap = round] (0.1,0) to[out=-90,in=-90, looseness=1.5] (-0.1,0);
		\draw[thick, line cap = round] (0.3,0) to[out=-90, in=90] (0.1,-0.3);  
		\draw[thick, line cap = round] (-0.3,0) to[out=-90, in=90] (-0.1,-0.3);  
\end{tikzpicture}}

\newcommand{\frobunit}{\begin{tikzpicture}[baseline={([yshift=-.6ex]current bounding box.center)}, scale=1.6]
		\draw[thick, line cap = round] (0,0) ellipse (0.1 and 0.03);
		\draw[thick, line cap = round] (-0.1,0) to[out=-90, in=-90, looseness=2.5] (0.1,0);
\end{tikzpicture}}

\newcommand{\frobcounit}{\begin{tikzpicture}[baseline={([yshift=-.6ex]current bounding box.center)}, scale=1.6]
		\draw[thick, line cap = round] (0,0) ellipse (0.1 and 0.03);
		\draw[thick, line cap = round] (-0.1,0) to[out=90, in=90, looseness=2.5] (0.1,0);
\end{tikzpicture}}

\newcommand{\longwalk}{\begin{tikzpicture}[baseline={([yshift=-.6ex]current bounding box.center)}, xscale=0.8,yscale=0.6]
		\node[left] at (0,-2) {\scriptsize -2};
		\node[left] at (0,-1) {\scriptsize -1};
		\node[left] at (0,0) {\scriptsize 0};
		\node[left] at (0,1) {\scriptsize 1};
		\node[left] at (0,2) {\scriptsize 2};
		\node[left] at (0,3) {\scriptsize 3};
		
		\draw[thin,red] (0,-2.5) -- (0,-2);
		\draw[thin] (0,-2) -- (0,-1);
		\draw[thin, red] (0,-1) -- (0,0); 
		\draw[thin] (0,0) -- (0,1); 
		\draw[thin, red] (0,1) -- (0,2);
		\draw[thin] (0,2) -- (0,3);
		\draw[thin, red] (0,3) -- (0,3.5);
		
		\draw[ultra thin, dashed]  (0,-2) -- (6.2,-2);
		\draw[ultra thin, dashed]  (0,-1) -- (6.2,-1);
		\draw[ultra thin, dashed]  (0,0) -- (6.2,0);
		\draw[ultra thin, dashed]  (0,1) -- (6.2,1);
		\draw[ultra thin, dashed]  (0,2) -- (6.2,2);
		\draw[ultra thin, dashed]  (0,3) -- (6.2,3);
		
		\draw[thick, line cap = round ] (0,0) -- (0.5,1);
		\draw[thick, line cap = round ] (0.5,1) -- (1,0);
		\draw[thick, red, line cap = round ] (1,0) -- (1.5,-1);
		\draw[thick, red, line cap = round ] (1.5,-1) -- (2,0);
		\draw[thick, line cap = round ] (2,0) -- (2.5,1);
		\draw[thick, red, line cap = round ] (2.5,1) -- (3,2);
		\draw[thick, red, line cap = round ] (3,2) -- (3.5,1);
		\draw[thick, red, line cap = round ] (3.5,1) -- (4,2);
		\draw[thick, red, line cap = round ] (4,2) -- (4.5,1);
		\draw[thick, line cap = round ] (4.5,1) -- (5,0);
		\draw[thick, line cap = round ] (5,0) -- (5.5,1);
		\draw[thick, line cap = round ] (5.5,1) -- (6,0);
		
		\draw[fill] (0,0) circle (0.05);
		\draw[fill] (0.5,1) circle (0.05);
		\draw[fill] (1,0) circle (0.05);
		\draw[fill] (1.5,-1) circle (0.05);
		\draw[fill] (2,0) circle (0.05);
		\draw[fill] (2.5,1) circle (0.05);
		\draw[fill] (3,2) circle (0.05);
		\draw[fill] (3.5,1) circle (0.05);
		\draw[fill] (4,2) circle (0.05);
		\draw[fill] (4.5,1) circle (0.05);
		\draw[fill] (5,0) circle (0.05);
		\draw[fill] (5.5,1) circle (0.05);
		\draw[fill] (6,0) circle (0.05);
		
\end{tikzpicture}}

\newcommand{\idempForLongwalk}{\begin{tikzpicture}[baseline={([yshift=-.6ex]current bounding box.center)}, yscale=0.9,xscale=1] 
		\draw[thick,line cap=round] (0,0.2) to[out=-90,in=180] (0.2,0) to[out=0, in=-90] (0.4,0.2);
		\draw[thick,line cap=round, red] (0.8,0.2) to[out=-90,in=180] (1,0) to[out=0, in=-90] (1.2,0.2);
		\draw[thick,line cap=round] (1.6,0.2) to[out=-90,in=180] (2.6,-0.6) to[out=0, in=-90] (3.6,0.2);
		\draw[thick,line cap=round, red] (2,0.2) to[out=-90,in=180] (2.6,-0.3) to[out=0, in=-90] (3.2,0.2);
		\draw[thick,line cap=round, red] (2.4,0.2) to[out=-90,in=180] (2.6,0) to[out=0, in=-90] (2.8,0.2);
		\draw[thick,line cap=round] (4,0.2) to[out=-90,in=180] (4.2,0) to[out=0, in=-90] (4.4,0.2);
		
		\node[above] at (0,0.3) {\scriptsize 1};
		\node[above] at (0.4,0.3) {\scriptsize 2};
		\node[above] at (0.8,0.3) {\scriptsize 3};
		\node[above] at (1.2,0.3) {\scriptsize 4};
		\node[above] at (1.6,0.3) {\scriptsize 5};
		\node[above] at (2,0.3) {\scriptsize 6};
		\node[above] at (2.4,0.3) {\scriptsize 7};
		\node[above] at (2.8,0.3) {\scriptsize 8};
		\node[above] at (3.2,0.3) {\scriptsize 9};
		\node[above] at (3.6,0.3) {\scriptsize 10};
		\node[above] at (4,0.3) {\scriptsize 11};
		\node[above] at (4.4,0.3) {\scriptsize 12};
		
\end{tikzpicture}}

\newcommand{\Symmet}[1]{\begin{tikzpicture}[baseline={([yshift=-.6ex]current bounding box.center)}, scale=0.9]
		\draw[very thick] (-0.8,-0.25) rectangle (0.8,0.25);
		\node at (0,0) {\scriptsize ${#1}$};
		\draw[very thick,line cap=round] (-0.65,0.25) -- (-0.65,0.45);
		\draw[very thick,line cap=round] (-0.4,0.25) -- (-0.4,0.45);
		\draw[very thick, line cap = round, dash pattern=on 0pt off 4\pgflinewidth] (-0.2,0.38) -- (0.2,0.38);
		\draw[very thick,line cap=round] (0.4,0.25) -- (0.4,0.45);
		\draw[very thick,line cap=round] (0.65,0.25) -- (0.65,0.45);
		
		\draw[very thick,line cap=round] (-0.65,-0.25) -- (-0.65,-0.45);
		\draw[very thick,line cap=round] (-0.4,-0.25) -- (-0.4,-0.45);
		\draw[very thick, line cap = round, dash pattern=on 0pt off 4\pgflinewidth] (-0.2,-0.38) -- (0.2,-0.38);
		\draw[very thick,line cap=round] (0.4,-0.25) -- (0.4,-0.45);
		\draw[very thick,line cap=round] (0.65,-0.25) -- (0.65,-0.45);
		
\end{tikzpicture}}

\newcommand{\SymStack}[2]{\begin{tikzpicture}[baseline={([yshift=-.5ex]current bounding box.center)}, scale=0.9]
		\draw[very thick] (-0.8,-0.25) rectangle (0.8,0.25);
		\draw[very thick] (-0.8,0.75) rectangle (0.8,1.25);
		\node at (0,0) {\scriptsize ${#1}$};
		\node at (0,1) {\scriptsize ${#2}$};
		
		\draw[very thick,line cap=round] (-0.65,0.25) -- (-0.65,0.75);
		\draw[very thick,line cap=round] (-0.4,0.25) -- (-0.4,0.75);
		\draw[very thick, line cap = round, dash pattern=on 0pt off 4\pgflinewidth] (-0.1,0.5) -- (0.3,0.5);
		\draw[very thick,line cap=round] (0.65,0.25) -- (0.65,0.75);
		
		\draw[very thick,line cap=round] (-0.65,-0.25) -- (-0.65,-0.45);
		\draw[very thick,line cap=round] (-0.4,-0.25) -- (-0.4,-0.45);
		\draw[very thick, line cap = round, dash pattern=on 0pt off 4\pgflinewidth] (-0.1,-0.38) -- (0.3,-0.38);
		\draw[very thick,line cap=round] (0.65,-0.25) -- (0.65,-0.45);
		
		\draw[very thick,line cap=round] (-0.65,1.25) -- (-0.65,1.45);
		\draw[very thick,line cap=round] (-0.4,1.25) -- (-0.4,1.45);
		\draw[very thick, line cap = round, dash pattern=on 0pt off 4\pgflinewidth] (-0.1,1.38) -- (0.3,1.38);
		\draw[very thick,line cap=round] (0.65,1.25) -- (0.65,1.45);
		
		\draw[very thick,line cap=round] (0.98,-0.45) -- (0.98,1.45);
		
\end{tikzpicture}}

\newcommand{\SymStackCross}[2]{\begin{tikzpicture}[baseline={([yshift=-.5ex]current bounding box.center)}, scale=0.9]
		\draw[very thick] (-0.8,-0.25) rectangle (0.8,0.25);
		\draw[very thick] (-0.8,0.75) rectangle (0.8,1.25);
		\node at (0,0) {\scriptsize ${#1}$};
		\node at (0,1) {\scriptsize ${#2}$};
		
		\draw[very thick,line cap=round] (-0.65,0.25) -- (-0.65,0.75);
		\draw[very thick,line cap=round] (-0.4,0.25) -- (-0.4,0.75);
		\draw[very thick, line cap = round, dash pattern=on 0pt off 4\pgflinewidth] (-0.1,0.5) -- (0.3,0.5);

		\draw[very thick,line cap=round] (-0.65,-0.25) -- (-0.65,-0.45);
		\draw[very thick,line cap=round] (-0.4,-0.25) -- (-0.4,-0.45);
		\draw[very thick, line cap = round, dash pattern=on 0pt off 4\pgflinewidth] (-0.1,-0.38) -- (0.3,-0.38);
		\draw[very thick,line cap=round] (0.65,-0.25) -- (0.65,-0.45);
		
		\draw[very thick,line cap=round] (-0.65,1.25) -- (-0.65,1.45);
		\draw[very thick,line cap=round] (-0.4,1.25) -- (-0.4,1.45);
		\draw[very thick, line cap = round, dash pattern=on 0pt off 4\pgflinewidth] (-0.1,1.38) -- (0.3,1.38);		
		\draw[very thick,line cap=round] (0.65,1.25) -- (0.65,1.45);
		
		\draw[very thick,line cap=round,line join = round] (0.65,0.25) -- (0.98,0.75) -- (0.98,1.45);
		\draw[very thick,line cap=round,line join = round] (0.65,0.75) -- (0.98,0.25) -- (0.98,-0.45);
\end{tikzpicture}}

\newcommand{\SymOne}{\begin{tikzpicture}[baseline={([yshift=-.6ex]current bounding box.center)}, scale=0.9]
		\draw[very thick] (-0.175,-0.25) rectangle (0.175,0.25);
		\node at (0,0) {\scriptsize $1$};
		\draw[very thick,line cap=round] (0,0.25) -- (0,0.45);
		\draw[very thick,line cap=round] (0,-0.25) -- (0,-0.45);
		
\end{tikzpicture}}

\newcommand{\SymLine}{\begin{tikzpicture}[baseline={([yshift=-.6ex]current bounding box.center)}, scale=0.9]
		\draw[very thick,line cap=round] (0,-0.45) -- (0,0.45);
\end{tikzpicture}}

\newcommand{\SymmetKirby}[1]{\begin{tikzpicture}[baseline={([yshift=-.6ex]current bounding box.center)}, scale=0.9]
		\draw[very thick] (-0.8,-0.25) rectangle (0.8,0.25);
		\node at (0,0) {\tiny ${#1}$};
		\draw[very thick,line cap=round] (-0.65,0.25) -- (-0.65,0.45);
		\draw[very thick,line cap=round] (-0.4,0.25) -- (-0.4,0.45);
		\draw[very thick, line cap = round, dash pattern=on 0pt off 4\pgflinewidth] (-0.2,0.38) -- (0.2,0.38);
		\draw[very thick,line cap=round] (0.4,0.25) -- (0.4,0.45);
		\draw[very thick,line cap=round] (0.65,0.25) -- (0.65,0.45);
		
		\draw[very thick,line cap=round] (-0.65,-0.25) to[in=-90, out=-90,looseness= 2] (-0.4,-0.25);
		\draw[very thick, line cap = round, dash pattern=on 0pt off 4\pgflinewidth] (-0.2,-0.38) -- (0.2,-0.38);
		\draw[very thick,line cap=round] (0.65,-0.25) to[in=-90, out=-90,looseness= 2] (0.4,-0.25);
		
		\draw[color=black, fill =white]  (0.525,-0.4) circle (0.06);
		\draw[color=black, fill =white]  (-0.525,-0.4) circle (0.06);	
\end{tikzpicture}}

\newcommand{\SymmetDottedCups}[1]{\begin{tikzpicture}[baseline={([yshift=-.6ex]current bounding box.center)}, scale=0.9]
		\draw[very thick] (-0.8,-0.25) rectangle (0.8,0.25);
		\node at (0,0) {\scriptsize ${#1}$};
		\draw[very thick,line cap=round] (-0.65,0.25) -- (-0.65,0.45);
		\draw[very thick,line cap=round] (-0.4,0.25) -- (-0.4,0.45);
		\draw[very thick, line cap = round, dash pattern=on 0pt off 4\pgflinewidth] (-0.2,0.38) -- (0.2,0.38);
		\draw[very thick,line cap=round] (0.4,0.25) -- (0.4,0.45);
		\draw[very thick,line cap=round] (0.65,0.25) -- (0.65,0.45);
		
		\draw[very thick,line cap=round] (-0.65,-0.25) to[in=-90, out=-90,looseness= 2] (-0.4,-0.25);
		\draw[very thick, line cap = round, dash pattern=on 0pt off 4\pgflinewidth] (-0.2,-0.38) -- (0.2,-0.38);
		\draw[very thick,line cap=round] (0.65,-0.25) to[in=-90, out=-90,looseness= 2] (0.4,-0.25);
		
		\draw[color=black, fill]  (0.525,-0.4) circle (0.05);
		\draw[color=black, fill]  (-0.525,-0.4) circle (0.05);
\end{tikzpicture}}

\newcommand{\SymmetDottedCupsCapped}[1]{\begin{tikzpicture}[baseline={([yshift=-.6ex]current bounding box.center)}, scale=0.9]
		\draw[very thick] (-0.75,-0.25) rectangle (1.25,0.25);
		\node at (0.25,0) {\scriptsize ${#1}$};
		\draw[very thick,line cap=round] (-0.65,0.25) -- (-0.65,0.45);
		\draw[very thick,line cap=round] (-0.4,0.25) -- (-0.4,0.45);
		\draw[very thick, line cap = round, dash pattern=on 0pt off 4\pgflinewidth] (-0.2,0.38) -- (0.2,0.38);
		\draw[very thick, line cap = round] (0.4, 0.25) -- (0.4,0.45);
		\draw[very thick, line cap = round] (0.65, 0.25) -- (0.65,0.45);
		\draw[very thick,line cap=round] (0.9,0.25) to[in=90,out=90,looseness=2] (1.15,0.25);
		
		\draw[very thick,line cap=round] (-0.65,-0.25) to[in=-90, out=-90,looseness= 2] (-0.4,-0.25);
		\draw[very thick, line cap = round, dash pattern=on 0pt off 4\pgflinewidth] (-0.2,-0.38) -- (0.2,-0.38);
		\draw[very thick,line cap=round] (0.65,-0.25) to[in=-90, out=-90,looseness= 2] (0.4,-0.25);
		\draw[very thick,line cap=round] (0.9,-0.25) to[in=-90, out=-90,looseness= 2] (1.15,-0.25);
		
		\draw[color=black, fill]  (0.525,-0.4) circle (0.05);
		\draw[color=black, fill]  (-0.525,-0.4) circle (0.05);
		\draw[color=black, fill]  (1.025,-0.4) circle (0.05);
		
		\draw[color=black, fill]  (1.025,0.4) circle (0.05);
\end{tikzpicture}}

\newcommand{\SymKirbyTwo}[1]{\begin{tikzpicture}[baseline={([yshift=-.6ex]current bounding box.center)}, scale=0.9]
		\draw[very thick] (-0.25,-0.25) rectangle (0.25,0.25);
		\node at (0,0) {\scriptsize ${#1}$};
		\draw[very thick,line cap=round] (-0.125,0.25) -- (-0.125,0.45);
		\draw[very thick,line cap=round] (0.125,0.25) -- (0.125,0.45);
		
		\draw[very thick,line cap=round] (-0.125,-0.25) to[in=-90, out=-90,looseness= 2] (0.125,-0.25);
		\draw[color=black, fill =white]  (0,-0.4) circle (0.06);
\end{tikzpicture}}

\newcommand{\SepidemCup}{\begin{tikzpicture}[baseline={([yshift=-.6ex]current bounding box.center)}, scale=1.2]
		\draw[very thick,line cap=round] (-0.125,0) to[in=-90, out=-90,looseness= 2] (0.125,0);
		\draw[color=black, fill =white]  (0,-0.15) circle (0.05);
\end{tikzpicture}}

\newcommand{\dcupbig}{\begin{tikzpicture}[baseline={([yshift=-.6ex]current bounding box.center)}, scale=1.2]
		\draw[very thick,line cap=round] (-0.125,0) to[out=-90,in=-90, looseness=2] (0.125,0);
		\draw[fill] (0,-0.15) circle (0.05);
\end{tikzpicture}}

\newcommand{\SymmetDoubleId}[1]{\begin{tikzpicture}[baseline={([yshift=-.6ex]current bounding box.center)}, scale=0.9]
		\draw[very thick] (-0.75,-0.25) rectangle (0.75,0.25);
		\node at (0,0) {\scriptsize ${#1}$};
		\draw[very thick,line cap=round] (-0.65,0.25) -- (-0.65,0.45);
		\draw[very thick,line cap=round] (-0.4,0.25) -- (-0.4,0.45);
		\draw[very thick, line cap = round, dash pattern=on 0pt off 4\pgflinewidth] (-0.2,0.38) -- (0.2,0.38);
		\draw[very thick,line cap=round] (0.4,0.25) -- (0.4,0.45);
		\draw[very thick,line cap=round] (0.65,0.25) -- (0.65,0.45);
		
		\draw[very thick,line cap=round] (-0.65,-0.25) -- (-0.65,-0.45);
		\draw[very thick,line cap=round] (-0.4,-0.25) -- (-0.4,-0.45);
		\draw[very thick, line cap = round, dash pattern=on 0pt off 4\pgflinewidth] (-0.2,-0.38) -- (0.2,-0.38);
		\draw[very thick,line cap=round] (0.4,-0.25) -- (0.4,-0.45);
		\draw[very thick,line cap=round] (0.65,-0.25) -- (0.65,-0.45);
		
		\draw[very thick, line cap=round] (0.9,-0.45) -- (0.9,0.45);
		\draw[very thick, line cap=round] (1.1, -0.45) -- (1.1,0.45);
\end{tikzpicture}}

\newcommand{\SymStackCrossId}[2]{\begin{tikzpicture}[baseline={([yshift=-.5ex]current bounding box.center)}, scale=0.9]
		\draw[very thick] (-0.75,-0.25) rectangle (0.75,0.25);
		\draw[very thick] (-0.75,0.75) rectangle (0.75,1.25);
		\node at (0,0) {\scriptsize ${#1}$};
		\node at (0,1) {\scriptsize ${#2}$};
		
		\draw[very thick,line cap=round] (-0.65,0.25) -- (-0.65,0.75);
		\draw[very thick,line cap=round] (-0.4,0.25) -- (-0.4,0.75);
		\draw[very thick, line cap = round, dash pattern=on 0pt off 4\pgflinewidth] (-0.1,0.5) -- (0.3,0.5);

		\draw[very thick,line cap=round] (-0.65,-0.25) -- (-0.65,-0.45);
		\draw[very thick,line cap=round] (-0.4,-0.25) -- (-0.4,-0.45);
		\draw[very thick, line cap = round, dash pattern=on 0pt off 4\pgflinewidth] (-0.1,-0.38) -- (0.3,-0.38);
		\draw[very thick,line cap=round] (0.65,-0.25) -- (0.65,-0.45);
		
		\draw[very thick,line cap=round] (-0.65,1.25) -- (-0.65,1.45);
		\draw[very thick,line cap=round] (-0.4,1.25) -- (-0.4,1.45);
		\draw[very thick, line cap = round, dash pattern=on 0pt off 4\pgflinewidth] (-0.1,1.38) -- (0.3,1.38);		
		\draw[very thick,line cap=round] (0.65,1.25) -- (0.65,1.45);
		
		\draw[very thick,line cap=round,line join = round] (0.65,0.25) -- (0.9,0.75) -- (0.9,1.45);
		\draw[very thick,line cap=round,line join = round] (0.65,0.75) -- (0.9,0.25) -- (0.9,-0.45);
		
		\draw[very thick,line cap=round,line join = round] (1.15, -0.45) -- (1.15,1.45);
\end{tikzpicture}}

\newcommand{\SymStackIdCross}[2]{\begin{tikzpicture}[baseline={([yshift=-.5ex]current bounding box.center)}, scale=0.9]
		\draw[very thick] (-0.75,-0.25) rectangle (0.75,0.25);
		\draw[very thick] (-0.75,0.75) rectangle (0.75,1.25);
		\node at (0,0) {\scriptsize ${#1}$};
		\node at (0,1) {\scriptsize ${#2}$};
		
		\draw[very thick,line cap=round] (-0.65,0.25) -- (-0.65,0.75);
		\draw[very thick,line cap=round] (-0.4,0.25) -- (-0.4,0.75);
		\draw[very thick, line cap = round, dash pattern=on 0pt off 4\pgflinewidth] (-0.1,0.5) -- (0.3,0.5);

		\draw[very thick,line cap=round] (-0.65,-0.25) -- (-0.65,-0.45);
		\draw[very thick,line cap=round] (-0.4,-0.25) -- (-0.4,-0.45);
		\draw[very thick, line cap = round, dash pattern=on 0pt off 4\pgflinewidth] (-0.1,-0.38) -- (0.3,-0.38);
		\draw[very thick,line cap=round] (0.65,-0.25) -- (0.65,-0.45);
		
		\draw[very thick,line cap=round] (-0.65,1.25) -- (-0.65,1.45);
		\draw[very thick,line cap=round] (-0.4,1.25) -- (-0.4,1.45);
		\draw[very thick, line cap = round, dash pattern=on 0pt off 4\pgflinewidth] (-0.1,1.38) -- (0.3,1.38);		
		\draw[very thick,line cap=round] (0.65,1.25) -- (0.65,1.45);
		
		\draw[very thick,line cap=round,line join = round] (0.65,0.25) -- (0.65,0.75);
		\draw[very thick,line cap=round,line join = round]  (0.9,-0.45) -- (0.9,0.25) -- (1.15,0.75) -- (1.15,1.45);
		
		\draw[very thick,line cap=round,line join = round] (1.15, -0.45) -- (1.15,0.25) -- (0.9,0.75) -- (0.9,1.45);
\end{tikzpicture}}

\newcommand{\SymStacksTwoCrossDown}[3]{\begin{tikzpicture}[baseline={([yshift=-.5ex]current bounding box.center)}, scale=0.9]
		\draw[very thick] (-0.75,-0.25) rectangle (0.75,0.25);
		\draw[very thick] (-0.75,0.75) rectangle (0.75,1.25);
		\draw[very thick] (-0.75,1.75) rectangle (0.75,2.25);
		\node at (0,0) {\scriptsize ${#1}$};
		\node at (0,1) {\scriptsize ${#2}$};
		\node at (0,2) {\scriptsize ${#3}$};

		\draw[very thick,line cap=round] (-0.65,0.25) -- (-0.65,0.75);
		\draw[very thick,line cap=round] (-0.4,0.25) -- (-0.4,0.75);
		\draw[very thick, line cap = round, dash pattern=on 0pt off 4\pgflinewidth] (-0.1,0.5) -- (0.3,0.5);

		\draw[very thick,line cap=round] (-0.65,-0.25) -- (-0.65,-0.45);
		\draw[very thick,line cap=round] (-0.4,-0.25) -- (-0.4,-0.45);
		\draw[very thick, line cap = round, dash pattern=on 0pt off 4\pgflinewidth] (-0.1,-0.38) -- (0.3,-0.38);
		\draw[very thick,line cap=round] (0.65,-0.25) -- (0.65,-0.45);
		
		\draw[very thick,line cap=round] (-0.65,1.25) -- (-0.65,1.75);
		\draw[very thick,line cap=round] (-0.4,1.25) -- (-0.4,1.75);
		\draw[very thick, line cap = round, dash pattern=on 0pt off 4\pgflinewidth] (-0.1,1.5) -- (0.3,1.5);		
		\draw[very thick,line cap=round] (0.65,1.25) -- (0.65,1.75);
		
		\draw[very thick,line cap=round] (-0.65,2.25) -- (-0.65,2.45);
		\draw[very thick,line cap=round] (-0.4,2.25) -- (-0.4,2.45);
		\draw[very thick, line cap = round, dash pattern=on 0pt off 4\pgflinewidth] (-0.1,2.38) -- (0.3,2.38);		
		\draw[very thick,line cap=round] (0.65,2.25) -- (0.65,2.45);

		\draw[very thick,line cap=round,line join = round]  (0.9,-0.45) -- (0.9,0.25) -- (0.65,0.75);
		\draw[very thick,line cap=round, line join = round]  
		(0.65,0.25) -- (0.9,0.75) -- (0.9,1.25) -- (1.15,1.75) -- (1.15,2.45);
		
		\draw[very thick,line cap=round,line join = round] (1.15, -0.45) -- (1.15,1.25) -- (0.9,1.75) -- (0.9,2.45);
\end{tikzpicture}}

\newcommand{\SymStacksTwoCrossUp}[3]{\begin{tikzpicture}[baseline={([yshift=-.5ex]current bounding box.center)}, scale=0.9]
		\draw[very thick] (-0.75,-0.25) rectangle (0.75,0.25);
		\draw[very thick] (-0.75,0.75) rectangle (0.75,1.25);
		\draw[very thick] (-0.75,1.75) rectangle (0.75,2.25);
		\node at (0,0) {\scriptsize ${#1}$};
		\node at (0,1) {\scriptsize ${#2}$};
		\node at (0,2) {\scriptsize ${#3}$};

		\draw[very thick,line cap=round] (-0.65,0.25) -- (-0.65,0.75);
		\draw[very thick,line cap=round] (-0.4,0.25) -- (-0.4,0.75);
		\draw[very thick, line cap = round, dash pattern=on 0pt off 4\pgflinewidth] (-0.1,0.5) -- (0.3,0.5);

		\draw[very thick,line cap=round] (-0.65,-0.25) -- (-0.65,-0.45);
		\draw[very thick,line cap=round] (-0.4,-0.25) -- (-0.4,-0.45);
		\draw[very thick, line cap = round, dash pattern=on 0pt off 4\pgflinewidth] (-0.1,-0.38) -- (0.3,-0.38);
		\draw[very thick,line cap=round] (0.65,-0.25) -- (0.65,-0.45);
		
		\draw[very thick,line cap=round] (-0.65,1.25) -- (-0.65,1.75);
		\draw[very thick,line cap=round] (-0.4,1.25) -- (-0.4,1.75);
		\draw[very thick, line cap = round, dash pattern=on 0pt off 4\pgflinewidth] (-0.1,1.5) -- (0.3,1.5);		
		
		\draw[very thick,line cap=round] (-0.65,2.25) -- (-0.65,2.45);
		\draw[very thick,line cap=round] (-0.4,2.25) -- (-0.4,2.45);
		\draw[very thick, line cap = round, dash pattern=on 0pt off 4\pgflinewidth] (-0.1,2.38) -- (0.3,2.38);		
		\draw[very thick,line cap=round] (0.65,2.25) -- (0.65,2.45);
		
		\draw[very thick,line cap=round,line join = round] (0.65,0.25) -- (0.65,0.75);
		
		\draw[very thick,line cap=round,line join = round]  (0.9,-0.45) -- (0.9,0.25) -- (1.15,0.75) -- (1.15,2.45);
		
		\draw[very thick,line cap=round,line join = round] (1.15, -0.45) -- (1.15,0.25) -- (0.9,0.75) -- (0.9,1.25) -- (0.65,1.75);
		
		\draw[very thick,line cap=round,line join = round] (0.65,1.25) -- (0.9,1.75) -- (0.9,2.45);
\end{tikzpicture}}

\newcommand{\SymFourStackCrosses}[4]{\begin{tikzpicture}[baseline={([yshift=-.5ex]current bounding box.center)}, scale=0.9]
		\draw[very thick] (-0.75,-0.25) rectangle (0.75,0.25);
		\draw[very thick] (-0.75,0.75) rectangle (0.75,1.25);
		\draw[very thick] (-0.75,1.75) rectangle (0.75,2.25);
		\draw[very thick] (-0.75,2.75) rectangle (0.75,3.25);
		\node at (0,0) {\scriptsize ${#1}$};
		\node at (0,1) {\scriptsize ${#2}$};
		\node at (0,2) {\scriptsize ${#3}$};
		\node at (0,3) {\scriptsize ${#4}$};

		\draw[very thick,line cap=round] (-0.65,0.25) -- (-0.65,0.75);
		\draw[very thick,line cap=round] (-0.4,0.25) -- (-0.4,0.75);
		\draw[very thick, line cap = round, dash pattern=on 0pt off 4\pgflinewidth] (-0.1,0.5) -- (0.3,0.5);

		\draw[very thick,line cap=round] (-0.65,-0.25) -- (-0.65,-0.45);
		\draw[very thick,line cap=round] (-0.4,-0.25) -- (-0.4,-0.45);
		\draw[very thick, line cap = round, dash pattern=on 0pt off 4\pgflinewidth] (-0.1,-0.38) -- (0.3,-0.38);
		\draw[very thick,line cap=round] (0.65,-0.25) -- (0.65,-0.45);
		
		\draw[very thick,line cap=round] (-0.65,1.25) -- (-0.65,1.75);
		\draw[very thick,line cap=round] (-0.4,1.25) -- (-0.4,1.75);
		\draw[very thick, line cap = round, dash pattern=on 0pt off 4\pgflinewidth] (-0.1,1.5) -- (0.3,1.5);		
		\draw[very thick,line cap=round] (0.65,1.25) -- (0.65,1.75);
		
		\draw[very thick,line cap=round] (-0.65,2.25) -- (-0.65,2.75);
		\draw[very thick,line cap=round] (-0.4,2.25) -- (-0.4,2.75);
		\draw[very thick, line cap = round, dash pattern=on 0pt off 4\pgflinewidth] (-0.1,2.5) -- (0.3,2.5);

		\draw[very thick,line cap=round] (-0.65,3.25) -- (-0.65,3.45);
		\draw[very thick,line cap=round] (-0.4,3.25) -- (-0.4,3.45);
		\draw[very thick, line cap = round, dash pattern=on 0pt off 4\pgflinewidth] (-0.1,3.38) -- (0.3,3.38);		
		\draw[very thick,line cap=round] (0.65,3.25) -- (0.65,3.45);

		\draw[very thick,line cap=round,line join = round]  (0.9,-0.45) -- (0.9,0.25) -- (0.65,0.75);
		
		\draw[very thick,line cap=round, line join = round]  
		(0.65,0.25) -- (0.9,0.75) -- (0.9,1.25) -- (1.15,1.75) -- (1.15,3.45);
		
		\draw[very thick,line cap=round,line join = round] (1.15, -0.45) -- (1.15,1.25) -- (0.9,1.75) -- (0.9,2.25) -- (0.65,2.75);
		
		\draw[very thick,line cap=round] (0.65,2.25) -- (0.9,2.75) -- (0.9,3.45);
\end{tikzpicture}}

\newcommand{\SymTwoCappedCupped}[1]{\begin{tikzpicture}[baseline={([yshift=-.6ex]current bounding box.center)}, scale=0.9]
		\draw[very thick] (-0.225,-0.25) rectangle (0.225,0.25);
		\node at (0,0) {\scriptsize ${#1}$};
		\draw[very thick,line cap=round] (-0.125,-0.25) to[in=180, out=-90] (0,-0.4) to[in=-90, out=0] (0.125,-0.25);
		\draw[very thick,line cap=round] (-0.125,0.25) to[in=180, out=90] (0,0.4) to[in=90, out=0] (0.125,0.25);
		\draw[color=black, fill]  (0,-0.4) circle (0.05);
		\draw[color=black, fill]  (0,0.4) circle (0.05);
\end{tikzpicture}}

\newcommand{\DoublyDottedCircle}{\begin{tikzpicture}[baseline={([yshift=-.6ex]current bounding box.center)}, scale=0.9]
		\draw[very thick] (0,0) circle (0.15);
		\draw[color=black, fill]  (0,-0.15) circle (0.05);
		\draw[color=black, fill]  (0,0.15) circle (0.05);
\end{tikzpicture}}

\newcommand{\UndottedCircle}{\begin{tikzpicture}[baseline={([yshift=-.6ex]current bounding box.center)}, scale=0.9]
		\draw[very thick] (0,0) circle (0.15);
\end{tikzpicture}}

\newcommand{\CrossedDottedCircle}{\begin{tikzpicture}[baseline={([yshift=-.7ex]current bounding box.center)}, scale=0.9]
		\draw[fill] (0,-0.15) circle (0.05);
		\draw[fill] (0,0.4) circle (0.06);
		\draw[very thick,line cap=round] (-0.125,0) to[out=-90,in=180] (0,-0.15) to[out=0, in=-90] (0.125,0);
		\draw[very thick,line cap=round] (-0.125,0) -- (0.125,0.25);
		\draw[very thick,line cap=round] (0.125,0) -- (-0.125,0.25);
		\draw[very thick,line cap=round] (-0.125,0.25) to[out=90,in=180] (0,0.4) to[out=0, in=90] (0.125,0.25);
\end{tikzpicture}}

\newcommand{\RightTwist}{\begin{tikzpicture}[baseline={([yshift=-.7ex]current bounding box.center)}, scale=0.9]
		
		\draw[very thick,line cap=round] (-0.125,0) to[out=-90,in=180] (0,-0.15) to[out=0, in=-90] (0.125,0);
		\draw[very thick,line cap=round] (-0.375,0) -- (-0.125,0.25);
		\draw[very thick,line cap=round] (-0.125,0) -- (-0.375,0.25);
		\draw[very thick, line cap = round] (0.125, 0) -- (0.125,0.25);
		\draw[very thick, line cap = round] (-0.375,0.25) -- (-0.375,0.5);
		\draw[very thick, line cap = round] (-0.375,0) -- (-0.375,-0.25); 
		\draw[very thick,line cap=round] (-0.125,0.25) to[out=90,in=180] (0,0.4) to[out=0, in=90] (0.125,0.25);
\end{tikzpicture}}

\newcommand{\OneLine}{\begin{tikzpicture}[baseline={([yshift=-.7ex]current bounding box.center)}, scale=0.9]

		\draw[very thick, line cap = round] (0,-0.25) -- (0,0.5);
\end{tikzpicture}}

\newcommand{\CrossedDottedCup}{\begin{tikzpicture}[baseline={([yshift=-.7ex]current bounding box.center)}, scale=0.9]
		\draw[fill] (0,-0.15) circle (0.05);
		\draw[very thick,line cap=round] (-0.125,0) to[out=-90,in=180] (0,-0.15) to[out=0, in=-90] (0.125,0);
		\draw[very thick,line cap=round] (-0.125,0) -- (0.125,0.25);
		\draw[very thick,line cap=round] (0.125,0) -- (-0.125,0.25);
\end{tikzpicture}}

\newcommand{\CrossedDottedCap}{\begin{tikzpicture}[baseline={([yshift=-.7ex]current bounding box.center)}, scale=0.9]
		\draw[fill] (0,0.4) circle (0.05);
		\draw[very thick,line cap=round] (-0.125,0) -- (0.125,0.25);
		\draw[very thick,line cap=round] (0.125,0) -- (-0.125,0.25);
		\draw[very thick,line cap=round] (-0.125,0.25) to[out=90,in=180] (0,0.4) to[out=0, in=90] (0.125,0.25);
\end{tikzpicture}}

\newcommand{\DottedCup}{\begin{tikzpicture}[baseline={([yshift=-.7ex]current bounding box.center)}, scale=0.9]
		\draw[fill] (0,-0.15) circle (0.05);
		\draw[very thick,line cap=round] (-0.125,0) to[out=-90,in=180] (0,-0.15) to[out=0, in=-90] (0.125,0);
\end{tikzpicture}}

\newcommand{\DottedCap}{\begin{tikzpicture}[baseline={([yshift=-.7ex]current bounding box.center)}, scale=0.9]
		\draw[fill] (0,0.4) circle (0.05);
		\draw[very thick,line cap=round] (-0.125,0.25) to[out=90,in=180] (0,0.4) to[out=0, in=90] (0.125,0.25);
\end{tikzpicture}}

\newcommand{\SymmetTripleCrossAndCapped}[1]{\begin{tikzpicture}[baseline={([yshift=-.6ex]current bounding box.center)}, scale=0.9]
		\draw[very thick] (-0.75,-0.25) rectangle (0.75,0.25);
		\node at (0,0) {\scriptsize ${#1}$};
		\draw[very thick,line cap=round] (-0.65,0.25) -- (-0.65,0.45);
		\draw[very thick,line cap=round] (-0.4,0.25) -- (-0.4,0.45);
		\draw[very thick, line cap = round, dash pattern=on 0pt off 4\pgflinewidth] (-0.2,0.38) -- (0.2,0.38);
		\draw[very thick,line cap=round] (0.4,0.25) -- (0.4,0.45);
		
		\draw[very thick,line cap=round] (-0.65,-0.25) -- (-0.65,-0.45);
		\draw[very thick,line cap=round] (-0.4,-0.25) -- (-0.4,-0.45);
		\draw[very thick, line cap = round, dash pattern=on 0pt off 4\pgflinewidth] (-0.2,-0.38) -- (0.2,-0.38);
		\draw[very thick,line cap=round] (0.4,-0.25) -- (0.4,-0.45);
		
		\draw[very thick,line cap=round] (0.65,-0.25) -- (0.9,-0.75);
		\draw[very thick,line cap=round] (0.9,-0.75) to[out=-90,in=180] (1.025,-0.9) to[out=0,in=-90] (1.15,-0.75);
		\draw[fill] (1.025,-0.9) circle (0.05);
		\draw[very thick,line cap=round] (1.15,-0.75) -- (1.15,-0.25);
		\draw[very thick,line cap=round] (0.9,-0.25) -- (0.65,-0.75);
		
		\draw[very thick,line cap=round] (1.15,-0.25) -- (0.9,0.25);
		\draw[very thick,line cap=round] (0.9,-0.25) --(1.15,0.25);
		
		\draw[very thick,line cap=round] (1.15,0.75) -- (1.15,0.25);
		
		\draw[very thick,line cap=round] (0.9,0.25) -- (0.65,0.75);
		\draw[very thick,line cap=round] (0.65,0.25) -- (0.9,0.75);		
		
		\draw[very thick,line cap=round] (0.9,0.75) to[out=90,in=180] (1.025,0.9) to[out=0,in=90] (1.15,0.75);
		\draw[fill] (1.025,0.9) circle (0.05);
\end{tikzpicture}}

\newcommand{\SymmetNextToTripleCrossAndCapped}[1]{\begin{tikzpicture}[baseline={([yshift=-.6ex]current bounding box.center)}, scale=0.9]
		\draw[very thick] (-0.75,-0.25) rectangle (0.75,0.25);
		\node at (0,0) {\scriptsize ${#1}$};
		\draw[very thick,line cap=round] (-0.65,0.25) -- (-0.65,0.45);
		\draw[very thick,line cap=round] (-0.4,0.25) -- (-0.4,0.45);
		\draw[very thick, line cap = round, dash pattern=on 0pt off 4\pgflinewidth] (-0.2,0.38) -- (0.2,0.38);
		\draw[very thick,line cap=round] (0.4,0.25) -- (0.4,0.45);
		\draw[very thick,line cap=round] (0.65,0.25) -- (0.65,0.45);
		
		\draw[very thick,line cap=round] (-0.65,-0.25) -- (-0.65,-0.45);
		\draw[very thick,line cap=round] (-0.4,-0.25) -- (-0.4,-0.45);
		\draw[very thick, line cap = round, dash pattern=on 0pt off 4\pgflinewidth] (-0.2,-0.38) -- (0.2,-0.38);
		\draw[very thick,line cap=round] (0.4,-0.25) -- (0.4,-0.45);
		\draw[very thick,line cap=round] (0.65,-0.25) -- (0.65,-0.45);		
		
		\draw[very thick,line cap=round] (0.9,-0.25) -- (1.15,-0.75);
		\draw[very thick,line cap=round] (1.15,-0.75) to[out=-90,in=180] (1.275,-0.9) to[out=0,in=-90] (1.4,-0.75);
		\draw[fill] (1.275,-0.9) circle (0.05);
		\draw[very thick,line cap=round] (1.4,-0.75) -- (1.4,-0.25);
		\draw[very thick,line cap=round] (1.15,-0.25) -- (0.9,-0.75);
		
		\draw[very thick,line cap=round] (1.4,-0.25) -- (1.15,0.25);
		\draw[very thick,line cap=round] (1.15,-0.25) --(1.4,0.25);
		
		\draw[very thick,line cap=round] (1.4,0.75) -- (1.4,0.25);
		
		\draw[very thick,line cap=round] (1.15,0.25) -- (0.9,0.75);
		\draw[very thick,line cap=round] (0.9,0.25) -- (1.15,0.75);		
		
		\draw[very thick,line cap=round] (1.15,0.75) to[out=90,in=180] (1.275,0.9) to[out=0,in=90] (1.4,0.75);
		\draw[fill] (1.275,0.9) circle (0.05);
		
		\draw[very thick,line cap=round] (0.9, -0.25) -- (0.9,0.25);
\end{tikzpicture}}

\newcommand{\SymmetDoubleQuadrupleCrossAndCapped}[2]{\begin{tikzpicture}[baseline={([yshift=-.6ex]current bounding box.center)}, scale=0.9]
		\draw[very thick] (-0.75,0.25) rectangle (0.75,0.75);
		\node at (0,0.5) {\scriptsize ${#1}$};
		\draw[very thick] (-0.75,-0.25) rectangle (0.75,-0.75);
		\node at (0,-0.5) {\scriptsize ${#2}$};
		
		\draw[very thick,line cap=round] (-0.65,0.75) -- (-0.65,0.95);
		\draw[very thick,line cap=round] (-0.4,0.75) -- (-0.4,0.95);
		\draw[very thick, line cap = round, dash pattern=on 0pt off 4\pgflinewidth] (-0.2,0.88) -- (0.2,0.88);
		\draw[very thick,line cap=round] (0.4,0.75) -- (0.4,0.95);
		\draw[very thick,line cap=round] (0.65,0.75) -- (0.65,0.95);
		
		\draw[very thick,line cap=round] (-0.65,-0.75) -- (-0.65,-0.95);
		\draw[very thick,line cap=round] (-0.4,-0.75) -- (-0.4,-0.95);
		\draw[very thick, line cap = round, dash pattern=on 0pt off 4\pgflinewidth] (-0.2,-0.88) -- (0.2,-0.88);
		\draw[very thick,line cap=round] (0.4,-0.75) -- (0.4,-0.95);
		\draw[very thick,line cap=round] (0.65,-0.75) -- (0.65,-0.95);		

		\draw[very thick,line cap=round] (-0.65,-0.25) -- (-0.65,0.25);
		\draw[very thick,line cap=round] (-0.4,-0.25) -- (-0.4,0.25);
		\draw[very thick, line cap = round, dash pattern=on 0pt off 4\pgflinewidth] (-0.2,0) -- (0.2,0);
		\draw[very thick,line cap=round] (0.4,-0.25) -- (0.4,0.25);

		\draw[very thick,line cap=round] (0.9,-0.25) -- (1.15,-0.75);
		\draw[very thick,line cap=round] (1.15,-0.75) to[out=-90,in=180] (1.275,-0.9) to[out=0,in=-90] (1.4,-0.75);
		\draw[fill] (1.275,-0.9) circle (0.05);
		\draw[very thick,line cap=round] (1.4,-0.75) -- (1.4,-0.25);
		\draw[very thick,line cap=round] (1.15,-0.25) -- (0.9,-0.75);
		
		\draw[very thick,line cap=round] (1.4,-0.25) -- (1.15,0.25);
		\draw[very thick,line cap=round] (1.15,-0.25) --(1.4,0.25);
		
		\draw[very thick,line cap=round] (1.4,0.75) -- (1.4,0.25);
		
		\draw[very thick,line cap=round] (1.15,0.25) -- (0.9,0.75);
		\draw[very thick,line cap=round] (0.9,0.25) -- (1.15,0.75);		
		
		\draw[very thick,line cap=round] (1.15,0.75) to[out=90,in=180] (1.275,0.9) to[out=0,in=90] (1.4,0.75);
		\draw[fill] (1.275,0.9) circle (0.05);
		
		\draw[very thick,line cap=round] (0.65, -0.25) -- (0.9,0.25);
		\draw[very thick,line cap=round] (0.9, -0.25) -- (0.65,0.25);
		
		\draw[very thick,line cap=round] (0.9, 0.75) -- (0.9,0.95);
		\draw[very thick,line cap=round] (0.9, -0.75) -- (0.9,-0.95);
\end{tikzpicture}}

\newcommand{\SymmetDoubleDCappedDCupped}[2]{\begin{tikzpicture}[baseline={([yshift=-.6ex]current bounding box.center)}, scale=0.9]
		\draw[very thick] (-0.75,0.75) rectangle (1.25,0.25);
		\node at (0.25,0.5) {\scriptsize ${#1}$};
		\draw[very thick] (-0.75,-0.25) rectangle (1.25,-0.75);
		\node at (0.25,-0.5) {\scriptsize ${#2}$};
		\draw[very thick,line cap=round] (-0.65,0.75) -- (-0.65,0.95);
		\draw[very thick,line cap=round] (-0.4,0.75) -- (-0.4,0.95);
		\draw[very thick, line cap = round, dash pattern=on 0pt off 4\pgflinewidth] (-0.2,0.88) -- (0.2,0.88);
		\draw[very thick, line cap = round] (0.4, 0.75) -- (0.4,0.95);
		\draw[very thick, line cap = round] (0.65, 0.75) -- (0.65,0.95);
		
		\draw[very thick,line cap=round] (-0.65,-0.75) to[in=-90, out=-90,looseness= 2] (-0.4,-0.75);
		\draw[very thick, line cap = round, dash pattern=on 0pt off 4\pgflinewidth] (-0.2,-0.88) -- (0.2,-0.88);
		\draw[very thick,line cap=round] (0.65,-0.75) to[in=-90, out=-90,looseness= 2] (0.4,-0.75);
		\draw[very thick,line cap=round] (0.9,-0.75) to[in=-90, out=-90,looseness= 2] (1.15,-0.75);
		
		\draw[very thick,line cap=round] (-0.65,-0.25) -- (-0.65,0.25);
		\draw[very thick,line cap=round] (-0.4,-0.25) -- (-0.4,0.25);
		\draw[very thick, line cap = round, dash pattern=on 0pt off 4\pgflinewidth] (-0.2,0) -- (0.2,0);
		\draw[very thick,line cap=round] (0.4,-0.25) -- (0.4,0.25);
		\draw[very thick,line cap=round] (0.65,-0.25) -- (0.65,0.25);
		
		\draw[very thick,line cap=round] (0.9,0.75) -- (0.9,0.95);
		\draw[very thick,line cap=round] (1.15,0.75) -- (1.15,0.95);
		
		\draw[color=black, fill]  (0.525,-0.9) circle (0.05);
		\draw[color=black, fill]  (-0.525,-0.9) circle (0.05);
		\draw[color=black, fill]  (1.025,-0.9) circle (0.05);
		
		\draw[very thick,line cap=round] (0.9,0.25) to[in=-90,out=-90,looseness=2] (1.15,0.25);
		\draw[very thick,line cap=round] (0.9,-0.25) to[in=90,out=90,looseness=2] (1.15,-0.25);		
		\draw[color=black, fill]  (1.025,0.105) circle (0.05);
		\draw[color=black, fill]  (1.025,-0.105) circle (0.05);
		
\end{tikzpicture}}

\definecolor{kirby}{RGB}{215,72,148}

\newcommand{\kirbycol}[1]{\begin{tikzpicture}[baseline={([yshift=-.9ex]current bounding box.center)}, scale=0.9]
		\draw[very thick,line cap=round] (-0.4,0.25) -- (-0.4,0.45);
		\draw[very thick,line cap=round] (0.4,0.25) -- (0.4,0.45);
		\draw[very thick, line cap = round, dash pattern=on 0pt off 4\pgflinewidth] (-0.2,0.38) -- (0.2,0.38);
		\draw[very thick, kirby] (-0.6,-0.25) rectangle (0.6,0.25);
		\node at (0,-0.05) {\scriptsize $#1$};
\end{tikzpicture}}

\newcommand{\kirbywide}[1]{\begin{tikzpicture}[baseline={([yshift=-.85ex]current bounding box.center)}, scale=0.9]
		\draw[very thick,line cap=round] (-0.65,0.25) -- (-0.65,0.45);
		\draw[very thick,line cap=round] (-0.4,0.25) -- (-0.4,0.45);
		\draw[very thick, line cap = round, dash pattern=on 0pt off 4\pgflinewidth] (-0.2,0.38) -- (0.2,0.38);
		\draw[very thick,line cap=round] (0.4,0.25) -- (0.4,0.45);
		\draw[very thick,line cap=round] (0.65,0.25) -- (0.65,0.45);		
		\draw[very thick, kirby] (-0.8,-0.25) rectangle (0.8,0.25);
		\node at (0,-0.05) {\scriptsize ${#1}$};
\end{tikzpicture}}

\newcommand{\kirbyTwo}[1]{\begin{tikzpicture}[baseline={([yshift=-.9ex]current bounding box.center)}, scale=0.9]
		\draw[very thick,line cap=round] (-0.125,0.25) -- (-0.125,0.45);
		\draw[very thick,line cap=round] (0.125,0.25) -- (0.125,0.45);
			
		\draw[very thick, kirby] (-0.3,-0.25) rectangle (0.3,0.25);
		\node at (0,0) {\scriptsize ${#1}$};
\end{tikzpicture}}

\newcommand{\kirbyCap}[1]{\begin{tikzpicture}[baseline={([yshift=-.9ex]current bounding box.center)}, scale=0.9]
		\draw[very thick,line cap=round] (-0.65,0.25) -- (-0.65,0.45);
		\draw[very thick,line cap=round] (-0.3,0.25) -- (-0.3,0.45);
		\draw[very thick,line cap=round] (0.3,0.25) -- (0.3,0.45);
		
		\draw[very thick,line cap=round] (0.65,0.25) -- (0.65,0.45);
		\draw[very thick, line cap = round, dash pattern=on 0pt off 1.8\pgflinewidth] (-0.56,0.38) -- (-0.37,0.38);
		\draw[very thick, line cap = round, dash pattern=on 0pt off 1.8\pgflinewidth] (0.39,0.38) -- (0.56,0.38);
		\draw[very thick,line cap=round] (-0.125,0.25) to[in=90, out=90,looseness= 2] (0.125,0.25);	
		\draw[very thick, kirby] (-0.8,-0.25) rectangle (0.8,0.25);
		\node at (0,-0.05) {\scriptsize $#1$};	
\end{tikzpicture}}

\newcommand{\kirbydCap}[1]{\begin{tikzpicture}[baseline={([yshift=-1.2ex]current bounding box.center)}, scale=0.9]
		\draw[very thick,line cap=round] (-0.65,0.25) -- (-0.65,0.45);
		\draw[very thick,line cap=round] (-0.3,0.25) -- (-0.3,0.45);
		\draw[very thick,line cap=round] (0.3,0.25) -- (0.3,0.45);
		
		\draw[very thick,line cap=round] (0.65,0.25) -- (0.65,0.45);
		\draw[very thick, line cap = round, dash pattern=on 0pt off 1.8\pgflinewidth] (-0.56,0.38) -- (-0.37,0.38);
		\draw[very thick, line cap = round, dash pattern=on 0pt off 1.8\pgflinewidth] (0.39,0.38) -- (0.56,0.38);
		\draw[very thick,line cap=round] (-0.125,0.25) to[in=90, out=90,looseness= 2] (0.125,0.25);
		\node at (0,0.38) {\scriptsize $\bullet$};
		\draw[very thick, kirby] (-0.8,-0.25) rectangle (0.8,0.25);
		\node at (0,-0.05) {\scriptsize $#1$};
\end{tikzpicture}}

\newcommand{\SymmetKirbyCappedRight}[1]{\begin{tikzpicture}[baseline={([yshift=-.6ex]current bounding box.center)}, scale=0.9]
		\draw[very thick,line cap=round] (-0.65,0.25) -- (-0.65,0.45);
		\draw[very thick,line cap=round] (-0.4,0.25) -- (-0.4,0.45);
		\draw[very thick, line cap = round, dash pattern=on 0pt off 4\pgflinewidth] (-0.2,0.38) -- (0.2,0.38);
		\draw[very thick, line cap = round] (0.4, 0.25) -- (0.4,0.45);
		\draw[very thick, line cap = round] (0.65, 0.25) -- (0.65,0.45);
		\draw[very thick,line cap=round] (0.9,0.25) to[in=90,out=90,looseness=2] (1.15,0.25);
		
		\draw[color=black, fill]  (1.025,0.4) circle (0.05);
		
		\draw[very thick,kirby] (-0.75,-0.25) rectangle (1.25,0.25);
		\node at (0.25,0) {\scriptsize ${#1}$};
\end{tikzpicture}}

\newcommand{\dcupid}{\begin{tikzpicture}[baseline={([yshift=-.7ex]current bounding box.center)}, scale=0.8]
		\draw[fill] (0,0) circle (0.075);
		\draw[very thick,line cap=round] (-0.2,0.2) to[out=-90,in=180] (0,0) to[out=0, in=-90] (0.2,0.2);
		\draw[very thick,line cap=round] (-0.2,0.2) -- (-0.2,0.8);
		\draw[very thick,line cap=round] (0.2,0.2) -- (0.2,0.8);
\end{tikzpicture}}
\newcommand{\dcupcross}{\begin{tikzpicture}[baseline={([yshift=-.7ex]current bounding box.center)}, scale=0.8]
		\draw[fill] (0,0) circle (0.075);
		\draw[very thick,line cap=round] (-0.2,0.2) to[out=-90,in=180] (0,0) to[out=0, in=-90] (0.2,0.2);
		\draw[very thick,line cap=round] (-0.2,0.2) -- (0.2,0.8);
		\draw[very thick,line cap=round] (0.2,0.2) -- (-0.2,0.8);
\end{tikzpicture}}

\newcommand{\dcupcapcup}{\begin{tikzpicture}[baseline={([yshift=-.7ex]current bounding box.center)}, scale=0.8]
		\draw[fill] (0,0) circle (0.075);
		\draw[very thick,line cap=round] (-0.2,0.2) to[out=-90,in=180] (0,0) to[out=0, in=-90] (0.2,0.2);
		\draw[very thick,line cap=round] (-0.2,0.2) to[out=90,in=180] (0,0.4) to[out=0, in=90] (0.2,0.2);
		\draw[very thick,line cap=round] (-0.2,0.8) to[out=-90,in=180] (0,0.6) to[out=0, in=-90] (0.2,0.8);
		
\end{tikzpicture}}

\newcommand{\DottedCapNaturality}{\begin{tikzpicture}[baseline={([yshift=-.6ex]current bounding box.center)}, scale=0.9]
		\draw[very thick,line cap=round] (0,0)--(0,0.5);
		
		\draw[very thick,line cap=round] (0,0.5) to[in=180,out=90] (0.125,0.65) to[out=0, in=90] (0.25,0.5);	
		\draw[color=black, fill]  (0.125,0.65) circle (0.05);
		
		\draw[very thick,line cap=round] (0.25,0.5) -- (0.5,0);
		\draw[very thick,line cap=round] (0.25,0) -- (0.5,0.5);
		
		\draw[very thick,line cap=round] (0,0) -- (0.25,-0.5);
		\draw[very thick,line cap=round] (0,-0.5) -- (0.25,0);
		
		\draw[very thick,line cap=round] (0.5,0) -- (0.5,-0.5);
		
\end{tikzpicture}}

\newcommand{\DottedCapNaturalityOne}{\begin{tikzpicture}[baseline={([yshift=-.6ex]current bounding box.center)}, scale=0.9]
		\draw[very thick,line cap=round] (0,0)--(0,0.5);
		
		\draw[very thick,line cap=round] (0,0.5) to[in=180,out=90] (0.125,0.65) to[out=0, in=90] (0.25,0.5);	
		\draw[color=black, fill]  (0.125,0.65) circle (0.05);
		
		\draw[very thick,line cap=round] (0.25,0.5) -- (0.25,0);
		\draw[very thick,line cap=round] (0.5,0) -- (0.5,0.5);
		
		\draw[very thick,line cap=round] (0,0) -- (0,-0.5);
		\draw[very thick,line cap=round] (0.25,-0.5) -- (0.25,0);
		
		\draw[very thick,line cap=round] (0.5,0) -- (0.5,-0.5);
		
\end{tikzpicture}}

\newcommand{\DottedCapNaturalityTwo}{\begin{tikzpicture}[baseline={([yshift=-.6ex]current bounding box.center)}, scale=0.9]
		\draw[very thick,line cap=round] (0,0)--(0,0.5);
		
		\draw[very thick,line cap=round] (0,0.5) to[in=180,out=90] (0.125,0.65) to[out=0, in=90] (0.25,0.5);	
		\draw[color=black, fill]  (0.125,0.65) circle (0.05);
		
		\draw[very thick,line cap=round] (0.25,0.5) -- (0.25,0);
		\draw[very thick,line cap=round] (0.5,0) -- (0.5,0.5);
		
		\draw[very thick,line cap=round] (0,-0.5) to[in=180,out=90] (0.125,-0.35) to[out=0, in=90] (0.25,-0.5);
		
		\draw[very thick,line cap=round] (0,0) to[in=180,out=-90] (0.125,-0.15) to[out=0, in=-90] (0.25,0);	
		
		\draw[very thick,line cap=round] (0.5,0) -- (0.5,-0.5);
\end{tikzpicture}}

\newcommand{\DottedCapNaturalityThree}{\begin{tikzpicture}[baseline={([yshift=-.6ex]current bounding box.center)}, scale=0.9]
		\draw[very thick,line cap=round] (0,0)--(0,0.5);
		
		\draw[very thick,line cap=round] (0,0.5) to[in=180,out=90] (0.125,0.65) to[out=0, in=90] (0.25,0.5);	
		\draw[color=black, fill]  (0.125,0.65) circle (0.05);
		
		\draw[very thick,line cap=round] (0,0) -- (0,-0.5);
		\draw[very thick,line cap=round] (0.25,0) -- (0.25,-0.5);
		
		\draw[very thick,line cap=round] (0.25,0) to[in=180,out=90] (0.375,0.15) to[out=0, in=90] (0.5,0);
		
		\draw[very thick,line cap=round] (0.25,0.5) to[in=180,out=-90] (0.375,0.35) to[out=0, in=-90] (0.5,0.5);	
		
		\draw[very thick,line cap=round] (0.5,0) -- (0.5,-0.5);
\end{tikzpicture}}

\newcommand{\DottedCapNaturalityFour}{\begin{tikzpicture}[baseline={([yshift=-.6ex]current bounding box.center)}, scale=0.9]
		\draw[very thick,line cap=round] (0,0)--(0,0.5);
		
		\draw[very thick,line cap=round] (0,0.5) to[in=180,out=90] (0.125,0.65) to[out=0, in=90] (0.25,0.5);	
		\draw[color=black, fill]  (0.125,0.65) circle (0.05);
		
		\draw[very thick,line cap=round] (0,-0.5) to[in=180,out=90] (0.125,-0.35) to[out=0, in=90] (0.25,-0.5);
		
		\draw[very thick,line cap=round] (0,0) to[in=180,out=-90] (0.125,-0.15) to[out=0, in=-90] (0.25,0);	
		
		\draw[very thick,line cap=round] (0.25,0) to[in=180,out=90] (0.375,0.15) to[out=0, in=90] (0.5,0);
		
		\draw[very thick,line cap=round] (0.25,0.5) to[in=180,out=-90] (0.375,0.35) to[out=0, in=-90] (0.5,0.5);	
		
		\draw[very thick,line cap=round] (0.5,0) -- (0.5,-0.5);
\end{tikzpicture}}

\newcommand{\DottedDiagCap}{\begin{tikzpicture}[baseline={([yshift=-.6ex]current bounding box.center)}, scale=0.9]
		
		\draw[very thick,line cap=round] (0.25,-0.5) to[in=180,out=90] (0.375,-0.35) to[out=0, in=90] (0.5,-0.5);	
		\draw[color=black, fill]  (0.25,0) circle (0.05);
		
		\draw[very thick,line cap=round] (0.25,0) -- (0.5,0.5);
		
		\draw[very thick,line cap=round] (0,-0.5) -- (0.25,0) -- (0.5,0.5);
		
\end{tikzpicture}}

\newcommand{\DiagDottedCap}{\begin{tikzpicture}[baseline={([yshift=-.6ex]current bounding box.center)}, scale=0.9]
		
		\draw[very thick,line cap=round] (0.25,-0.5) to[in=180,out=90] (0.375,-0.35) to[out=0, in=90] (0.5,-0.5);	
		\draw[color=black, fill]  (0.375,-0.35) circle (0.05);
		
		\draw[very thick,line cap=round] (0.25,0) -- (0.5,0.5);
		
		\draw[very thick,line cap=round] (0,-0.5) -- (0.25,0) -- (0.5,0.5);
		
\end{tikzpicture}}

\newcommand{\CapDottedId}{\begin{tikzpicture}[baseline={([yshift=-.6ex]current bounding box.center)}, scale=0.9]
		
		\draw[very thick,line cap=round] (0,-0.5) to[in=180,out=90] (0.125,-0.35) to[out=0, in=90] (0.25,-0.5);	
		\draw[color=black, fill]  (0.5,0) circle (0.05);		
		\draw[very thick,line cap=round] (0.5,-0.5) -- (0.5,0.5);
		
\end{tikzpicture}}

\newcommand{\DottedCapId}{\begin{tikzpicture}[baseline={([yshift=-.6ex]current bounding box.center)}, scale=0.9]
		
		\draw[very thick,line cap=round] (0,-0.5) to[in=180,out=90] (0.125,-0.35) to[out=0, in=90] (0.25,-0.5);	
		\draw[color=black, fill]  (0.125,-0.35) circle (0.05);		
		\draw[very thick,line cap=round] (0.5,-0.5) -- (0.5,0.5);
		
\end{tikzpicture}}

\begin{document}
\title{A construction of surface skein TQFTs and their extension to 4-dimensional 2-handlebodies}
\author{Leon J. Goertz}
\address{Fachbereich Mathematik, Universit\"at Hamburg, 
Bundesstra{\ss}e 55, 
20146 Hamburg, Germany
}
\email{leon.goertz@uni-hamburg.de}
\begin{abstract}
    	For a commutative Frobenius algebra $A$, we construct a $(2,3,3+\varepsilon)$-dimensional TQFT $\sktft_A$ that assigns to a 3-manifold a skein module of embedded $A$-decorated surfaces. These surface skein modules have been first defined by Asaeda--Frohman and Kaiser using skein relations that generalize the combinatorics of Bar-Natan's dotted cobordisms. For 3-manifolds with boundary, we show that surface skein modules carry an action by a certain surface skein category associated with the boundary, which yields a gluing formalism. Our main result concerns a partial extension of $\sktft_A$ to dimension 4, which uses an inductive state-sum construction following Walker. As an example, the equivariant version of Lee's deformation of dotted cobordisms yields a TQFT that extends to 4-dimensional 2-handlebodies but not 3-handlebodies. Finally, we characterize the attachment of 4-dimensional 2-handles by means of a certain Kirby color and use it to compute the invariants of 4-dimensional 2-handlebodies in examples.
\end{abstract}
\maketitle
\tableofcontents

\section{Introduction}
Skein theory provides a framework for globalizing local algebro-categorical data to manifold invariants and has deep connections to representation theory and mathematical physics. For example, the skein theory associated to ribbon categories of quantum group representations gives rise not only to the Reshetikhin--Turaev link invariants \cite{ReshetikhinTuraevRibbon}, but also to skein modules of 3-manifolds that play a fundamental role for related topological quantum field theories (TQFTs) \cite{ReshetikhinTuraev,WittenJonesPoly}.
\smallskip

To situate our work in the landscape of skein theories, we first recall that the Temperley--Lieb category, as a monoidal category, is closely related to the Turaev--Viro 3-dimensional TQFT for quantum $\mathfrak{sl}_2$ \cite{TuraevViro,RobertsTV}. When equipped with a braiding, the Temperley--Lieb category and the Kauffman bracket are connected to the Crane--Yetter 4-dimensional TQFT \cite{CraneYetter} for quantum $\mathfrak{sl}_2$. The construction of skein lasagna modules \cite{PaulInvariantsOf4Man} from Khovanov homology---which detects smooth exotic structures \cite{RenWillisExotic}---can be understood as a categorification of a Crane--Yetter TQFT and is based on a locally linear braided monoidal 2-category. Forgetting the braiding, one obtains the Bar-Natan monoidal 2-category which is the basis for the constructions of Asaeda--Frohman--Kaiser \cite{AsaedaFrohman, KaiserFrobAlg} and categorifies the Temperley--Lieb category. The relative positioning of these theories, which are all related to quantum link invariants, is summarized in the table below, see also \cite[Section 4]{WedrichSurveyLinkHomTQFT}. Note that the \emph{categorified skein theories} in the rightmost column have not yet been extended upwards in dimension\footnote{The Douglas--Reutter \cite{DouglasReutter} theory based on fusion 2-categories extends to 4-manifolds and can be described via skein theory akin to the Asaeda--Frohman--Kaiser theory considered here.}.
\begin{center}
\begin{tabular}{c||c|c}
     & linear categories & locally linear 2-categories \\
     \hline \hline
     monoidal & Turaev--Viro \cite{TuraevViro} &  Asaeda--Frohman--Kaiser \cite{AsaedaFrohman, KaiserFrobAlg} \\
     \hline
     braided & Crane--Yetter \cite{CraneYetter} & Morrison--Walker--Wedrich \cite{PaulInvariantsOf4Man} \\
\end{tabular}
\end{center}
 The purpose of this paper is to demonstrate such an extension of the Asaeda--Frohman--Kaiser TQFT from a 3-dimensional skein theory to a 4-manifold invariant---thereby providing a type of categorification of Turaev--Viro theory using Bar-Natan’s \emph{dotted cobordisms} \cite{BarNKhoHomTangles}.

\subsection{Surface skein theory as a TQFT}
Following the definition by Asaeda--Frohman \cite{AsaedaFrohman} and Kaiser \cite{KaiserFrobAlg}, the \emph{surface skein module} $\Sk_A(M,c)$ of a 3-manifold $M$ and boundary condition given by a 1-manifold $c\subset \partial M$ is spanned by isotopy classes of embedded surfaces bounding $c$ with decorations by elements of a commutative Frobenius algebra $A$ over a commutative ring $\kk$. These are considered up to local skein relations determined by $A$:
\begin{align}\label{eq:skeinrel}
	\asheet{a\quad\; b} = \asheet{m(a,b)}\;, \qquad\quad\asphere = \varepsilon(a)\,,\qquad\quad \aneck = \sum_i\; \cutupdown{x_i}{y_i}\;.
\end{align}
Here $a,b\in A$ are decorations on the shown surface, $m$ and $\varepsilon$ denote the multiplication and counit of $A$, and the last relation uses the comultiplication with $\Delta(a)=\sum_i x_i\otimes y_i$. For details see Definition~\ref{def:skeinmod}.

\begin{exm*}
    The commutative Frobenius algebra underlying Bar-Natan's local construction \cite{BarNKhoHomTangles} of Khovanov homology \cite{KhovCatJonesPoly} is the algebra $A_{\mathrm{BN}}=\kk[x]/(x^2)$ with counit $\varepsilon\colon x\mapsto 1$ and $1\mapsto 0$ and comultiplication $\Delta\colon 1\mapsto 1\otimes x + x\otimes 1$. A dot on a surface represents its decoration by $x\in A_{\mathrm{BN}}$. The surface skein relations for $A_{\mathrm{BN}}$ are
    \begin{align*}
        \ddsheet = 0\,,\qquad\quad \dsphere = 1\,, \qquad\quad \sphere = 0\,, \qquad\quad \neck = \cutddown + \cutdup\;.
    \end{align*}
\end{exm*}
The surface skein module of the solid torus for $A_{\mathrm{BN}}$ has been studied by Russell \cite{RussellBNSolidTorus} and Heyman \cite{HeymanPhDThesis}. In the PhD thesis \cite{FadaliPhDThesis}, surface skein modules for $A_{\mathrm{BN}}$ with checkerboard coloring are considered. Several variants of Khovanov's construction with different underlying Frobenius algebras have been explored \cite{LeeDeform,BarNatanMorrsion,KhovanovFrobeniusI,NaotPhDThesis,MackaayVaz}. Surface skein modules for the Frobenius algebra $\kk[x]/(p)$ with a polynomial $p\in \kk[x]$ and with counit $\varepsilon\colon x^{k}\mapsto \delta_{k,N-1}$ are studied in \cite{BoernerDrube}. In more recent work, Kaiser \cite{KaiserTunneling} develops methods to study the structure of surface skein modules using tunneling graphs.
\smallskip

The \emph{surface skein category} $\skcat_A(\Sigma)$ of a surface $\Sigma$ has objects given by embedded 1-manifolds in $\Sigma$ and $\kk$-linear morphism spaces given by the surface skein modules of $\Sigma\times [0,1]$ with appropriate boundary conditions. The composition is inherited from a stacking operation in the interval direction; see Definition~\ref{def:skeincat}.

\begin{thmx}[Construction \ref{con:theTQFT}]
    For a commutative Frobenius algebra $A$, surface skein theory assembles into a $(2,3,3+\varepsilon)$-dimensional TQFT, that is, a functor of symmetric monoidal bicategories 
\begin{align*}
	\sktft_A\colon \cob_{2,3,3+\varepsilon} \quad &\to \quad \mor(\kcat)\\
    \Sigma \quad &\mapsto \quad \skcat_A(\Sigma)\\
    M\colon \Sigma \leftarrow \Sigma' \quad &\mapsto \quad {}_{\skcat_A(\Sigma)}\Sk_A(M)_{\skcat_A(\Sigma')}\\
    \text{diffeomorphisms}\quad &\mapsto \quad \text{bimodule isomorphisms}.
\end{align*}
\end{thmx}
We call $\sktft_A$ the \emph{surface skein TQFT} or \emph{Asaeda--Frohman--Kaiser TQFT}. The source is the symmetric monoidal bicategory $\cob_{2,3,3+\varepsilon}$ with closed surfaces as objects, 3-dimensional cobordisms as 1-morphisms, and diffeomorphisms up to isotopy as 2-morphisms. The target is the symmetric monoidal Morita bicategory of $\kk$-linear categories $\mor(\kcat)$. The surface skein TQFT $\sktft_A$ assigns to a closed surface $\Sigma$ the surface skein category $\skcat_A(\Sigma)$ and to a closed 3-manifold $M$ the surface skein module $\Sk_A(M)$. If $M\colon \Sigma \leftarrow \Sigma'$ is a cobordism between two closed surfaces $\Sigma$ and $\Sigma'$, the surface skein module of $M$ has boundary conditions given by closed embedded 1-manifolds in the boundary of $M$. It then carries a left and right action of the corresponding surface skein categories of the boundary surfaces making it into a bimodule.

\begin{rmk}
The Asaeda--Frohman--Kaiser TQFT $\sktft_A$ is manifestly a bicategorical version of a $(3+\varepsilon)$-TQFT controlled by a disklike 3-category in the sense of \cite{WalkerTQFT06,MorrsionWalkerBlob}. As an alternative to such a top-down approach, one can also extract a locally linear monoidal 2-category $\mathbf{BN}_A$ with duals and adjoints, see e.g.\ \cite[Proposition 2.5]{HRWBorderedInvariantsKh}, with the aim of using it to construct an fully local TQFT from the bottom up. We expect that $\mathbf{BN}_A$ is 3-dualizable when considered as an object in a suitable target 4-category, e.g.\ a higher Morita 4-category in the framework of Johnson-Freyd--Scheimbauer \cite{JFScheimbauerMorita} or a ``pointless'' analog thereof, and that the 3-dimensional part of the corresponding TQFT has a skein-theoretic description in terms of the surface skein modules defined here. Along these lines, the work in this paper on the extension of $\sktft_A$ to 4-dimensional handles is related to studying fractional 4-dualizability of $\mathbf{BN}_A$ as an object in the higher Morita 4-category.
\end{rmk}

\begin{rmk}
    Note that the construction of $\sktft_A$ is based on \emph{linear} surface skein modules and categories. Gluing the skein modules of 3-manifolds $M$ and $N$ along common boundary $\Sigma$ is captured by the (relative) tensor product over $\skcat_A(\Sigma)$. One could also consider a \emph{derived} version, whose construction starts with the monoidal 2-category $\mathbf{BN}_A$, now considered as locally enriched in chain complexes, and then proceeds as in \cite{HRWBorderedInvariantsKh}. The gluing of derived skein modules is then realized as a derived tensor product over a derived version of a surface skein category.
\end{rmk}

\begin{rmk}
It is also natural to study the skein theory of \emph{foams} in $3$-manifolds, for example Blanchet's $\gl_2$-foams from \cite{MR2647055}, which are related to Bar-Natan's dotted cobordisms, see \cite{MR4598808}. For the skein modules, one replaces the boundary curves by webs, and the decorated embedded surfaces by decorated foams. We will consider such skein theories in future work.
\end{rmk}

\subsection{Extension of surface skein theory to 4-dimensional manifolds}
Since surface skein theory $\sktft_A$ assigns skein modules to 3-dimensional manifolds, one might hope it can assign $\kk$-linear maps between skein modules to 4-dimensional cobordisms between the 3-manifolds
\begin{align*}
    W\colon M\to M' \quad \mapsto \quad \sktft_A(W)\colon\Sk_A(M)\to \Sk_A(M').
\end{align*}
In Section \ref{sec:extension}, we construct such an extension to 4-manifolds by choosing a handle decomposition and assigning linear maps to 4-dimensional handles. For most commutative Frobenius algebras $A$, however, there does not exist an extension of $\sktft_A$ to all 4-manifolds. Instead, we construct invariants of \emph{4-dimensional $k$-handlebodies}, 4-manifolds built from handles of index $\leq k$ considered up to handle cancellations and handle slides involving only such handles. 
\smallskip

In Construction \ref{con:extension}, we apply Walker's universal state-sum \cite{WalkerUnivState} to $\sktft_A$ and inductively construct for $4$-dimensional $k$-handles a linear map between the skein modules of the 3-dimensional boundary regions. This procedure starts by choosing data for the 0-handle, and iteratively constructs for every $k\geq 1$, a pairing $p_k$ on the skein module of the attaching region of a $k$-handle. If $p_k$ is perfect, the dual copairing is used to construct the map assigned to a $k$-handle. If the pairings $p_l$ are perfect for all $l\leq k$, then Walker's construction provides an extension to 4-dimensional $k$-handlebodies, see \cite[Theorem 4.1.2]{WalkerUnivState} and Theorem \ref{thm:WalkerExtension}. A related framework for extending TQFTs via handle attachments appears in \cite{HaiounHandleTQFT}.
\smallskip

We study how the perfectness of the pairing and hence the extendability of the surface skein TQFT $\sktft_A$ depends on the commutative Frobenius algebra $A$. An extension to 0-handles is data and the possible choices are scalars in the ground ring. For an invertible scalar, the extension to 1-handles is then guaranteed by the perfectness of the Frobenius pairing. 
\begin{thmx}[Corollary \ref{cor:allAExtendToOneH}]
For every commutative Frobenius algebra $A$, the surface skein theory $\sktft_A$ extends to $4$-dimensional 1-handlebodies. The extensions are parameterized by invertible scalars and assign to 4-dimensional cobordisms maps that only differ by a rescaling dependent on Euler characteristic. Up to this scalar, the invariant assigned by $\sktft_A$ to 4-dimensional 1-handlebody $W$ is the linear functional
\begin{align*}
    \sktft_A\colon \Sk_A(\partial W) \to \kk, \quad S\mapsto \eval_A(S)
\end{align*}
given on a decorated surface $S$ in $\partial W$ as the \emph{abstract evaluation} of $S$, defined in Construction \ref{con:abstractEval} as $\kk$-linearization and extension to decorated surfaces of the 2-dimensional TQFT associated to $A$.
\end{thmx}

\begin{exm*}[Theorem \ref{thm:rkOneTheory} and Proposition \ref{prop:rkOneAndDW}]
 Every commutative ring $\kk$ can itself be given the structure of a commutative Frobenius algebra, see Example~\ref{exm:trivFA}. If $2\in \kk^\times$, the surface skein theory $\sktft_\kk$ associated to the Frobenius algebra $\kk$ extends to all 4-dimensional $k$-handles. Let $W$ be a closed, connected, oriented 4-manifold. Then $\sktft_\kk$ assigns to $W$ the value
	\begin{align*}
		\sktft_{\kk}(W) = \frac{1}{2}|H_3(W;\ZZ/2)| = \mathcal{DW}_{\ZZ/2}(W)\in \kk
	\end{align*} 
    where $\mathcal{DW}_{\ZZ/2}(W)$ is the invariant assigned to $W$ by 4-dimensional Dijkgraaf-Witten theory for $G=\ZZ/2$.
\end{exm*}

\begin{prop*}[Proposition \ref{prop:Ext2handlesSeparable}]
    A necessary condition for the extension to 2-handles is the \emph{strong separability} of the commutative Frobenius algebra $A$, that is, the invertibility of the handle element $H=m\circ \Delta(1)\in A$. 
\end{prop*}

\begin{exm*}[Example \ref{exm:BNnoExtension}]
    The surface skein theory $\sktft_{A_\mathrm{BN}}$ associated to Bar-Natan's dotted cobordisms $A_\mathrm{BN}$ does not extend to 4-dimensional 2-handlebodies.
\end{exm*}

In Section \ref{sec:KirbyCol}, we focus on the commutative Frobenius algebra
\begin{align*}
	A_\alpha = K[\alpha^{\pm1}][x]/(x^2-\alpha), \quad\quad \varepsilon\colon x\mapsto 1, \; 1\mapsto 0,\quad  \quad \Delta\colon 1\mapsto 1\otimes x + x\otimes 1,
\end{align*}
an equivariant version of Lee's deformation of Bar-Natan's dotted cobordisms \cite{LeeDeform,BarNatanMorrsion}. Here $K$ is a field with $2\in K^\times$ and we set $\kk:= K[\alpha^{\pm1}]$. Note that we choose the parameter $\alpha$ to be invertible. This ensures the invertibility of the handle element $H=2x$ in $A_\alpha$. 

\begin{thmx}\label{thmx:alphaExtension}
    The surface skein theory $\sktft_{A_\alpha}$ associated to the commutative Frobenius algebra $A_\alpha$ extends to 4-dimensional 2-handlebodies, but not to 3-handlebodies.
\end{thmx}

TQFTs with this level of extendability are of interest in the context of 4-dimensional analogs of the Andrews--Curtis conjecture, see Remark~\ref{rmk:AC}. A possibly related example of a surface skein TQFT with equivalent extendability appears in (yet unpublished) work of Walker \cite{WalkerPartitionTQFT}. Examples sourced from ribbon categories, i.e.\ skein theory with codimension 2 skeins, are the subject of e.g.\ \cite{BeliakovaDeRenzi4d2h,Beliakova_2023,beliakova2023algebraizationlowdimensionaltopology}.
\smallskip

In general, for a commutative Frobenius algebra $A$, the extension of the surface skein theory $\sktft_A$ to 4-dimensional 2-handles, if it exists, can be computed explicitly using a \emph{Kirby color} as we explain next.

\subsection{A Kirby color for 4-dimensional 2-handles}

\begin{thmx}[Construction \ref{con:TheInvariant}]\label{thmx:TheInvariant}
    Consider a commutative Frobenius algebra $A$ such that the surface skein TQFT $\sktft_A$ extends to 4-dimensional 2-handlebodies. Then there exists a family $\{\omega_{n}\in A^{\otimes n}\mid n\geq 0\}$, called \emph{Kirby color}, such that the linear functional 
    \begin{align*}
	   \sktft_A(W) \colon \Sk_A(\partial W)\to \kk
    \end{align*}
    given as the invariant of a 4-dimensional 2-handlebody $W$ with boundary $\partial W$ can be described as follows. Let $L^\vee_j\subset \partial W$ denote the knots obtained as boundaries of the cocores of the 2-handles. Consider a decorated surface $S\in \Sk_A(\partial W)$ intersecting every $L^\vee_j$ transversely in a finite set $P_j$. Puncturing $S$ at all such intersection points $P:=\sqcup_j P_j$ and applying the 2-dimensional TQFT\footnote{Suitably $\kk$-linearized and extended to decorated cobordisms as in Construction~\ref{con:abstractEvalext}.} associated to $A$ defines an evaluation 
    \begin{align*}
    \eval_A(S,P)\colon  \bigotimes_j A^{\otimes P_j} &\to \kk \\
                \bigotimes_j \omega_{|P_j|} & \mapsto  \sktft_A(W)(S).
    \end{align*}
\end{thmx}

For the proof of Theorem~\ref{thmx:TheInvariant}, we carefully study the surface skein module of the solid torus $S^1\times B^2$, the attaching region for a 4-dimensional 2-handle. To construct the map associated to the 2-handle, we can restrict to the boundary conditions $(2n,0)$ of the surface skein module of the solid torus, $2n$ curves parallel to the longitude. For Bar-Natan theory $A_{\mathrm{BN}}$, this has been studied by Russell \cite{RussellBNSolidTorus} and Heyman \cite{HeymanPhDThesis}. We analyze its structure for commutative Frobenius algebras $A$ using different methods. In Section \ref{sec:KirbyCol}, we first establish a general framework to compute the structure of the skein module of the solid torus by computing idempotents of the endomorphism algebras of the $\kk$-linear Bar-Natan disk category $\BN_A(B^2,2n)$ with $2n$ points on the boundary $\partial B^2$. We obtain the skein module of the solid torus as the trace
\begin{align*} 
	\Tr(\BN_A(B^2,2n))\cong \Sk_A(S^1\times B^2,(2n,0)).
\end{align*}

If the pairing $p_2$ on this skein module is perfect, then the attachment of a 2-handle intersecting a surface $2n$ times can be modeled by the Kirby color $\omega_{2n}\in \Sk_A(S^1\times B^2, (2n,0)) \hookrightarrow A^{\otimes 2n}$ as stated in Theorem D. Due to homological obstructions, the Kirby color for an odd number of curves parallel to the longitude is zero.

\begin{rmk}
It is important to note that there are several related, but different notions of Kirby color in similar contexts. For instance, the Kirby color constructed by Hogancamp--Rose--Wedrich \cite{HRWKirbyColorForKh} is an object in an appropriate completion of the \emph{dotted Temperley--Lieb category} and models the attachment of 4-dimensional 2-handles of skein lasagna modules as constructed in \cite{MNSkeinLasagna2Handle, MWWSkeinLasagnaHandle}.
It is conceivable, however, that their Kirby color and its generalization to commutative Frobenius algebras $A$ may be applicable to $\sktft_A$, namely to model 3-dimensional 2-handle attachments. By comparison, our Kirby color for $\sktft_A$ is relevant one dimension higher.
\end{rmk}

For the proof of Theorem~\ref{thmx:alphaExtension}, we study the skein module of the solid torus for $A_\alpha$ and explicitly check that the pairing $p_2$ is perfect. As first step, we determine the general structure of the idempotents in the Bar-Natan disk category $\BN_{A_\alpha}(B^2,2n)$ in Proposition~\ref{prop:structureOfIdem} and \ref{prop:isoClassesOfIdem}. In Corollary~\ref{cor:isoClassesWalks} we show that a $\kk$-basis on the surface skein module $\Sk_{A_\alpha}(S^1\times B^2, (2n,0))$ is provided by the set of walks on $\ZZ$ of length $2n$ starting and ending at 0. We compute the pairing on this basis and show that is perfect, see Proposition~\ref{prop:pairing}.
From the copairing, we compute the Kirby color modeling the attachment of 4-dimensional 2-handles for $A_\alpha$ and find two concrete expressions. 
\smallskip

Viewing the Kirby color for $\sktft_{A_\alpha}$ as an element in $A_{\alpha}^{\otimes 2n}$, we arrive at the first expression: Let $z\in (\ZZ/2)^{2n}$ and write $T_z=x^{z_1}\otimes \dots \otimes x^{z_{2n}} \in A_{\alpha}^{\otimes 2n}$ for the element with $|z|=|\{i \mid z_i=1\}|$ many $x$s. The collection of all $T_z$ forms a basis of $A_\alpha$. We show in Theorem~\ref{thm:KirbyAlpha} that the Kirby color for $\sktft_{A_\alpha}$ has the form
\begin{align*}
    \omega_{2n} = \frac{1}{2^{2n}}\sum_{k=0}^{n}\frac{1}{\alpha^k}S(n-k,k)\sum_{\substack{z\in (\ZZ/2)^{2n} \\|z|=2k}} \mathrm{sign}(z,k)T_z \;\in A_\alpha^{\otimes 2n}
\end{align*}
where $S(n-k,k)$ are the super Catalan numbers and $\mathrm{sign}(z,k)\in \{\pm1\}$.\\

For the second expression, we describe elements of the skein module of the solid torus in terms of $A$-\emph{decorated Temperley--Lieb} diagrams, see Proposition \ref{prop:SolidtorusAnddTL}. For any given Frobenius algebra $A$, these diagrams assemble into a $\kk$-linear monoidal category $\dTL_A$, a generalization of the dotted Temperley--Lieb category. Specializing to $A_\alpha$, we have 
\begin{align*}
    \dTL_{A_\alpha}(0,2n)\cong \Sk_{A_\alpha}(S^1\times B^2,(2n,0)).
\end{align*}
On this module we have an action of the symmetric group $S_{2n}$. More generally, if the field $K$ has characteristic $\mathrm{char}(K)=0$, we form for every $m\geq 0$ the symmetrizer
\begin{align*}
    \Symmet{\mathrm{Sym}_{m}}\; =\; \frac{1}{m!}\sum_{\sigma\in S_m} P_\sigma \in \dTL_{A_\alpha}(m,m).
\end{align*}
and use this to obtain the second expression of the Kirby color:
\begin{thmx}[Theorem \ref{thm:KirbySymmetrizer}]
The Kirby color for $\sktft_{A_\alpha}$, considered as family of elements in $\dTL_{A_\alpha}(0,2n)$, is given by
\begin{align*}
    \kirbycol{\omega_{2n}}\;= \tfrac{1}{2^n} \tbinom{2n}{n}\; \SymmetKirby{\mathrm{Sym}_{2n}} \; \in \dTL_{A_\alpha}(0,2n),\quad \text{where}\quad \SepidemCup\; =\; \tfrac{1}{2\alpha}\;\dcupbig\;\in \dTL_{A_\alpha}(0,2)
\end{align*}
corresponds to the separability idempotent in $A_\alpha^{\otimes 2}$, and $\mathrm{Sym}_{2n}$ is the symmetrizer on $2n$ strands.
\end{thmx}

The Kirby color satisfies an annulus capping property which is explained in Remark \ref{rmk:kirbyRecursiveCapping} and proved explicitly for $A_\alpha$ in Corollary \ref{cor:annulusCapping}.
	\begin{align*}
		\kirbyCap{\omega_{2n}} = 0\in \dTL_{A_\alpha}(0,2n) \quad \quad \text{and} \quad \quad \kirbydCap{\omega_{2n}}=\kirbywide{\omega_{2n-2}}\;\in \dTL_{A_\alpha}(0,2n-2).
	\end{align*}
We expect that this lends itself to determining the Kirby color for commutative Frobenius algebras in greater generality. 
\smallskip

Finally, we compute the invariant of the $4$-dimensional $2$-handlebodies $S^2\times B^2$, $S^1\times B^3$, and ${S^1\times S^1\times B^2}$ using Construction \ref{con:TheInvariant} and our expressions for the Kirby color in Proposition~\ref{prop:InvSTwoBTwo}, Remark~\ref{rmk:InvThickenedCircle} and Proposition~\ref{prop:InvThickenedTorus}, respectively.

\subsection*{Conventions}
All manifolds are considered smooth, if not otherwise stated. All rings and algebras are always assumed to be unital and associative. By $\kk$, we denote a commutative (unital) ring.
\subsection*{Acknowledgments}
The author would like to thank Jesse Cohen, Catherine Meusburger, David Reutter and Max-Niklas Steffen for helpful discussions. The author especially thanks Kevin Walker and Paul Wedrich for suggesting this project, and for their guidance and insight along the way.
The author acknowledges support from the Deutsche Forschungsgemeinschaft (DFG, German Research Foundation) under Germany's Excellence Strategy - EXC 2121 ``Quantum Universe'' - 390833306 and partial funding from the Collaborative Research Center - SFB 1624 ``Higher structures, moduli spaces and integrability'' - 506632645.

\section{Surface skein theory from commutative Frobenius algebras}
In this section, for a commutative Frobenius algebras $A$, we construct surface skein modules of $3$-manifolds and surface skein categories of closed surfaces. We then recall definitions of cobordism and Morita bicategories which will be the source and target for the surface skein TQFT $\sktft_A$. We start with a brief recap of Frobenius algebras and their relation to 2-dimensional TQFT.
\subsection{Frobenius algebras}
Let $\kk$ be a commutative ring. A Frobenius algebra\footnote{Some authors call this a Frobenius extension.} over $\kk$ is a tuple $A=(A,m,\eta, \Delta,\varepsilon)$ such that $(A,m,\eta)$ is an associative algebra over $\kk$ with multiplication $m\colon A\otimes A\to A$ and an injective map $\eta\colon \kk \to A$ as a unit, and $(A,\Delta,\varepsilon)$ is a coassociative coalgebra over $\kk$ with comultiplication $\Delta\colon  A\to A\otimes A$ and counit $\varepsilon \colon A\to \kk$ satisfying the Frobenius relation
\begin{align*}
	(m \otimes \id ) \circ (\id \otimes \Delta) = \Delta \circ m = (\id \otimes m) \circ (\Delta \otimes \id) 
\end{align*}
that is, the comultiplication $\Delta \colon A\to A\otimes A$ is an $(A,A)$-bimodule map. 
The Frobenius pairing $\beta\colon A\otimes A\to \kk$ is a perfect pairing defined as $\beta = \varepsilon \circ m$. The comuliplication $\Delta\circ \eta \colon \kk \to A\otimes A$ provides a copairing. Recall that the comultiplication in a Frobenius algebra can be recovered from the underlying associative algebra together with the counit $\varepsilon$. 
We will work with Frobenius algebras that are free as a $\kk$-module, so that we can choose bases $(x_i)$ and $(y_i)$ of $A$ over $\kk$ such that $\Delta(1) = \sum_i x_i\otimes y_i$ and $\beta(x_i,y_j)=\delta_{ij}$. See \cite{KadisonNewExOFrobExt} and \cite{KockFA} for more details on Frobenius algebras.

It is a well-known fact \cite{AbramsTFTFrobAlg} that commutative Frobenius algebras $A$ that are projective over $\kk$ correspond to 2-dimensional TQFTs $\mathcal{F}_A$, that is, symmetric monoidal functors 
\begin{align*}
\mathcal{F}_A\colon \cob_2 \to \kk\mathsf{Mod} 
\end{align*}
where $\cob_2$ is the cobordism category with objects closed oriented 1-dimensional manifolds and morphisms oriented 2-dimensional cobordisms. The monoidal structure is given by disjoint union. The standardly oriented\footnote{The oppositely oriented circle is orientation-preservingly diffeomorphic to this standard circle.} circle $S^1$ is mapped to the $\kk$-module $A$ and the Frobenius structure maps arise as images of the following standard cobordisms, drawn to be read from bottom to top:
\begin{equation}\label{eq:standardcobs}
    \frobunit \mapsto \eta,\quad\frobcounit\mapsto \varepsilon,\quad  \frobmult\mapsto m, \quad \frobcomult\mapsto \Delta.
\end{equation}

For example, the value $\mathcal{F}_A(S^1\times S^1)$ is the map $\varepsilon\circ m\circ \Delta\circ \eta$, i.e.\ the endomorphism of $\kk$ given by multiplying by the rank of $A$. 

\begin{rmk}
    There is an \emph{unoriented} version of the 2-dimensional cobordism category denoted by $\ucob$. Objects are unoriented closed 1-manifolds. Morphisms are unoriented 2-dimensional cobordisms. These are generated under composition and disjoint union by the usual cup, cap, pair of pants and upside-down pair of pants morphisms together with two additional morphisms: the crosscap (the once-punctured $\RP^2$) and the cobordism induced by the orientation-reversing diffeomorphism of the circle. See \cite[Section 3.1]{CzenkyUnoriented} for more details. Symmetric monoidal functors $\ucob\to \kk\mathsf{Mod}$ are called \emph{unoriented} 2-dimensional TQFTs. Turaev--Turner \cite[Proposition 2.9]{TuraevTurnerUnoriented} show that unoriented 2-dimensional TQFTs correspond to commutative Frobenius algebras with an \emph{extended structure}\footnote{Note that a Frobenius algebra with an extended structure is not to be confused with extending a TQFT to other dimensions.}. This classification starts by evaluating an unoriented TQFT on standard circles and cobordisms as in \eqref{eq:standardcobs}, which implies the existence of an underlying commutative Frobenius algebra $A=(A,m,\eta,\Delta,\varepsilon)$. An \emph{extended structure} on $A$ in the sense of \cite[Definition 2.5]{TuraevTurnerUnoriented}, see also \cite[Definition 2.10]{CzenkyUnoriented}, is a tuple $(\phi, \theta)$. Here, $\phi\colon A\to A$ is an involution of Frobenius algebras arising as the image of the orientation-reversing cobordism of the circle, and the element $\theta\in    A$ arises as the image of the crosscap. This data is required to satisfy the following relations.
    \begin{enumerate}[(i)]
	   \item $m(\phi \otimes \id)(\Delta(1)) = \theta^2$,
	   \item $\phi(\theta a) = \theta a$ for all $a\in A$.
    \end{enumerate}
\end{rmk}

\begin{exm}[Trivial Frobenius algebra]\label{exm:trivFA}
	Let $\kk$ be a commutative ring and let $u\in \kk^\times$ be a unit in $\kk$. Denote by $\kk^u$ the commutative Frobenius algebra with $\kk$ as underlying $\kk$-module and counit given by $\varepsilon\colon 1\mapsto u$. The comultiplication is $\Delta\colon 1\mapsto u^{-1}1\otimes 1$. If $u=1$, we simply write $\kk=\kk^u$.
\end{exm}
 In Section \ref{sec:TrivialTheoryDW}, we consider surface skein theory associated to $\kk^u$ with the additional assumption that $2\in \kk^\times$.

\begin{exm}[Frobenius algebras in link homology theories]\label{exm:FAExmLinkHom}
	We give the following examples of free commutative Frobenius algebras over a commutative ring $\kk$ that are relevant for us. All of these are versions of the Frobenius algebras underlying constructions in link homology theories.
	\begin{enumerate}[(i)]
		\item Khovanov homology \cite{KhovCatJonesPoly} is governed by the Frobenius algebra $A_{\mathrm{BN}} = \kk[x]/(x^2)$ which is free of rank 2 with basis $\{1,x\}$ and has counit and comultiplication characterized by
		\begin{align*}
			\varepsilon\colon 1\mapsto 0, \quad x\mapsto 1&&	\Delta\colon 1\mapsto 1\otimes x + x\otimes 1, \quad x\mapsto x\otimes x.
		\end{align*}
		\item A modification of $A_{\mathrm{BN}}$ and the most important example for us is the following. Let $K$ be a field with $2\in K^\times$ and let $\kk=K[\alpha^{\pm1}]$. Consider the Frobenius algebra $A_\alpha =\kk[x]/(x^2-\alpha)$ free of rank 2 over $\kk$ with basis $\{1,x\}$. The counit and comultiplication are given on this basis by
		\begin{align*}
			\varepsilon\colon 1\mapsto 0, \quad x\mapsto 1&& \Delta\colon 1\mapsto 1\otimes x + x\otimes 1, \quad x\mapsto x\otimes x + \alpha 1\otimes 1.
		\end{align*}
		After setting $\alpha=1$, we recover the Frobenius algebra used in Lee's deformation of Khovanov homology \cite{LeeDeform, BarNatanMorrsion}. We refer to $A_\alpha$ as the \emph{equivariant version} of Lee's Frobenius algebra.
		\item For a generalization to rank $N\geq 2$, we consider $\kk=K[\beta^{\pm1}]$ for a field $K$ with $N\in K^\times$ and a Frobenius algebra structure on $A_\beta =\kk[x]/(x^N-\beta)$. A basis is given by $\{1, x, \dots, x^{N-1}\}$. The counit and comultiplication are characterized by
		\begin{align*}
			\varepsilon\colon x^{k}\mapsto \delta_{k,N-1}, && \Delta\colon 1\mapsto \sum_{i=0}^{N-1} x^{N-1-i}\otimes x^i.
		\end{align*}
		We obtain the values $\Delta(x^k)$ from $\Delta(1)$ by multiplication since $\Delta$ is a bimodule map. 
		\item Let $K$ be a field with $N\in K^\times$ and $\kk=K[\mu_1, \dots, \mu_N, \mathrm{disc}(\mu_1,\dots,\mu_N)^{-1}]$ where $\mathrm{disc}(\mu_1,\dots,\mu_N)$ is the discriminant of the polynomial $p=\prod_i (x-\mu_i)$ as a function of the roots $\mu_i$. Consider the Frobenius algebra $A_\mu =\kk[x]/(\prod_i (x-\mu_i))$. The counit is characterized by
		\begin{align*}
			\varepsilon\colon x^k\mapsto \delta_{k,N-1}
		\end{align*}
		and the comultiplication can be derived from $\varepsilon$. 
	\end{enumerate}
\end{exm}
We will see in Example \ref{exm:BNnoExtension}, that surface skein theory associated to $A_{\mathrm{BN}}$ extends to 4-dimensional 0- and 1-handles, but not beyond. The surface skein theory associated to $A_\alpha$ does extend to 4-dimensional 2-handles, as shown in Section \ref{sec:KirbyCol}, using the invertibility assumptions on $\alpha$ and $2$.
In Remark~\ref{rmk:genRkN} and \ref{rmk:copRkN}, we comment on the generalization of our results for $A_\alpha$ to the Frobenius algebra $A_\beta$, which uses the invertibility of $\beta$ and $N$. 
Furthermore, we sketch how our results for $A_\alpha$ and $A_\beta$ can be modified for the Frobenius algebra $A_\mu$ in Remark~\ref{rmk:genSemisimpleFA}. 

In the following, we indicate the dependence on the commutative Frobenius algebra as a subscript. To simplify notation, we often only use the subscript $\mathrm{BN}$, $\alpha$, $\beta$, or $\mu$ for the Frobenius algebras $A_{\mathrm{BN}}$, $A_\alpha$, $A_\beta$, or $A_\mu$ from Example~\ref{exm:FAExmLinkHom} respectively.

\begin{exm}
	Another example is the group algebra $\kk[G]$ for a finite group $G$ equipped with a comultiplication $\Delta \colon e\mapsto \sum_g g\otimes g^{-1}$ and counit $\varepsilon\colon g\mapsto \delta_{g,e}$. The group algebra for $G=\ZZ/N$ is closely related to the Frobenius algebra $A_\beta$ above. In particular, we have	$\kk[\ZZ/N]\cong\kk[x]/(x^N-1)$ as an algebra, but the comultiplication is given by $\Delta\colon 1\mapsto \sum_k x^k\otimes x^{N-k}$ and for the counit, we have $\varepsilon\colon x^k\mapsto \delta_{k,0}$. 
\end{exm}

\begin{rmk}
	Many of the above examples of Frobenius algebras admit a $\ZZ$-grading. While gradings are an important ingredient of constructions in link homology theories, we will work in an ungraded setting. If one were to use the $\ZZ$-grading on the Frobenius algebras, this would result in a $\ZZ$-grading on the $\kk$-modules and $\kk$-linear categories defined in the next section. We also expect that many of the constructions generalize to commutative Frobenius algebra objects in (locally small) symmetric monoidal categories. 
\end{rmk}

\subsection{Surface skein modules and categories}
In the following, we define skein modules of embedded surfaces in 3-manifolds for every commutative Frobenius algebra. These skein modules were first defined in \cite{AsaedaFrohman} and \cite{KaiserFrobAlg}. See also \cite{BoernerDrube}.
\begin{defn}[Decorated surfaces]\label{def:decSurface}
	Let $A$ be a commutative Frobenius algebra over a commutative ring $\kk$. An \emph{$A$-decorated surface} in a 3-manifold $M$ is a compact properly embedded surface $S\subset M$ with a map $\ell\colon \pi_0(S)\to A$ labeling every component of $S$ by an element of $A$.
\end{defn}

\begin{defn}[Skein module]\label{def:skeinmod}
	Let $M$ be a 3-manifold with boundary and let $c\subset \partial M$ be a closed 1-manifold. Let $A$ be a commutative Frobenius algebra over a commutative ring $\kk$. 
	\begin{enumerate}[(i)]
		\item We write $C_A(M,c)$ for the set of isotopy classes of $A$-decorated properly embedded surfaces in $M$ bounding $c$, with isotopies taken relative to the boundary. 
		\item We define the \emph{surface skein module} of $M$ with boundary $c$ to be the $\kk$-module
		\begin{align*}
			\Sk_A(M,c):= \kk\{C_A(M,c)\}/R_A(M,c)
		\end{align*}
		where $R_A(M,c)$ is the submodule of $\kk\{C_A(M,c)\}$ spanned by the following relations
		\begin{itemize}
			\item ($\kk$-multilinearity) The classes represented by $A$-decorated surfaces in $(M,c)$ are considered to be $\kk$-multilinear in the decorations on each connected component.
			\item (Sphere relations) 
			If a connected component of an $A$-decorated surface in $(M,c)$ is a 2-sphere bounding a ball in $M$, then this component with label $a\in A$ can be removed at the expense of scaling the remaining $A$-decorated surface by $\varepsilon(a)\in \kk$. 
            \item (Neck-cutting) Suppose an $A$-decorated surface $S$ in $(M,c)$ contains a simple closed curve $\gamma$ on a component decorated by $a\in A$, such that $\gamma$ bounds a disk $D$ in $M\setminus S$. Let $S'$ denote the surface obtained from $S$ by replacing a regular neighborhood of $\gamma$ by two parallel copies of $D$. Let $\Delta(a)=\sum_i x_i\otimes y_i$. If $\gamma$ is a separating curve, then we equate $S$ with the linear combination of $A$-decorated surfaces $S'$ with $x_i$ and $y_i$ on the components of the two disks. If $\gamma$ is non-separating, then the component of $S'$ containing both disks is given the decoration $\sum_i m(x_i,y_i)$. 
		\end{itemize}
	\end{enumerate}
We call the elements of $\Sk_A(M,c)$ \emph{(surface) skeins}. If $\partial M=\emptyset$, we also write $\Sk_A(M)=\Sk_A(M,\emptyset)$.
\end{defn}
\begin{rmk}
    Sometimes it is convenient to allow multiple labels on a connected component of an $A$-decorated surface. For this, we equate labels $a,b\in A$ on a connected component with the label $m(a,b)$. Then, this relation and the sphere and neck-cutting relations can be displayed pictorially as in \eqref{eq:skeinrel}.
\end{rmk}
\begin{rmk}\label{rmk:orientation}
	In the above definition, we consider \emph{unoriented} surface skeins.
	Alternatively, one could require \emph{oriented} surface skeins in an oriented 3-manifold $M$ bounding closed oriented 1-manifolds in $\partial M$ and analogously define an oriented surface skein module $\Sk^{\mathrm{or}}_A(M,c)$. Then, the neck-cutting relation can be applied from right to left only for two parallel disks having opposite orientations. One reason to consider oriented surface skein modules is their closer connection to skein lasagna modules as constructed in \cite{PaulInvariantsOf4Man}. Specifically, the embedding of a closed oriented 3-manifold $M$ into its thickening $M\times [0,1]$ induces a well-defined map
    \begin{align*}
        \Sk^{\mathrm{or}}_{A_{\mathrm{BN}}}(M)\to \mathcal{S}^2_0(M\times [0,1],\emptyset)
    \end{align*}
    from the oriented surface skein module of $M$ to the $\mathfrak{gl}_2$ skein lasagna module of $M\times [0,1]$ with empty link in the boundary. For a discussion of surface skein modules that deals with the question of orientation via checkerboard coloring, see \cite{FadaliPhDThesis}. 
\end{rmk}

\begin{rmk}\label{rmk:orientableSkeins}
    Asaeda--Frohman \cite{AsaedaFrohman} and Kaiser in his recent work \cite{KaiserTunneling} consider the surface skeins to be \emph{orientable}. However, to assemble surface skein theory into a TQFT, we require gluing properties which seem not to be achievable in the case of orientable surface skeins for the following reason. Consider the skein module $\Sk_A(S^2\times [0,1], c\times \{0,1\})$ where $c$ is the equator in $S^2$. Let $\id_c$ be the element represented by the surface $c\times [0,1]$ with decoration $1\in A$, and let $\phi_c$ be the element that is obtained by rotating $c$ by $\pi$ fixing two antipodal points of $c$, with decoration $1\in A$.
    Both of these skeins are represented by orientable surfaces. Gluing the skein module $\Sk_A(S^2\times [0,1], c\times \{0,1\})$ to itself should define a map into the skein module $\Sk_A(S^2\times S^1)$ gluing the element $\id_c$ to $\phi_c$. The resulting skein in $\Sk_A(S^2\times S^1)$ is represented by a Klein bottle embedded in $S^2\times S^1$ and is unorientable. To reconcile orientability when gluing, one could try to choose an orientation of the representing surfaces and induce an orientation on their boundary $c$, but the element $\id_c+\phi_c\in \Sk_A(S^2\times [0,1],c\times \{0,1\})$ does not allow for a consistent choice of orientation. 
\end{rmk}
\begin{rmk}
    One might expect that surface skein theory based on unoriented surfaces would require a commutative Frobenius algebra with an extended structure, corresponding to an unoriented TQFT. However, the guiding example $A_\mathrm{BN}$ from \cite{BarNKhoHomTangles} does not even admit an extended structure, as proven in \cite[Proposition 2.11]{CzenkyKestenQuinonezWaltonExtended}. In fact, one only needs a 2d TQFTs for orientable cobordisms, as described in the accompanying note \cite{GoertzWedrich}, which are classified by pairs $(A,\phi)$ where $A$ is a commutative Frobenius algebra and $\phi$ is an involution of Frobenius algebras. With the choice $\phi=\id_A$, every oriented 2d TQFT gives rise to an orientable one. In Construction \ref{con:abstractEval} and Lemma \ref{lem:absevalskein} below, we construct the abstract evaluation of surface skeins using the orientable 2d TQFT associated to $(A,\id_A)$, i.e.\ for surfaces of type $(\alpha$) and $(\gamma)$ below. Indeed, this provides a description of the surface skein module of the 3-ball in Example \ref{exm:FirstSkeinModules}.
\end{rmk}
\begin{con}[Abstract evaluation]\label{con:abstractEval}
    Let $A$ be a commutative Frobenius algebra over a commutative ring $\kk$. Recall that \emph{connected} compact surfaces $C$ are classified into three types:
    \begin{enumerate}[label=(\greek*)]
    \item orientable
    \item unorientable and $|\pi_0(\partial C)|+\chi(C) \in 2\ZZ+1$
    \item unorientable and $|\pi_0(\partial C)|+\chi(C) \in 2\ZZ$
    \end{enumerate}

\noindent Let $S$ be a (not necessarily connected) compact surface satisfying the condition
\begin{equation}\label{eq:agcondition}\text{Every connected component }C\text{ of }S\text{ is of type } (\alpha) \text{ or }(\gamma)
\end{equation}
equipped with a labeling $\ell\colon \pi_0(S)\to A$ of the connected components by elements of $A$. (Note that $|\pi_0(\partial C)|+\chi(C) \in 2\ZZ$ also holds in case $(\alpha)$.) Then we define the \emph{abstract evaluation} of $S$ to be the element 

    \begin{align*}
        \eval_A(S):= \bigotimes_{C\in \pi_0(S)} \Delta^{\pi_0(\partial C)}(m(H^{g(C)},\ell(C))) \in \bigotimes_{C\in \pi_0(S)} A^{\otimes \pi_0(\partial C)} = A^{\otimes \pi_0(\partial S)}
    \end{align*}
    where $H=m\circ \Delta(1)\in A$ is the handle element of $A$, $g(C):= 1- (|\pi_0(\partial C)|+\chi(C))/2$ and $\ell(C)$ are the \emph{genus} and label of the component $C$, and 
    \begin{align*}
        \Delta^{\pi_0(\partial C)}\colon  A\to A^{\otimes \pi_0(\partial C)}
    \end{align*}
    is the (unique) map with these given domain and codomain, which is assembled from the coalgebra maps $\Delta,\varepsilon$ of $A$.
\end{con}

\begin{con}[Abstract evaluation, extended]\label{con:abstractEvalext}
Let $A$ and $S$ be as in Construction~\ref{con:abstractEval}. For a given finite set $P\subset \mathrm{int}(S)$, we denote by $S^\circ$ the compact $A$-labeled surface obtained by removing a small disk around every point of $P$ from $S$. To the pair $(S,P)$ we then associate the \emph{abstract evaluation}:
    \begin{align*}
        \eval_A(S,P):= \bigotimes_{C\in \pi_0(S)} \Delta^{\pi_0(\partial C)} \circ m(H^{g(C)},-) \circ m(\ell(C),-) \circ m^{P\cap C} \colon  A^{\otimes P} \to A^{\otimes \pi_0(\partial S)}
    \end{align*}
    where we use notation as in Construction~\ref{con:abstractEval} and denote by   \begin{align*}
        m^{P\cap C}\colon  A^{\otimes P}\to A
    \end{align*}
    the (unique) map with these given domain and codomain, which is assembled from the algebra maps $m,\eta$ of $A$. Note that $\eval_A(S,P)(1\otimes \cdots \otimes 1) = \eval_A(S)$.
\end{con}

\begin{lem} \label{lem:absevalskein}
Let $M$ be a 3-manifold and $c\subset \partial M$, such that every properly embedded surface $S$ in $M$ with boundary $\partial S=c$ satisfies condition \eqref{eq:agcondition}. 
Then the abstract evaluation of $A$-decorated surfaces from Construction~\ref{con:abstractEval}, upon $\kk$-linearization, factors through a $\kk$-linear map 
\begin{align}
\label{eq:abs}
\Sk_A(M,c) \to A^{\otimes \pi_0(c)}.
\end{align}
\end{lem}
\begin{proof}
    It is straightforward to check that the abstract evaluation respects the defining skein relations of $\Sk_A(M,c)$.
\end{proof}

\begin{exm}\label{exm:FirstSkeinModules}
	Let $A$ be a commutative Frobenius algebra over a commutative ring $\kk$. 
	\begin{enumerate}[(i)]
		\item Every closed $1$-manifold $c\subset S^2=\partial B^3$ bounds a configuration of embedded disks $D$ and all properly embedded surfaces in $B^3$ are orientable. The $\kk$-linear map
        \[
A^{\otimes\pi_0(c)} \xrightarrow{\cong} \Sk_A(B^3,c),
        \]
        defined by sending an elementary tensor to the class of the $A$-decorated surface represented by $D$ with decoration given by the tensor, is an isomorphism of $\kk$-modules with inverse given by abstract evaluation, i.e.\ \eqref{eq:abs} specialized to $M=B^3$. See also \cite[Corollary 5.3]{KaiserFrobAlg}.
		\item The skein module of the 3-sphere $\Sk_A(S^3)=\Sk_A(S^3,\emptyset)$ is free of rank one spanned by the empty ($A$-decorated) surface. To see this, note that the same is true for the skein module $\Sk_A(B^3,\emptyset)$ by the previous item, and this skein module maps surjectively to $\Sk_A(S^3,\emptyset)$ along the standard embedding. This map is also injective, with a left-inverse given by abstract evaluation \eqref{eq:abs}.
        \item Let $M$ be the result of removing a finite number of embedded 3-balls from the interior of $B^3$. Then $M$ satisfies the hypotheses of Lemma~\ref{lem:absevalskein} since any properly embedded surface $S$ in $M$ can plugged by disk configuration in the removed 3-balls without changing the parity of $|\pi_0(\partial C)|+\chi(C)$ for any connected component $C$. Since the parity has to be even in $B^3$, the same is true for $M$.
        \item The hypotheses of Lemma~\ref{lem:absevalskein} are also satisfied for $M=\#_k S^1\times S^2$ with $k\geq 1$. Cutting embedded surfaces $S$ along generic $S^2$ fibers produces surfaces embedded in 3-manifolds as in (iii), without changing the parity of $|\pi_0(\partial C)|+\chi(C)$ for any connected component $C$.
        \item Consider $M=S^1\times B^2$ with boundary condition $c=S^1\times P_{2n}$ for some set $P_{2n}\in \partial B^2$ with $|P_{2n}|=2n$ for $n\geq 0$. Then the pair $(M,c)$ satisfies the hypotheses of Lemma~\ref{lem:absevalskein}. To see this, note  that attaching a 3-dimensional 2-handle to $M$ along $*\times S^1$ yields $B^3$, during which core-parallel disks cap off any embedded surface $S$ with boundary $c$, yielding a closed surface in $B^3$ without changing the parity of $|\pi_0(\partial C)|+\chi(C)$ for any connected component $C$ of $S$.
	\end{enumerate}	
\end{exm}

\begin{rmk}
Example~\ref{exm:FirstSkeinModules}.(iii) implies that the skein modules of Definition~\ref{def:skeinmod} admit alternative definitions akin to skein lasagna module for 4-manifolds \cite{PaulInvariantsOf4Man}, see e.g.\ \cite[Section 2.4]{WedrichSurveyLinkHomTQFT}.
\end{rmk}

The surface skein theory associates $\kk$-linear categories to surfaces. For a surface $\Sigma$, the morphism $\kk$-modules are computed as skein modules associated with the thickened surface $\Sigma \times [0,1]$.
\begin{defn}[Skein category]\label{def:skeincat}
	Let $\Sigma$ be a compact surface. For every commutative Frobenius algebra $A$ over a commutative ring $\kk$, we define the skein category of $\Sigma$ as the $\kk$-linear category $\skcat_A(\Sigma)$ with
	\begin{itemize}
		\item \textbf{Objects:} closed embedded 1-manifolds $c\subset \Sigma$.
		\item \textbf{Morphisms:} The $\kk$-module of morphisms between objects $c$ and $d$ is the skein module \begin{align*}
		    \skcat_A(\Sigma)(c,d) := \Sk_A(\Sigma\times [0,1], c\times \{0\}\sqcup d\times\{1\}).
		\end{align*} 
		\item \textbf{Composition:} $\skcat_A(\Sigma)(d,e) \otimes \skcat_A(\Sigma)(c,d) \to \skcat_A(\Sigma)(c,e)$ is induced by the map given on elementary tensors by gluing $A$-decorated surfaces along their common boundary $d$, multiplying the decorations on merging components using the Frobenius algebra multiplication, and rescaling in the interval direction. 
	\end{itemize}
    The identity morphism on an object $c$ is given by the $A$-decorated surface $c\times [0,1]$ with decoration $1\in A$ on every component. The composition is unital for these identity morphisms and associative, as can be seen using isotopies reparameterizing the unit interval $[0,1]$.
\end{defn}

\begin{rmk}\label{rmk:nonclosedsurfaceskeincat} Although Definition~\ref{def:skeincat} works perfectly well for compact surfaces with boundary, we will mostly use it in the closed case. Taking boundary of $\Sigma$ into account, it would be more natural to model the skein category of $\Sigma$ as module category for a locally linear 2-category associated to (a designated submanifold of) $\partial \Sigma$--- in this setting objects are properly embedded 1-manifolds that may end on the boundary; see \cite{HRWBorderedInvariantsKh} for a related construction in a derived setting. When specifying designated intervals in $\partial \Sigma$, one arrives at an analog of the stated/internal skein algebras (with multiple gates) \cite{MR3847209, MR4493620, MR4437512, jordan2025finitenessholonomicityskeinmodules} for the Asaeda--Frohman--Kaiser TQFT $\sktft_A$.
\end{rmk}

For a 3-manifold $M$ with boundary $\partial M$, the objects of the skein category $\skcat_A(\partial M)$ are called \emph{boundary conditions} for the skein module of $M$. Let $c\subset \partial M$ be a boundary condition. Recall the following arguments from \cite[Section 7]{KaiserFrobAlg}. Consider the long exact sequence
\begin{equation}\label{eq:LES}
	\begin{tikzcd}
		{H_2(M;\ZZ/2)} & {H_2(M,c;\ZZ/2)} & {H_1(c; \ZZ/2)} & {H_1(M;\ZZ/2).}
		\arrow[from=1-1, to=1-2]
		\arrow["\partial", from=1-2, to=1-3]
		\arrow["\iota", from=1-3, to=1-4]
	\end{tikzcd}
\end{equation}
Let $[c]\in H_1(c;\ZZ/2)$ be the fundamental class. By exactness, we have
\begin{align*}
	\iota ([c]) =0 \iff \partial^{-1}[c] \neq \emptyset.
\end{align*} 
If $\partial^{-1}[c]=\emptyset$, there are no surfaces in $M$ bounding $c$ and $\Sk_A(M,c)=\{0\}$. 
\begin{prop}\label{prop:HomologDecompOfSkMod}
Let $A$ be a commutative Frobenius algebra over a commutative ring $\kk$. Let $M$ be a 3-manifold with boundary, and $c\subset \partial M$ a boundary condition. Let $\xi\in H_2(M,c;\ZZ/2)$. Write $\Sk_A(M,c;\xi)\subseteq \Sk_A(M,c)$ for the submodule spanned by classes of surfaces represented by $\xi$. Then,
\begin{align*}
	\Sk_A(M,c) = \bigoplus_{\xi\in \partial^{-1}[c]} \Sk_A(M,c;\xi).
\end{align*}
\end{prop}
\begin{proof}
	This is \cite[Proposition 7.3]{KaiserFrobAlg}.
\end{proof}
\begin{rmk}
	A similar decomposition using integral homology applies to an oriented version of surface skein theory.
\end{rmk}
We recall the following standard terminology for surfaces embedded in 3-manifolds.
\begin{defn}
	Let $M$ be a 3-manifold possibly with boundary and $S$ a surface that is properly embedded in $M$. 
	\begin{enumerate}[(i)]
		\item A simple closed curve on $S$ is called \emph{essential} if it does not bound a disk in $S$.
		\item A \emph{compression disk} for $S$ is a disk $D$ embedded in $M$ such that $S\cap D=\partial D$, and $\partial D$ is an essential curve on $S$.
		\item An embedded surface $S$ in $M$ is called \emph{compressible} if it has a compression disk or if there is a component of $S$ that is a 2-sphere bounding a 3-ball in $M$. Otherwise $S$ is called \emph{incompressible}.
		\item $M$ is called \emph{irreducible} if every 2-sphere in $M$ is the boundary of a 3-ball in $M$.
		\item Two incompressible surfaces $S'$ and $S''$ in $M$ are called \emph{parallel} if there is an embedding of a thickened surface $\Phi\colon S\times [0,1] \to M$ with $\Phi(S \times \{0\})= S'$ and $\Phi(S\times \{1\})= S''$. 
	\end{enumerate}
\end{defn}
Note that sometimes the notion of incompressible surfaces excludes spheres. Incompressible spheres are sometimes also called \emph{essential spheres}. We will use these interchangeably.
\begin{lem}\label{lem:piInjective}
	Let $M$ be an orientable 3-manifold. Consider the embedding of a closed connected orientable surface $\iota\colon S\to M$ that is not a sphere. Then $S$ is incompressible if and only if the induced map 
	\begin{align*}
		\iota_*\colon  \pi_1(S) \to \pi_1(M)
	\end{align*}
	is injective.
\end{lem}
\begin{proof}
    If $D$ is a compression disk for $S$ in $M$, then $\partial D$ defines a non-trivial element in the kernel of $\iota_*$. Hence, if $\iota_*$ is injective, $S$ is incompressible in $M$. The other implication is Kneser's Lemma \cite{KneserLemma, StallingsGTandThreeMan}.
\end{proof}
\begin{rmk}
	The orientability assumptions in Lemma \ref{lem:piInjective} are required for Kneser's Lemma. In particular, if an embedded surface $S$ is non-orientable, it need not induce an injective map on the fundamental groups. For example, lens spaces $L(p,q)$ have $\pi_1(L(p,q))\cong \ZZ/p$ and therefore no orientable incompressible surfaces. However, there exist non-orientable incompressible surfaces in $L(p,q)$ if $p=2k$ and $1\leq q\leq k$ \cite{BredonWood}.
\end{rmk}
\begin{lem}\label{lem:SurfInsimplyCon}
    Let $M$ be a simply connected 3-manifold. Then $M$ is orientable and every closed surface $S\hookrightarrow M$ embedded in $M$ is orientable. 
\end{lem}
\begin{proof}
    For the first statement, see for example \cite[Proposition 3.25]{HatcherAT}. The second statement goes back to \cite{SamelsonOrientableHypersurf}. For a modern account in the context of 3-manifolds, see for example \cite[Corollary 9.1.8]{MartelliGeomTop}.
\end{proof}
For more details on the theory of incompressible surfaces in 3-manifolds, we refer to \cite{MartelliGeomTop} and \cite{hatcher3manifolds}.
\begin{prop}\label{prop:incompGens}
	Let $A$ be a commutative Frobenius algebra over a commutative ring $\kk$ and fix a spanning set $B\subset A$. Let $M$ be a closed 3-manifold. The skein module $\Sk_A(M,\emptyset)$ is spanned by the set of $A$-decorated surfaces in $M$, whose underlying surface is incompressible and all decorations are taken from the set $B$.
\end{prop}
\begin{proof}
	This is \cite[Theorem 9.1]{KaiserFrobAlg}.
\end{proof}
Since we always allow the empty skein, we obtain the following consequence.
\begin{cor}\label{cor:noIncomp}
	Let $A$ be a commutative Frobenius algebra over a commutative ring $\kk$. Let $M$ be a closed 3-manifold that does not contain any incompressible surfaces, then we have a $\kk$-linear surjection
    \[
    \kk \twoheadrightarrow \Sk_A(M,\emptyset), \quad 1 \mapsto [\emptyset].
    \]
\end{cor}

\subsection{The surface skein TQFT}
We define cobordism bicategories which will be the source categories of the surface skein TQFT. We follow definitions in \cite{HaiounHandleTQFT}. See \cite[Section 3.1]{SchommerPriesThesis} for a more detailed account on cobordism bicategories. 
\begin{defn}
	Let $\Sigma_0$ and $\Sigma_1$ be closed oriented surfaces. 
	The \emph{relative cobordism category} $\cob^{\Sigma_0,\Sigma_1}_{3,4}$ consists of
	\begin{itemize}
		\item \textbf{Objects:} oriented 3-manifolds $M$ with boundary $\partial M \cong \overline{\Sigma_0}\cup\Sigma_1$ equipped with collars called \emph{horizontal collars}.
		\item \textbf{Morphisms:} 4-dimensional oriented cobordisms $W\colon M \to N$ with corners. 
		That is, diffeomorphism classes of oriented 4-manifolds $W$ with corners and orientation-preserving diffeomorphisms embedding $M$ and $N$ as boundaries respecting the horizontal collars of $\Sigma_0$ and $\Sigma_1$, i.e.\ 
		\begin{align*}
			\partial W\cong \overline{M} \cup_{\Sigma_0\sqcup \overline{\Sigma}_1} ((\overline{\Sigma}_0\sqcup \Sigma_1) \times [0,1]) \cup_{\overline{\Sigma}_0\sqcup \Sigma_1} N,
		\end{align*}
		and equipped with \emph{vertical collars} on $M$ and $N$. 
        \item \textbf{Composition:} Gluing along the vertical collars. This is associative by a diffeomorphism appropriately reparameterizing the thickened horizontal collars.
	\end{itemize}
\end{defn}
\begin{defn}
	The symmetric monoidal cobordism bicategory $\cob_{2,3,4}$ consists of
	\begin{itemize}
		\item \textbf{Objects:} closed oriented 2-manifolds.
		\item \textbf{Morphism categories:} Relative cobordism categories $\Hom_{\cob_{2,3,4}}(\Sigma_0,\Sigma_1) = \cob^{\Sigma_0,\Sigma_1}_{3,4}$.
		\item \textbf{Composition of 1-morphisms:} Gluing along horizontal collars
		\item \textbf{Horizontal composition of 2-morphisms:} Gluing along thickened horizontal collars.
		\item \textbf{Symmetric monoidal structure:} Disjoint union with unit $\emptyset$ and swapping components.
	\end{itemize}
\end{defn}
\begin{defn}
	The symmetric monoidal cobordism bicategory $\cob_{2,3,3+\varepsilon}$ consists of
	\begin{itemize}
		\item The same objects and 1-morphisms as $\cob_{2,3,4}$.
		\item 2-morphisms are isotopy classes of orientation-preserving diffeomorphisms fixing the horizontal collars.
	\end{itemize}
\end{defn}
Note that there is a functor of symmetric monoidal bicategories from $\cob_{2,3,3+\varepsilon}$ to $\cob_{2,3,4}$ which is the identity on objects and 1-morphisms. On 2-morphisms, it assigns to an orientation-preserving diffeomorphism $\phi\colon M\to M'$ the cobordism $W_\phi\colon M\to M'$ induced by $\phi$.

The target of the surface skein TQFT is a Morita bicategory of $\kk$-linear categories $\mor(\kcat)$. We follow definitions in \cite[Section 3.2]{HRWBorderedInvariantsKh} although in our version, we do not have a grading on the linear categories. Also compare to the Morita bicategory of associative unital algebras $\mathsf{Alg}_\kk$ \cite[Section 8.9]{GordonPowerStreet}.
\begin{defn}[Bimodules and tensor products]
	Let $\A$ and $\B$ be $\kk$-linear categories. An $(\A,\B)$-bimodule is a $\kk$-linear functor $\A\otimes \B^{\op}\to \kmod$. 
    Let $\mathcal{F}$ and $\mathcal{F}'$ be $(\A,\B)$-bimodules. A bimodule map is a linear natural transformation $\alpha\colon \mathcal{F} \Rightarrow\mathcal{F}'$.
	The collection of $(\A,\B)$-bimodules together with bimodule maps form a $\kk$-linear category denoted by $\bim_{\A,\B}$. The composition is given by composition of natural transformations. To emphasize the left- and right-action, we sometimes write ${}_{\A}M_{\B}$ for an $(\A,\B)$-bimodule. For bimodules ${}_{\A}M_{\B}$ and ${}_{\A'}N_{\B'}$, the \emph{(absolute) tensor product} ${}_{\A\otimes \A'}(M\otimes_\kk N)_{\B\otimes \B'}$ defines an $(\A\otimes \A', \B\otimes \B')$-bimodule in the usual way.
	
	Let $\A,\B,\C$ be $\kk$-linear categories and $M= {}_{\A}M_{\B}\in \bim_{\A,\B}$ and $N={}_{\B}N_{\C}\in \bim_{\B,\C}$ be bimodules. Define the \emph{relative tensor product} $({}_{\A} M_{\B})\otimes_\B ({}_{\B}N_{\C})$ as the $(\A,\C)$-bimodule given by the cokernel
	\begin{equation}\label{eq:coker}
		(M\otimes_\B N) (a,c) := \mathrm{coker} \left(\begin{aligned}
			\bigoplus_{b,b'\in \B} M(a,b')\otimes \Hom(b,b')\otimes N(b,c) &\longrightarrow \bigoplus_{b\in \B} M(a,b)\otimes N(b,c)\\ f\otimes g\otimes h\quad  &\longmapsto (fg)\otimes h - f\otimes (gh)
		\end{aligned}\right)
	\end{equation}
	that is, coequalizing the left- and right-action.
\end{defn}
\begin{defn}[Morita bicategory]
	Let $\kk$ be a commutative ring. The symmetric monoidal Morita bicategory $\mor(\kcat)$ of $\kk$-linear categories consist of
	\begin{itemize}
		\item \textbf{Objects:} $\kk$-linear categories
		\item \textbf{Morphism categories:} $\Hom_{\mor(\kcat)}(\B,\A)$ is the $\kk$-linear category $\bim_{\A,\B}$ of $(\A,\B)$-bimodules with bimodule maps
		\item \textbf{Horizontal composition:} relative tensor product 
		\begin{align*}
			\bim_{\A,\B}\otimes \bim_{\B,\C}&\to \bim_{\A,\C}\\
			M\otimes N &\mapsto M \otimes_\B N
		\end{align*}
	\end{itemize}
      The symmetric monoidal structure is given by the absolute tensor product $\otimes_\kk$ and the symmetric braiding by swapping tensor factors.
\end{defn}
\begin{con}[Surface skein TQFT]\label{con:theTQFT}
	Let $A$ be a commutative Frobenius algebra over $\kk$. The \emph{surface skein TQFT} or the \emph{Asaeda--Frohman--Kaiser TQFT} $\sktft_A$ is the symmetric monoidal functor of bicategories
	\begin{align*}
		\sktft_A\colon \cob_{2,3,3+\varepsilon}^{\mathrm{vop}} \to \mor(\kcat)
	\end{align*}
    constructed as follows:
	\begin{itemize}
		\item \textbf{Objects:} A closed oriented surface $\Sigma$ is mapped to the skein category $\sktft_A(\Sigma) := \skcat_A(\Sigma)$.
		\item \textbf{1-morphisms:} To a cobordism $M\colon \Sigma\leftarrow \Sigma'$ the TQFT $\sktft_A$ assigns the $(\skcat_A(\Sigma),\skcat_A(\Sigma'))$-bimodule 
		\begin{align*}
			\sktft_A(M) :={}_{\skcat_A(\Sigma)}\Sk_A(M)_{\skcat_A(\Sigma')}\colon \skcat_A(\Sigma) \otimes \skcat_A(\Sigma')^\op&\to \kmod\\
			(c, c')&\mapsto \Sk_A(M; c, c')
		\end{align*}
        Here we write $\Sk_A(M; c, c')$ for $\Sk_A(M, c\sqcup c')$ to emphasize that $c$ and $c'$ belong to the out-going and in-going boundary of $M$ respectively.
		\item \textbf{2-morphisms:} Let $\varphi$ be an orientation-preserving diffeomorphism from $\Sigma \xleftarrow{M} \Sigma'$ to $\Sigma \xleftarrow{N} \Sigma'$ fixing the collars on $\Sigma$ and $\Sigma'$. The functor assigns to $\varphi$ the bimodule map 
		\begin{align*}
			\sktft_A(\varphi) \colon {}_{\skcat_A(\Sigma)}\Sk_A(M)_{\skcat_A(\Sigma')} 
            {\leftarrow} 
            {}_{\skcat_A(\Sigma)}\Sk_A(N)_{\skcat_A(\Sigma')}
		\end{align*}
		induced by pushing forward $A$-decorated surfaces along $\varphi^{-1}$. Note that this reverses the vertical direction of 2-morphisms.
        \end{itemize}
        It is clear that the assignments up to here constitute functors $\sktft_A(\Sigma, \Sigma')$ on the level of $\Hom$-categories. It remains to establish the compatibility with the horizontal composition.
        
        \begin{itemize}

    \item \textbf{Horizontal composition 2-cells:} Let $M\colon \Sigma' \leftarrow \Sigma''$ and $N\colon \Sigma''\leftarrow \Sigma$ be cobordisms and $M\cup_{\Sigma''}N$ their composite. For any triple of objects $c'\in \skcat_A(\Sigma')$, $c''\in \skcat_A(\Sigma'')$ and $c\in \skcat_A(\Sigma)$ the gluing of $A$-decorated surfaces induces a $\kk$-linear map on skein modules:

    \[
\Sk_A(M; c', c'') \otimes_{\kk} \Sk_A(N; c'', c) \to 
\Sk_A(M\cup_{\Sigma''}N; c', c) 
    \]
   The direct sum over these maps for all possible $c''$ is surjective and following \cite[Theorem 4.4.2]{WalkerTQFT06} the kernel is given by the image of a map as in \eqref{eq:coker}. Noting the functoriality of the construction in the arguments, $c',c$, we thus obtain an isomorphism of bimodules:
       \[
       \psi_A(M,N)\colon
{}_{\skcat_A(\Sigma')}\Sk_A(M) \otimes_{\skcat_A(\Sigma'')} \Sk_A(N)_{\skcat_A(\Sigma)} 
\xrightarrow{\cong}
{}_{\skcat_A(\Sigma')}\Sk_A(M\cup_{\Sigma''}N)_{\skcat_A(\Sigma)} 
    \]
   The isomorphisms depend naturally on $M$ and $N$. We will not name these isomorphisms in the following.   
   \item \textbf{Unit 2-cells:} As last bit of data required for a functor of bicategories we need to prescribe for every surface $\Sigma$ an invertible 2-cell
   \[
   \id_{\sktft_A(\Sigma)} = 
   {}_{\skcat_A(\Sigma)}\skcat_A(\Sigma)_{\skcat_A(\Sigma)}
   \xrightarrow{\cong}
   {}_{\skcat_A(\Sigma)}\Sk_A(\Sigma \times [0,1])_{\skcat_A(\Sigma)}
   = 
   \sktft_A(\Sigma, \Sigma)(\id_\Sigma).
   \]
Indeed, there is a tautological choice, given by the identity map on $A$-decorated surfaces upon inspecting the definition of $\skcat_A(\Sigma)$.

	\end{itemize}
    To show that $\sktft_A$ is a functor of bicategories, it still remains to verify compatibility with the unitors and associators. This amounts to checking that the corresponding coherence diagrams commute. We omit the details here. Symmetric monoidality follows because the constructions of surface skein modules and categories are manifestly monoidal under disjoint union. 
\end{con}
\begin{rmk}\label{rmk:gluingMorita}
    Let $A$ be a commutative Frobenius algebra over a commutative ring $\kk$ and let $\Sigma$ be a closed oriented surface. Write $\C=\skcat_A(\Sigma)$ and consider a full subcategory $\C'\hookrightarrow \C$ such that the induced functor $\Mat(\C')\to \Mat(\C)$ on the additive completion is an equivalence. Such $\C'$ are called \emph{Morita equivalent subcategories}. Then the bimodules 
    \begin{align*}
        {}_{\C}\C_{\C'} \quad \text{and} \quad {}_{\C'}\C_{\C}
    \end{align*}
    are inverse to each other. Indeed, using that $\C'\hookrightarrow \C$ is full, we obtain
    \begin{align*}
        {}_{\C'}\C_{\C} \otimes_\C {}_{\C}\C_{\C'} = {}_{\C'}\C_{\C'} = {}_{\C'}\C'_{\C'}.
    \end{align*}
    From the equivalence $\Mat(\C')\to \Mat(\C)$, it can be shown that 
    ${}_{\C}\C_{\C'} \otimes_{\C'} {}_{\C'}\C_{\C} ={}_{\C}\C_{\C}$. As a result, when gluing, we can replace $\C=\skcat_A(\Sigma)$ by a Morita equivalent subcategory $\C'$. In particular, for commutative Frobenius algebras $A$ that are free of rank $r$ over $\kk$, we have the following consequence. Let $\C'\subset \skcat_A(\Sigma)$ be the full subcategory consisting of those objects $c\in \skcat_A(\Sigma)$ that do not contain inessential circle components. Then a generalized \emph{delooping}\footnote{This is the standard terminology in link homology literature originating from \cite[Lemma 4.1]{BarNFastComp}.} argument shows that $\Mat(\C')\to \Mat(\skcat_A(\Sigma))$ is an equivalence. Choose $x_i, y_i\in A$ for $i=1,\dots, r$ such that $\Delta(1)=\sum_i x_i\otimes y_i$. An inessential circle is isomorphic to the $r$-fold direct sum of the vacuum via the map that has as $i$-th component a cup (resp.\ cap) decorated with $x_i$ (resp.\ $y_i$). One composite is the identity on the inessential circle by the neck-cutting relation and the other composite is the identity matrix of size $r$ by $\varepsilon(m(x_i,y_j))=\delta_{i,j}$.
\end{rmk}

\begin{rmk}
	Our construction is related to a version of the Asaeda--Frohman--Kaiser TQFT $\mathbf{Z}_A$ that is fully extended down and also takes a commutative Frobenius algebra $A$ as input. We expect $\mathbf{Z}_A$ to be a symmetric monoidal functor of symmetric monoidal 4-categories
	\begin{align*}
		\mathbf{Z}_A \colon \cob_{0,1,2,3,3+\varepsilon} \to \mor(\mathsf{2Cat}^{\text{loc.$\kk$-lin.}}_{\boxtimes})
	\end{align*}
	where $\cob_{0,1,2,3,3+\varepsilon}$ is an appropriate symmetric monoidal 4-category with finite sets of points as objects, $k$-bordisms with corners as $k$-morphisms for $k=1,2,3$, and isotopy classes of diffeomorphisms as 4-morphisms. The target $\mor(\mathsf{2Cat}^{\text{loc.$\kk$-lin.}}_{\boxtimes})$ is, conjecturally, an appropriate symmetric monoidal Morita 4-category of locally $\kk$-linear monoidal 2-categories in the sense of Johnson-Freyd--Scheimbauer \cite{JFScheimbauerMorita}. The point is sent to the locally $\kk$-linear monoidal Bar-Natan 2-category $\mathbf{BN}_A$
	\begin{align*}
		\mathbf{Z}_A(\mathrm{pt}) = \mathbf{BN}_A \in \mor(\mathsf{2Cat}^{\text{loc.$\kk$-lin.}}_{\boxtimes})
	\end{align*} 
	which is constructed analogously to the usual locally linear monoidal 2-category in the construction of Khovanov homology for $A_{\mathrm{BN}}$ categorifying the Temperley--Lieb category. One can show that $\mathbf{BN}_A$ has duals and adjoints, presumably providing the data for $\mathbf{BN}_A$ being a 3-dualizable object.
	For closed oriented 2-manifolds and 3-dimensional cobordisms between them, the assignments of $\sktft_A$ and $\mathbf{Z}_A$ agree. The fully extended down version $\mathbf{Z}_A$ has built in cutting and gluing of $k$-bordisms with corners for all $1\leq k\leq 3$. The extendability results for $\sktft_A$ should carry over to $\mathbf{Z}_A$.
\end{rmk}
\begin{rmk}
    In \cite{WalkerTQFT06} and \cite{WalkerUnivState}, $(n+\varepsilon)$-dimensional TQFTs are defined from more general skein theory. Note that these TQFTs assign to closed $n$-dimensional manifolds the dual of the skein module.

    Also note the vertical contravariance built into the definition of $\sktft_A$. The extension to 4d in Construction~\ref{con:extension} below, following \cite[Theorem 4.1.2]{WalkerUnivState}, reverses the direction of 2-morphisms. Using the invertibility of diffeomorphisms, we compatibly reversed the direction of 2-morphisms in Construction \ref{con:theTQFT}.
\end{rmk}

\section{Extension to 4-dimensional handles}\label{sec:extension}
\subsection{General construction}
Let $W_k\colon M \to N$ be a 4-dimensional oriented cobordism between oriented 3-manifolds, consisting of a single index $k$ handle. Recall that a 4-dimensional $k$-handle $H_k= B^k\times B^{4-k}$ has 
\emph{attaching region} $\alpha_k:= S^{k-1}\times B^{4-k}$ and \emph{belt region} $\beta_k := B^k\times S^{3-k}$. Those share the common boundary $\delta_k:=S^{k-1}\times S^{3-k}$. 
 We view a $k$-handle as a cobordism with corners
\begin{align*}
	H_k\colon \alpha_k \to \beta_k
\end{align*}
and can write 
\[
M = M' \cup_{\delta_k} \alpha_k
, \quad 
N = M' \cup_{\delta_k} \beta_k
, \quad 
W_k = (M\times [0,1]) \cup_{\alpha_k \times \{1\}} H_k.
\]
We call $W_k$ a \emph{handle cobordism}. To extend $\sktft_A$ to 4d $k$-handles, it needs to assign to $H_k$ a bimodule map of skein module functors
\begin{align*}
	m_k = \sktft_A(H_k)\colon {}_{\skcat_A(\delta_k)}\Sk_A(\beta_k)\to {}_{\skcat_A(\delta_k)}\Sk_A(\alpha_k).
\end{align*}
Note that the surface skein TQFT reverses the order of 2-morphisms.
Then we can define maps on skein modules by choosing $\sktft_A(W_k)$ to be the unique map that makes the following diagram commute.
\[\begin{tikzcd}
            \sktft_A(N) &
            \sktft_A(M' \cup_{\delta_k} \beta_k)& [1.5em]
            \Sk_A(M') \otimes_{\skcat_A(\delta_k)} \Sk_A(\beta_k)
            \\
            \sktft_A(M) &
            \sktft_A(M' \cup_{\delta_k} \alpha_k)&[1.5em]
            \Sk_A(M') \otimes_{\skcat_A(\delta_k)} \Sk_A(\alpha_k)
            \arrow["\sktft_A(W_k)",from=1-1, to=2-1]
            \arrow["\id \otimes\sktft_A(H_k)",from=1-3, to=2-3]
         	\arrow["{\sktft_A(M',\beta_k)}"',from=1-3, to=1-2]
        	\arrow["{\sktft_A(M',\alpha_k)}"', from=2-3, to=2-2]
            \arrow[equal, from=1-1, to=1-2]
            \arrow[equal, from=2-1, to=2-2]
\end{tikzcd}\]

\begin{defn}
	Let $k\in \{0,1,\dots, 4\}$. A \emph{$(4,k)$-handlebody} or a  \emph{$4$-dimensional $k$-handlebody} is an oriented 4-manifold presented by means of a handle decomposition with all handles of index $\leq k$. Two $(4,k)$-handlebodies are \emph{$j$-equivalent} if their handle decompositions are related by handle slides and handle cancellations involving only handles of index $\leq j$. 
\end{defn}
Note that $j$-equivalence is an at least as strong notion of equivalence as diffeomorphism. 

\begin{rmk}\label{rmk:AC}
    It is an open question whether diffeomorphic 4-dimensional 2-handlebodies are necessarily 2-equivalent. Restricted to 4-dimensional 2-handlebodies diffeomorphic to $B^4$, this question is analogous to the \emph{Andrews--Curtis Conjecture} \cite{AndrewsCurtis}:  that any balanced\footnote{A presentation of a group is \emph{balanced} if the number of generators is equal to the number of relations.} presentations of the trivial group can be transformed into the empty presentation via Nielsen transformations---the analogues of handle moves of 4-dimensional handles of index $\leq 2$. These conjectures are still open, although widely believed to be false with proposed counterexamples e.g.\ in \cite{GompfAK4Sphere}. Invariants sensitive to 4-dimensional 2-handlebodies such as TQFTs defined for 4-dimensional handles of index $\leq 2$, but not for handles of index 3 could provide tools to detect these counterexamples. For more details on this, and another construction of invariants of 4-dimensional 2-handlebodies, we refer to \cite{BeliakovaDeRenzi4d2h}.
\end{rmk}

The notion of $j$-equivalence also makes sense for 4-dimensional oriented cobordisms with corners built from handles of index $\leq j$ (and diffeomorphisms).

\begin{defn}
    Let $k\in \{0,1,\dots, 4\}$. We define $\cob_{2,3,3+k/4}$ to be the symmetric monoidal bicategory with the same objects and 1-morphisms as $\cob_{2,3,4}$. The 2-morphisms are 4-dimensional oriented cobordisms with corners that are compositions of handle cobordisms $W_l$ for handles of index $l\leq k$ and of cobordisms $W_\varphi$ induced by orientation-preserving diffeomorphisms $\varphi$. These are considered up to $k$-equivalence. 
\end{defn}

For $0\leq k\leq 3$, there is an obvious symmetric monoidal functor of bicategories $\cob_{2,3,3+k/4}\to \cob_{2,3,3+(k+1)/4}$. For $k=4$, we have an equivalence $\cob_{2,3,3+4/4}\to \cob_{2,3,4}$ since every 4-dimensional oriented cobordism admits a handle decomposition which is unique up to $4$-equivalence. The symmetric monoidal functor of bicategories $\cob_{2,3,3+\varepsilon}\to \cob_{2,3,4}$ assigning to every orientation-preserving diffeomorphism $\varphi\colon M\to N$ the cobordism $W_\varphi$ factors through each of the $\cob_{2,3,3+k/4}$ via the functors above. 

Let $A$ be a commutative Frobenius algebra over a commutative ring $\kk$ and let $k\in \{0,1,\dots,4\}$. Consider the surface skein theory 
\begin{align*}
    \sktft_A\colon \cob_{2,3,3+\varepsilon}^{\mathrm{vop}} \to \mor(\kcat)
\end{align*} 
associated to $A$. In the following, we construct maps assigned to handle cobordisms by defining the handle maps $m_l=\sktft_A(H_l)$ for handles of index $l\leq k$. By gluing them via the handle cobordisms, this extends $\sktft_A$ to a symmetric monoidal functor of bicategories
\begin{align*}
    \cob_{2,3,3+k/4}^{\mathrm{vop}}\to \mor(\kcat).
\end{align*}
In particular, this provides an invariant of $(4,k)$-handlebodies. The construction below follows the arguments in the proof of Walker's universal state sum construction \cite[Theorem 4.1.2]{WalkerUnivState}. See also \cite{WalkerTQFT06}. Note, however, that we work with slightly different assumptions. We will comment on this difference below.

\begin{con}[Extension to $k$-handles]\label{con:extension}
	Let $A$ be a commutative Frobenius algebra over a commutative ring $\kk$ such that for every $k\in \{0,1,\dots,4\}$ and every boundary condition $c\subset \delta_k$, the skein module $\Sk_A(\alpha_k,c)$ of the attaching region of the $k$-handle is free over $\kk$. Consider the surface skein TQFT $\sktft_A$. The construction proceeds iteratively.
    \smallskip

    \noindent\textbf{Base case.} The universal state sum construction for handles of index $k=0$ requires a choice of initial data. In the case of interest, the 0-handle data is a bimodule map 
	\begin{align*}
		m_0=\sktft_A(H_0)\colon {}_{\skcat_A(\emptyset)}\Sk_A(S^3)_{\skcat_A(\emptyset)} \to {}_{\skcat_A(\emptyset)}\Sk_A(\emptyset)_{\skcat_A(\emptyset)}.
	\end{align*}
     Recall that by Example \ref{exm:FirstSkeinModules}, we have
		$\Sk_A(S^3) \cong \kk\{\emptyset\}$ and $\Sk_A(\emptyset) = \kk$, each with the obvious bimodule structure\footnote{In the following, we will often suppress such canonical actions of $\skcat_A(\emptyset)$ from the notation.} of $\skcat_A(\emptyset)$, which has a single object with endomorphisms $\kk$. As a consequence, the choice of $m_0$ is unique up to the scalar given by the evaluation of the empty skein
    $\emptyset \mapsto \ev(\emptyset)\in \kk$.
	In the following, we assume that $\ev(\emptyset)\in \kk^\times$. Otherwise the pairings $p_k$ cannot be perfect. There are no conditions on handle slides and cancellation for 0-handles.

    \smallskip

    \noindent\textbf{Inductive step.}
	Now, we assume that we have already obtained the $k$-handle map $m_k$ for $k\geq 0$. To construct $m_{k+1}$, we will proceed in the following steps.
	\begin{enumerate}[(1)]
		\item For every boundary condition $c\subset \delta_{k+1}$, we construct a pairing $p_{k+1}$ on the skein module $\Sk_A(\alpha_{k+1}, c)$ using the previously obtained maps $m_k$ and $m_0$.
		\item If the pairing $p_{k+1}$ is perfect, we obtain an associated copairing $c_{k+1}$.
		\item We construct the $(k+1)$-handle map $m_{k+1}$ from the data of the copairing $c_{k+1}$ and the map $m_0$.
	\end{enumerate}

	\noindent\textbf{(1) Constructing the pairings.} Let  $c\subset \partial\alpha_{k+1} = \delta_{k+1}$ be a boundary condition. Given the handle maps $m_k$ and $m_0$, we define the pairing 
	\begin{align*}
    p_{k+1}^{(c)}\colon\Sk_A(\alpha_{k+1},c)\otimes \Sk_A(\alpha_{k+1},c)\to \kk
	\end{align*}
	as the composite
	\begin{align*}
		\Sk_A(\alpha_{k+1},c)\otimes \Sk_A(\alpha_{k+1},c) =&\  \Sk_A(S^k\times B^{3-k},c)\otimes\Sk_A(S^k\times B^{3-k},c)\\
		\xrightarrow{\text{glue}}&\ \Sk_A(S^k\times S^{3-k})\\
		=&\ \Sk_A( B^k\times S^{3-k} \cup_{S^{k-1}\times S^{3-k}} B^k\times S^{3-k})\\
		\cong&\ \Sk_A(B^k\times S^{3-k})\otimes_{\skcat_A(S^{k-1}\times S^{3-k})} \Sk_A(B^k\times S^{3-k}) \\
        \xrightarrow{\id \otimes m_k}&\ \Sk_A(B^k\times S^{3-k})\otimes_{\skcat_A(S^{k-1}\times S^{3-k})} \Sk_A(S^{k-1}\times B^{4-k}) \\
        \cong &\ \Sk_A( B^k\times S^{3-k} \cup_{S^{k-1}\times S^{3-k}} S^{k-1}\times B^{4-k})\\
        = &\ \Sk_A( S^3)\\
        \xrightarrow{m_0}&\ \kk.
	\end{align*} 
    Here, we used that $S^k\times S^{3-k}$, obtained from gluing the attaching regions, is the boundary of the 4-manifold $S^k\times B^{4-k}$ which is built from one $0$-handle and one $k$-handle.

\smallskip
	\noindent\textbf{(2) Obtaining the copairing.} If the pairing $p_{k+1}^{(c)}$ for a boundary condition $c$ is perfect, there exists an associated dual copairing 
    \begin{align*}
        \kk &\to \Sk_A(\alpha_{k+1},c)\otimes \Sk_A(\alpha_{k+1},c)\\
        1 &\mapsto \sum_i x^{(c)}_i\otimes y^{(c)}_i
    \end{align*}
    where $(x^{(c)}_i)_{i}$ and $(y^{(c)}_i)_{i}$ are dual bases in the sense that $p_{k+1}^{(c)}(y^{(c)}_i \otimes x^{(c)}_j)=\delta_{i,j}$.

 Suppose that $p_{k+1}^{(c)}$ is perfect for all boundary conditions $c\subset\delta_{k+1}$. Then we claim that the associated copairings assemble into a bimodule map
	\begin{align*}
		c_{k+1}\colon {}_{\skcat_A(\delta_{k+1})}\Sk_A(I\times \delta_{k+1})_{\skcat_A(\delta_{k+1})}
        &\to 
        {}_{\skcat_A(\delta_{k+1})}\Sk_A(\alpha_{k+1})\otimes \Sk_A(\alpha_{k+1})_{\skcat_A(\delta_{k+1})}
        \\
        \id_c 
        &\mapsto 
        \sum_i x^{(c)}_i\otimes y^{(c)}_i
	\end{align*}
    \begin{proof}[Proof that $c_{k+1}$ is a bimodule map] Note that a bimodule map from the regular bimodule for $\skcat_A(\delta_{k+1})$ is determined by its values on identity morphisms $\id_c$ on all objects $c\subset \delta_{k+1}$. To show that $c_{k+1}$ defines a bimodule map, we need to verify for every morphism $f\colon c \xleftarrow{} d$ in $\skcat_A(\delta_{k+1})$:
    \begin{equation}\label{eq:bimodule}
    \sum_i x^{(c)}_i\otimes (y^{(c)}_i\circ f) = \sum_j (f\circ x^{(d)}_j)\otimes y^{(d)}_j \text{ in } \Sk_A(\alpha_{k+1},c)\otimes \Sk_A(\alpha_{k+1},d)
    \end{equation}
    Note that every element of $\Sk_A(\alpha_{k+1},c)\otimes \Sk_A(\alpha_{k+1},d)$ can be expanded in the basis $(x_i^{(c)}\otimes y_j^{(d)})$ with coefficient extracted by the linear functional
    \begin{align*}
    \Sk_A(\alpha_{k+1},c)\otimes \Sk_A(\alpha_{k+1},d) &\to \kk\\
    x\otimes y &\mapsto p_{k+1}^{(c)}(y_i^{(c)}\otimes x)\cdot p_{k+1}^{(d)}(y \otimes x_j^{(d)}) 
    \end{align*}
    Applying this functional to both sides of the desired equation \eqref{eq:bimodule}, we obtain $p_{k+1}^{(d)}((y^{(c)}_i\circ f) \otimes x_j^{(d)})$ and $p_{k+1}^{(c)}(y^{(c)}_i \otimes (f\circ x_j^{(d)}))$. These coefficients agree by construction of the pairings, which proves \eqref{eq:bimodule}.
    \end{proof}

    \smallskip
	\noindent\textbf{(3) Constructing the handle map.}
	We can now construct the map for the $(k+1)$-handle 
    \begin{align*}
        m_{k+1}\colon {}_{\skcat_A(\delta_{k+1})}\Sk_A(\beta_{k+1})\to {}_{\skcat_A(\delta_{k+1})}\Sk_A(\alpha_{k+1})
    \end{align*}
    as the composition
	\begin{align*}
		{}_{\skcat_A(\delta_{k+1})}\Sk_A(\beta_{k+1})=&\ {}_{\skcat_A(\delta_{k+1})}\Sk_A(B^{k+1}\times S^{2-k})\\ 
        =&\ {}_{\skcat_A(\delta_{k+1})}\Sk_A(I\times S^{k}\times S^{2-k}  \cup_{\{1\}\times S^{k}\times S^{2-k}} \{1\}\times B^{k+1}\times S^{2-k})\\
		\cong&\ {}_{\skcat_A(\delta_{k+1})}\Sk_A(I\times S^{k}\times S^{2-k})\otimes_{\skcat_A(S^{k}\times S^{2-k})} \Sk_A(B^{k+1}\times S^{2-k})\\
		\xrightarrow{c_{k+1}\otimes \id}&\ {}_{\skcat_A(\delta_{k+1})} \Sk_A(\alpha_{k+1})\otimes_{\kk}\Sk_A(\alpha_{k+1}) \otimes_{\skcat_A(\delta_{k+1})}\Sk_A(\beta_{k+1})\\
		\cong&\ {}_{\skcat_A(\delta_{k+1})}\Sk_A(\alpha_{k+1})\otimes_{\kk} \Sk_A(S^{3})\\
		\xrightarrow{\id\otimes m_0}&\ {}_{\skcat_A(\delta_{k+1})}\Sk_A(\alpha_{k+1})\otimes_{\kk} \kk \cong {}_{\skcat_A(\delta_{k+1})}\Sk_A(\alpha_{k+1}).
	\end{align*}
	Note that if $k+1=4$, the belt region is empty.

\end{con}
For later convenience, we have listed the relevant regions and boundaries of $k$-handles in Table \ref{tab:khandle}.
\begin{table}[h!]\label{tab:khandle}
	\begin{center}
		\caption{$4$-dimensional $k$-handles and their boundary regions and corners}
		\begin{tabular}{c|c|c|c|c}
			$k$ & $k$-handle $H_k$ & attaching region $\alpha_k$ & common boundary $\delta_k$ & belt region $\beta_k$ \\
			\hline
			0 & $B^0\times B^{4}$ & $\emptyset$ & $\emptyset$ & $S^3$ \\
			1 &  $B^1\times B^{3}$ & $S^0\times B^3$ & $S^0\times S^2$ & $B^1\times S^2$ \\
			2 &  $B^2\times B^{2}$ & $S^1\times B^2$ & $S^1\times S^1$ & $B^2\times S^1$ \\
			3 &  $B^3\times B^{1}$ & $S^2\times B^1$ & $S^2\times S^0$ & $B^3\times S^0$ \\
			4& $B^4\times B^0$ & $S^3$ & $\emptyset$ & $\emptyset$ 
		\end{tabular}
	\end{center}
\end{table}

Recall from \cite[Section 4.2]{WalkerUnivState} that the inductive construction has invariance under handle slides and handle cancellation built in: Invariance under handle cancellation comes from invariance under isotopy, and invariance under handle slides from associativity of gluing \cite[Lemma 4.2.5]{WalkerUnivState}. For us, \cite[Theorem 4.1.2]{WalkerUnivState} does not immediately apply because of several differences in the assumptions. The first difference is that here, we work over a ground ring $\kk$ rather than a field. The right notion for us is therefore the perfectness of the pairing rather than non-degeneracy. Also note the freeness assumptions in Construction \ref{con:extension}. The second and perhaps more important difference is that in \cite{WalkerUnivState}, pairings and handle maps for handles of all indices are constructed. For this, the finite-dimensionality of the skein modules of the attaching regions and semisimplicity of the skein categories are required. Since we induce up to handles of a fixed index $k$, no requirements for higher handles are necessary. Furthermore, finite-dimensionality is only assumed to guarantee the perfectness of the pairings, but we check this by hand for each handle. The semisimplicity assumption is used to expand morphisms in a basis of minimal idempotents. This is again used in the proof of perfectness of the pairings. However, it is also used in \cite[Lemma 4.2.5]{WalkerUnivState} required for proof of invariance under handle cancellations. It turns out, this holds in greater generality and follows from the construction using the pairing and copairing. These modifications to the approach of \cite{WalkerUnivState} result in the following theorem below. Also compare an $(\infty,2)$-categorical version in \cite[Proposition 3.4.19]{LurieCobHyp} and future work by Reutter--Walker \cite{ReutterWalker} announced in \cite{WalkerUnivState}.

\begin{thm}\label{thm:WalkerExtension}
    Let $A$ be a commutative Frobenius algebra over a commutative ring $\kk$. Consider the associated surface skein TQFT
    \begin{align*}
        \sktft_A\colon \cob_{2,3,3+\varepsilon}^{\mathrm{vop}} \to \mor(\kcat).
    \end{align*}
    Fix an evaluation $\ev(\emptyset)\in \kk^\times$ of the empty skein in $\Sk_A(S^3)$ determining the 0-handle map
    \begin{align*}
        m_0 \colon {}_{\skcat_A(\emptyset)}\Sk_A(S^3)\to {}_{\skcat_A(\emptyset)}\Sk_A(\emptyset).
    \end{align*}
    Let $k\leq 4$ be the largest index, for which the skein modules $\Sk_A(\alpha_l,c)$ with $l\leq k$ are free over $\kk$, and all pairings $p_l$ are perfect. Then, Construction \ref{con:extension} provides a well-defined and unique extension to a symmetric monoidal functor of bicategories
    \begin{align*}
        \sktft_A\colon \cob_{2,3,3+k/4}^{\mathrm{vop}} \to \mor(\kcat)
    \end{align*}
    satisfying $\sktft_A(H_l)=m_l$ for $l\leq k$. 
\end{thm}
\begin{proof}
    By the above discussion, we only have to show that Construction \ref{con:extension} is invariant under handle cancellation. In particular, for $1\leq l\leq k$ and any pair of 4-dimensional handles of index $l-1$ and $l$ that are in canceling configuration, we will show that the composite of the handle maps $m_l$ and $m_{l-1}$ yields a bimodule map
    \begin{align*}
        \varphi\colon {}_{\skcat_A(S^2)}\Sk_A(B^3) \xrightarrow{\text{canceling }(l-1,l)\text{-pair}} {}_{\skcat_A(S^2)}\Sk_A(B^3)
    \end{align*}
    which we prove to be the identity. To this end, we use the non-degeneracy of the pairing on $\Sk_A(B^3)$ and the commuting diagram
    \begin{center}
        \begin{tikzcd}
    	   \Sk_A(B^3) \otimes_{\skcat_A(S^2)} \Sk_A(B^3) & \Sk_A(B^3) \otimes_{\skcat_A(S^2)} \Sk_A(B^3) \\
    	    \Sk_A(S^3)  & \Sk_A(S^3) & \kk
    	\arrow["\id\otimes \varphi", from=1-1, to=1-2]
    	\arrow["\cong",from=1-1, to=2-1]
    	\arrow["\cong", from=1-2, to=2-2]
    	\arrow["\Phi", from=2-1, to=2-2]
        \arrow["m_0", from=2-2, to=2-3]
        \end{tikzcd}
    \end{center}
    to argue that $\phi$ is the identity if and only $\Phi$, the induced map on $\Sk_A(S^3)$, is the identity. Equivalently, this is the case if and only if $m_0\circ \Phi=m_0$ (since $m_0$ is an isomorphism), and this we will show below.
    
    First we recall the construction of the pairing $p_l^{(c)}$ using the bimodule maps $m_{l-1}$ and $m_0$, and the construction of $m_l$ using $c_l$ and $m_0$. We write $p_l$ for the induced (bimodule) map
    \begin{align*}
        p_l\colon \Sk_A(\alpha_l)\otimes_{\skcat_A(\delta_l)} \Sk_A(\alpha_l) \to \kk.
    \end{align*} and record the zig-zag relation for the pairing and copairing maps:
    \begin{equation}\label{eq:stringstraight}
        (p_l\otimes \id)\circ (\id\otimes c_l) = \id\quad \text{as bimodule maps}\quad \Sk_A(\alpha_l)_{\skcat_A(\delta_l)}\to \Sk_A(\alpha_l)_{\skcat_A(\delta_l)}.
    \end{equation}
    The 3-sphere decomposes as
    \begin{align*}
        S^3 = S^{l-1}\times B^{4-l} \cup_{S^{l-1}\times S^{3-l}} B^l\times S^{3-l} = \alpha_l \cup_{\delta_l} \beta_l.
    \end{align*}
    To compute $m_0 \circ \Phi$, we first apply $m_l$ to $\beta_l$, and then follow up with $m_{l-1}$ and finally $m_0$:
    \begin{align*}
        \Sk_A(S^3)
         =&\ \Sk_A(S^{l-1}\times B^{4-l} \cup_{\delta_l} B^{l}\times S^{3-l}) \\
         =&\ \Sk_A(S^{l-1}\times B^{4-l} \cup_{\delta_l} I\times S^{l-1}\times S^{3-l}  \cup_{\delta_l} B^{l}\times S^{3-l})\\
         \cong &\ \Sk_A(S^{l-1}\times B^{4-l}) \otimes_{\skcat_A(\delta_l)} \Sk_A(I\times S^{l-1}\times S^{3-l})  \otimes_{\skcat_A(\delta_l)} \Sk_A(B^{l}\times S^{3-l})\\
         \xrightarrow{\id\otimes c_{l}\otimes \id}&\ \Sk_A(S^{l-1}\times B^{4-l}) \otimes_{\skcat_A(\delta_l)} \Sk_A(S^{l-1}\times B^{4-l}) \otimes_\kk \Sk_A(S^{l-1}\times B^{4-l})\otimes_{\skcat_A(\delta_l)} \Sk_A(B^{l}\times S^{3-l})\\
         \cong &\ \Sk_A(S^{l-1}\times B^{4-l}) \otimes_{\skcat_A(\delta_l)} \Sk_A(S^{l-1}\times B^{4-l}) \otimes_\kk \Sk_A(S^3)\\
         \xrightarrow{\id\otimes \id\otimes m_0} &\ \Sk_A(S^{l-1}\times B^{4-l}) \otimes_{\skcat_A(S^{l-1}\times S^{3-l})} \Sk_A(S^{l-1}\times B^{4-l})\otimes_\kk \kk\\
         \cong &\ \Sk_A(S^{l-1}\times S^{4-l})\\
         \cong &\ \Sk_A(B^{l-1}\times S^{4-l})\otimes_{\skcat_A(S^{l-2}\times S^{4-l})} \Sk_A(B^{l-1}\times S^{4-l})\\
         \xrightarrow{\id \otimes m_{l-1}} &\ \Sk_A(B^{l-1}\times S^{4-l})\otimes_{\skcat_A(S^{l-2}\times S^{4-l})} \Sk_A(S^{l-2}\times B^{5-l})\\
         \cong &\ \Sk_A(S^3)\\
         \xrightarrow{m_0}& \kk
    \end{align*}
    The last three lines compose to the pairing $p_l$ and we note that this makes sense for $l=1$ if we regard spheres and balls with negative exponents as empty. 
    By far-commutativity, we can move $m_0\colon \Sk_A(S^3)\to \kk$ to the end of the composition, so that \eqref{eq:stringstraight} can be applied to cancel $(\id \otimes c_l)$ against $(p_l\otimes \id)$. The only term remaining in this computation of $m_0\circ \Phi$ is $m_0$, which completes the proof.
\end{proof}

Let $W$ be a $(4,k)$-handlebody viewed as a cobordism $W\colon \emptyset \to \partial W$. Then $\partial W$ is a closed oriented 3-manifold with skein module $\sktft_A(\partial W) =\Sk_A(\partial W)$ and the surface skein theory $\sktft_A$ assigns to $W$ the linear functional
	\begin{align*}
		\sktft_A(W)\colon \Sk_A(\partial W)\to \kk
	\end{align*}
by gluing the handle maps $m_l=\sktft_A(H_l)$ for handles of index $l\leq k$ via handle cobordisms.

\begin{prop}\label{prop:evEulerChar}
	Let $A$ be a commutative Frobenius algebra over a commutative ring $\kk$ and consider $\sktft_A$. Assume that there exists an extension of $\sktft_A$ to $(4,k)$-handlebodies for some $k\leq 4$ and some choice of evaluation data for $\ev(\emptyset)\in \kk^\times$. Then $\sktft_A$ extends to $(4,k)$-handlebodies for any choice of evaluation data $y=\ev(\emptyset)\in \kk^\times$. Write $\sktft_A^y$ for such an extension. Then the invariant on a $(4,k)$-handlebody $W$ satisfies
	\begin{align*}		\sktft^y_A(W)=y^{\chi(W)}\cdot\sktft_A^1(W)
	\end{align*}
	where $\chi(W)$ is the Euler characteristic of $W$.	 
\end{prop}
\begin{proof}
	The handle map $m_0$ only contributes invertible scalars to $p_l$, $c_l$ and $m_l$ for all $l\leq k$, and does not affect the perfectness of the pairings $p_l$. Tracing through Construction \ref{con:extension}, we record the occurring factors of $y=\ev(\emptyset)\in \kk^\times$. The 0-handle carries a factor of $y$. The construction of the pairing $p_1$ in step (1) uses the map $m_0$ two times. Hence, if it is perfectness, the copairing $c_1$  carries a factor of $y^{-2}$. In the construction of $m_1$ in step (3) another factor of $y$ is introduced by $m_0$. Hence, $m_1$ contributes the scalar $y^{-1}$. Using step (1) again, we obtain that $p_2$ has a factor of $y^0$ and so does $c_2$, if it exists. This yields a factor of $y$ in $m_2$. Repeating the above, we find that $m_0,m_2,m_4$ contribute $y$, and $m_1,m_3$ contribute $y^{-1}$.
\end{proof}

\begin{con}[1-handles]\label{con:extToOneHandles}
    Let $A$ be a commutative Frobenius algebra over a commutative ring $\kk$. We construct the extension of $\sktft_A$ to $(4,1)$-handlebodies by applying Construction \ref{con:extension}. For 0-handles, we can restrict to the choice $\ev(\emptyset)=1\in \kk^\times$ by Proposition \ref{prop:evEulerChar}. The attaching region of a 1-handle $B^1\times B^3$ is $\alpha_1=S^0\times B^3=B^3\sqcup B^3$. Let $c,d\subset \partial B^3$ be boundary conditions. Then by gluing the two components separately, we can write
    \begin{align*}
        p_1^{(c\sqcup d)} = \tilde{p}_1^{(c)} \otimes \tilde{p}_1^{(d)}
    \end{align*}
    for a pairing
    \begin{align*}
        \tilde{p}_1^{(c)}\colon \Sk_A(B^3,c)\otimes \Sk_A(B^3,c)\to \kk.
    \end{align*}    
    The pairing $p_1^{(c\sqcup d)}$ is perfect for all boundary conditions if and only if $\tilde{p}_1^{(c)}$ is. Recall from Example~\ref{exm:FirstSkeinModules} that the surface skein module of $B^3$ with boundary condition $c$ is 
    \begin{align*}
        \Sk_A(B^3,c) \cong A^{\otimes \pi_0(c)}.
    \end{align*}
    Under this identification, the pairing $\tilde{p}_1^{(c)}$ is the tensor product of the Frobenius pairings in every factor of $A$ and hence perfect. It follows that the copairing for $p_1^{(c\sqcup d)}$ 
    \begin{align*}
        \kk \to \Sk_A(S^0\times B^3, c\sqcup d)\otimes \Sk_A(S^0\times B^3, c\sqcup d) \cong A^{\otimes \pi_0(c\sqcup d)}\otimes A^{\otimes \pi_0(c\sqcup d)}
    \end{align*}
    is the tensor product $(\Delta \circ \eta)^{\otimes \pi_0(c\sqcup d)}$ of the Frobenius copairing $\Delta\circ \eta\colon \kk \to A\otimes A$.
    The bimodule map
    \begin{align*}
         c_1 \colon {}_{\skcat_A(S^0\times S^2)}\Sk_A(I\times S^0\times S^2)_{\skcat_A(S^0\times S^2)} &\to {}_{\skcat_A(S^0\times S^2)}\Sk_A(S^0\times B^3)\otimes \Sk_A(S^0\times B^3)_{\skcat_A(S^0\times S^2)}
     \end{align*}
    from the regular bimodule is determined by its values on elements $\id_{c\sqcup d} = I\times (c\sqcup d)\in \Sk_A(I\times S^0\times S^2, c\sqcup d)$. It sends the cylinder $\id_{c\sqcup d}$ to the neck-cut cylinder in $\Sk_A(S^0\times B^3,c\sqcup d)\otimes \Sk_A(S^0\times B^3, c\sqcup d)$. Note that the cylinder skein can be neck-cut inside $\Sk_A(I\times S^0\times S^2, c\sqcup d)$ by the skein relations. As a consequence, the bimodule map $c_1$ is already determined\footnote{ Using Remark \ref{rmk:gluingMorita} and $\skcat_A(S^0\times S^2)=\skcat_A(S^2)\otimes \skcat_A(S^2)$, this is a consequence of the fact that $\skcat_A(S^2)$ is Morita equivalent to the full subcategory on the vacuum object $\emptyset\subset S^2$.} by its value 
    \begin{align*}
        \id_\emptyset\mapsto \emptyset\otimes \emptyset\in \Sk_A(S^0\times B^3, \emptyset)\otimes \Sk_A(S^0\times B^3, \emptyset).
    \end{align*}

    By the above, we obtain that $\sktft_A$ assigns to a $(4,1)$-handlebody $W$ the linear functional
    \begin{align*}
        \sktft_A(W) \colon \Sk_A(\partial W) &\to \kk\\
        S&\mapsto \eval_A(S)
    \end{align*}
    given by abstract evaluation as in Lemma~\ref{lem:absevalskein}, see also Example~\ref{exm:FirstSkeinModules}.(iv).
\end{con}

As a direct consequence of Theorem \ref{thm:WalkerExtension}, we have the following.
\begin{cor}\label{cor:allAExtendToOneH}
 	For every commutative Frobenius algebra $A$ over a commutative ring $\kk$, the surfaces skein theory $\sktft_A$ extends to $(4,1)$-handlebodies by Construction~\ref{con:extToOneHandles}.
 \end{cor}
 
\begin{exm}
\label{exm:s2s1eval}
Let $A$ be a commutative Frobenius algebra over a commutative ring $\kk$.
Then the  $(4,1)$-handlebody $B^3\times S^1$ induces the $\kk$-linear map
\[
\Sk_A(S^2\times S^1) \to \Sk_A(\emptyset)=\kk,\quad S \mapsto \eval_A(S)
\]
which coincides with the abstract evaluation of $A$-decorated skeins $S$ from Construction~\ref{con:abstractEval}, see also Lemma~\ref{lem:absevalskein}.
\end{exm}

We now consider the skein module of the thickened sphere $S^2\times B^1$, the attaching region of the 4-dimensional 3-handle.
\begin{lem}\label{lem:incompInThickSTwo}
    Isotopy classes of closed incompressible surfaces in $S^2\times B^1$ are represented by parallel copies of essential spheres $S^2\times \{*\}$.
\end{lem}
\begin{proof}
    By Lemma \ref{lem:SurfInsimplyCon}, it follows that every surface in $S^2\times B^1$ is orientable. Applying Lemma \ref{lem:piInjective}, it follows that the only connected closed incompressible surfaces can be essential spheres which can be isotoped to $S^2\times \{*\}$ by an innermost disk argument.
\end{proof}
Recall from Definition~\ref{def:skeincat} that $\Sk_A(S^2\times B^1, \emptyset) = \End_{\skcat_A(S^2)}(\emptyset)$ with empty boundary condition $\emptyset \subset S^2$ affords the structure of an algebra by gluing and rescaling in the $B^1$ direction. Denote by $\mathsf{S}_a\in \Sk_A(S^2\times B^1,\emptyset)$ the isotopy class of an essential sphere $S^2\times \{*\}\subset S^2\times B^1$ with decoration $a\in A$. Let $(a_i)_{i\in I}$ be a generating set for $A$. As an algebra, $\Sk_A(S^2\times B^1, \emptyset)$ is generated by $\{\mathsf{S}_{a_i}\mid i \in I\}$ where we write $k$ parallel spheres with decorations as linear combinations of words in the $\mathsf{S}_{a_i}$. Relations between them arise from tubing between two adjacent parallel copies. In particular, a surface given by two parallel decorated essential spheres joined by a tube is compressible. Applying the neck-cutting relation to the tube results in a sum where each summand consists of two decorated essential spheres. Compare \cite[Lemma 3.2 and 3.3]{AsaedaFrohman} in the case of $A_{\mathrm{BN}}$. More general arguments in this direction are developed in \cite{KaiserTunneling}. Explicitly, for the skein module $\Sk_A(\Sigma\times B^1,\emptyset)$ with a closed surface $\Sigma$, see \cite[Theorem 11.2]{KaiserTunneling}. For $\Sigma=S^2$, we have the following statement.
\begin{prop}\label{prop:TensorAlgebraSkModSphere}
	Let $A$ be a commutative Frobenius algebra. Then, the map that sends an elementary tensor $a_1\otimes \dots \otimes a_n \in A^{\otimes n}$ to the element
    \begin{align*}
        \mathsf{S}_{a_1}\cdots \mathsf{S}_{a_n}\in \Sk_A(S^2\times B^1, \emptyset)
    \end{align*} 
    given by $n$ parallel copies of essential spheres $S^2\times \{*\}$ decorated with $a_1,\dots, a_n$ induces an isomorphism of $\kk$-algebras
	\begin{align*}
		T(A)/(\Delta(a)-\varepsilon(a) \mid a\in A)\xrightarrow{\cong} \Sk_A(S^2\times B^1, \emptyset) 
	\end{align*}
	where $T(A)=\bigoplus_{k\geq0} A^{\otimes k}$ is the tensor algebra of $A$.
\end{prop}
\begin{proof}
	This is shown in \cite[Theorem 11.2]{KaiserTunneling} for the skein module of thickened surfaces $\Sigma_g\times B^1$ of genus $g\geq 0$. While Kaiser considers orientable skeins, his result for genus $g=0$ is still applicable in our setting of unoriented skeins by Lemma~\ref{lem:incompInThickSTwo}.
\end{proof}
Recall the Frobenius algebra $A_\alpha$ over $\kk=K[\alpha^{\pm1}]$ from Example \ref{exm:FAExmLinkHom}. We study the extension of surface skein theory $\sktft_\alpha$ for $A_\alpha$ to 4-dimensional 2-handles in Section \ref{sec:KirbyCol}. However, the TQFT $\sktft_\alpha$ does not extend to 3-handles.
\begin{cor}\label{cor:alphaSkmodThickSphere}
    Denote by $\mathsf{S},\mathsf{D}\in \Sk_\alpha(S^2\times B^1,\emptyset)$ the isotopy class of an essential sphere $S^2\times \{*\}\subset S^2\times B^1$ with decoration $1\in A_\alpha$ and $x\in A_\alpha$ respectively. As a $\kk$-algebra, we have
    \begin{align*}
        \Sk_\alpha(S^2\times B^1,\emptyset) = \kk\langle \mathsf{S}, \mathsf{D}\rangle
        /(\mathsf{S} \mathsf{D} + \mathsf{D} \mathsf{S},
        \mathsf{D} \mathsf{D} + \alpha \mathsf{S} \mathsf{S} - 1).
    \end{align*}
    A $\kk$-basis for $\Sk_\alpha(S^2\times B^1,\emptyset)$ is given by
    \begin{align*}
        \{\mathsf{S}^k, \mathsf{S}^k\mathsf{D} \mid k\geq 0\}.
    \end{align*}
    In particular, it is infinite-dimensional as a $\kk$-module.
\end{cor}
\begin{proof}
    We apply Bergman's diamond lemma \cite{BergmanDiamondLemma} to the left-hand side of Proposition \ref{prop:TensorAlgebraSkModSphere}, specialized to the Frobenius algebra $A_\alpha$ over $\kk$, which  is free with basis $\{1,x\}$. The map from Proposition \ref{prop:TensorAlgebraSkModSphere} identifies the tensor algebra $T(A_\alpha)$ with the free $\kk$-algebra $\kk\langle\mathsf{S}, \mathsf{D}\rangle$. 
    Since the counit on the basis elements is given as $\varepsilon(1)=0$ and $\varepsilon(x)=1$, we obtain that the ideal $(\Delta(a)-\varepsilon(a) \mid a\in A_\alpha)$ is generated by the relations 
    \begin{align}
        \mathsf{S} \mathsf{D} + \mathsf{D} \mathsf{S} & = 0,\label{eq:tubingWith1}\\
        \mathsf{D} \mathsf{D} + \alpha \mathsf{S} \mathsf{S} - 1 &=0.\label{eq:tubingWithx}
    \end{align}
    Choose the monomial order on $\kk\langle\mathsf{S}, \mathsf{D}\rangle$ given by degree lexicographic order with $\mathsf{S}<\mathsf{D}$. Then, this defines a total order with
    \begin{align*}
        1<\mathsf{S}<\mathsf{D}< \mathsf{S}\mathsf{S}<\mathsf{S}\mathsf{D}< \mathsf{D}\mathsf{S}<\mathsf{D}\mathsf{D}<\dots
    \end{align*}
    The relations \eqref{eq:tubingWith1} and \eqref{eq:tubingWithx} have leading words $\mathsf{D}\mathsf{S}$ and $\mathsf{D}\mathsf{D}$ respectively. Next, we need to check for ambiguities. The only minimal overlap ambiguities are $\mathsf{D}\mathsf{D}\mathsf{S}$ and $\mathsf{D}\mathsf{D}\mathsf{D}$. There are no inclusion ambiguities. For $\mathsf{D}\mathsf{D}\mathsf{S}$, we have the resolution of ambiguities:
    \begin{align*}
        \mathsf{D}\mathsf{D}\mathsf{S} &\xrightarrow{\eqref{eq:tubingWith1}} -\mathsf{D}\mathsf{S}\mathsf{D} \xrightarrow{\eqref{eq:tubingWith1}}\mathsf{S}\mathsf{D}\mathsf{D} \xrightarrow{\eqref{eq:tubingWithx}} -\alpha \mathsf{S}\mathsf{S}\mathsf{S}+ \mathsf{S},\\
        \mathsf{D}\mathsf{D}\mathsf{S} &\xrightarrow{\eqref{eq:tubingWithx}} -\alpha \mathsf{S}\mathsf{S}\mathsf{S}+ \mathsf{S}.
    \end{align*}
    For $\mathsf{D}\mathsf{D}\mathsf{D}$, we have:
    \begin{align*}
        \mathsf{D}\mathsf{D}\mathsf{D} &\xrightarrow{\eqref{eq:tubingWithx}} -\alpha\mathsf{D}\mathsf{S}\mathsf{S} + \mathsf{D} \xrightarrow{\eqref{eq:tubingWith1}} \alpha\mathsf{S}\mathsf{D}\mathsf{S} + \mathsf{D}\xrightarrow{\eqref{eq:tubingWith1}} 
        -\alpha\mathsf{S}\mathsf{S}\mathsf{D}+ \mathsf{D},\\
        \mathsf{D}\mathsf{D}\mathsf{D} &\xrightarrow{\eqref{eq:tubingWithx}} -\alpha\mathsf{S}\mathsf{S}\mathsf{D} + \mathsf{D}.
    \end{align*}
    It follows by the diamond lemma that the set of irreducible monomials, that is, monomials that do not contain the leading words $\mathsf{D} \mathsf{S}$ or $\mathsf{D} \mathsf{D}$ form a basis. With $\mathsf{S}^0=\emptyset$, we obtain that $\{\mathsf{S}^k, \mathsf{S}^k\mathsf{D} \mid k\geq 0\}$ is a $\kk$-basis for $\Sk_\alpha(S^2\times B^1,\emptyset)$. 
\end{proof}
\begin{cor}\label{cor:alphaNonExtensionTo3}
    The surface skein TQFT $\sktft_\alpha$ cannot be extended to 4-dimensional 3-handles.
\end{cor}
\begin{proof}
    Corollary \ref{cor:alphaSkmodThickSphere} shows that the surface skein module of attaching region for the 4-dimensional 3-handle $S^2\times B^1$ is infinite-dimensional. Hence, there cannot be a perfect pairing defined on this skein module and the surface skein TQFT $\sktft_\alpha$ does not extend to 3-handles.
\end{proof}
\begin{rmk}
    We expect Corollary \ref{cor:alphaNonExtensionTo3} to be true for a large class of free commutative Frobenius algebras of rank greater than one.
\end{rmk}

\subsection{Extending rank 1 surface skein theory}\label{sec:TrivialTheoryDW}
Let $\kk$ be a commutative ring for which $2\in \kk^\times$ is invertible. Let $u\in \kk^\times$ be a unit and $\kk^u$ the Frobenius algebra from Example \ref{exm:trivFA}, which is free of rank 1 over $\kk$. The counit and coproduct are given by $\varepsilon\colon 1 \mapsto u$ and $\Delta\colon 1\mapsto u^{-1}1\otimes 1$. If $u=1$, we simply denote the Frobenius algebra by $\kk=\kk^1$.

We now consider surface skein theory associated to $\kk^u$ and construct its extension to 4-dimensional $k$-handles for all $k$.
 
\begin{prop}\label{prop:SkModForRkOne}
	Let $M$ be a 3-manifold and $\kk^u$ as above. Then, for a boundary condition $c\subset \partial M$, the long exact sequence \eqref{eq:LES} provides an isomorphism $\Sk_{\kk^u}(M,c)\cong \kk\{\delta^{-1}[c]\}$. In particular, there is a non-canonical isomorphism $\Sk_{\kk^u}(M,c)\cong \kk\{H_2(M;\ZZ/2)\}$. 
\end{prop}
\begin{proof}
     Recall the decomposition from Proposition \ref{prop:HomologDecompOfSkMod}. The skein module $\Sk_{\kk^u}(M,c; \xi)$ is of rank 1 for all $\xi\in \partial^{-1}[c]$, and it follows from \eqref{eq:LES} that $\partial^{-1}[c]$ is a $H_2(M;\ZZ/2)$-torsor. See \cite[Example 8.3(d)]{KaiserFrobAlg} for $u=1$ and note that the arguments apply to general scalars $u\in \kk^\times$. Also compare \cite[Remark 4.2(e)]{KaiserFrobAlg}.
\end{proof}
\begin{prop}\label{prop:bdyConditionsRk1}
    Let $\Sigma$ be a compact oriented surface. The set of isomorphism classes of objects in the skein category $\skcat_{\kk^u}(\Sigma)$ is in bijection to $H_1(\Sigma;\ZZ/2)$.
\end{prop}
\begin{proof}
    The skein category $\skcat_{\kk^u}(\Sigma)$ is graded by $H_1(\Sigma;\ZZ/2)$: every object $c$ represents a first homology class $[c]\in H_1(\Sigma;\ZZ/2)$ and any nonzero morphisms between two objects exhibits the underlying 1-manifolds as homologous. Conversely, if two objects $c_1,c_2$ of $\skcat_{\kk^u}(\Sigma)$ satisfy $[c_1]=[c_2]$, then we can find an embedded surface $S\subset \Sigma \times [0,1]$, which yields an isomorphism from $c_1$ to $c_2$ in $\skcat_{\kk^u}(\Sigma)$. 
\end{proof}
\begin{con}[Extension for $\kk^u$]\label{con:extendingRkOne}
	We apply Construction \ref{con:extension} to surface skein theory $\sktft_{\kk^u}$ associated to the Frobenius algebra $\kk^u$. To simplify notation, we denote the skein modules and skein categories with index $\kk$ instead of $\kk^u$. First note that as a consequence of Proposition \ref{prop:SkModForRkOne} the skein modules of all 3-manifolds are finite-dimensional and free. By Proposition \ref{prop:bdyConditionsRk1}, isomorphism classes of boundary conditions for the skein module of a compact oriented 3-manifold are given by first homology classes $c\in H_1(\partial M; \ZZ/2)$. Recall from Proposition \ref{prop:HomologDecompOfSkMod}  that the skein module is non-zero only for boundary conditions $c$ with $\delta^{-1}[c]\neq \emptyset$. We call such homology classes \emph{admissible boundary conditions} for the skein module of $M$.

    \smallskip

	\noindent\textbf{0-handles.} The datum for 0-handles is determined by the choice of an evaluation $\ev(\emptyset)\in \kk^\times$ of the empty skein $\emptyset\in \Sk_\kk(S^3)$. By Proposition \ref{prop:evEulerChar}, we can restrict to the case $\ev(\emptyset)=1$ for $m_0$.

    \smallskip
	\noindent\textbf{1-handles.} 
        By Construction \ref{con:extToOneHandles}, the bimodule map
		\begin{align*}
		m_1\colon {}_{\skcat_\kk(S^0\times S^2)}\Sk_\kk(B^1\times S^2)\to {}_{\skcat_\kk(S^0\times S^2)}\Sk_\kk(S^0\times B^3)
		\end{align*}
        is given by abstract evaluation of the skeins in $\Sk_{\kk}(B^1\times S^2,c)$ for every boundary condition $c$. By Remark~\ref{rmk:gluingMorita}, we only need to consider the boundary condition $\emptyset\sqcup \emptyset\in \skcat_{\kk}(S^0\times S^2)$ and obtain
		\begin{align*}
			m_1\colon \Sk_\kk(B^1\times S^2, \emptyset\sqcup \emptyset)&\to \Sk_\kk(S^0\times B^3,\emptyset\sqcup \emptyset)\\
			\emptyset&\mapsto \emptyset\\
			\mathsf{S}&\mapsto u\cdot\emptyset.
		\end{align*}
\smallskip

	\noindent\textbf{2-handles.} A 2-handle $B^2\times B^2$ has attaching region $S^1\times B^2$ and belt region $B^2\times S^1$. For the admissible boundary conditions of the solid torus $S^1\times B^2$, we only have to consider the homology classes $(0,0)$ and $(0,1)$ in $H_1(S^1\times S^1; \ZZ/2)$ by Remark \ref{rmk:gluingMorita} and Proposition \ref{prop:bdyConditionsRk1}. The skein module of $S^1\times B^2$ is zero for the boundary conditions $c=(1,0)$ and $c=(1,1)$. 
    For $c=(0,0)$ we compute the pairing $p_2^{(c)}$ as the composition 
	\begin{align*}
			\Sk_\kk(S^1\times B^2,(0,0)) \otimes \Sk_\kk(S^1\times B^2,(0,0)
			\xrightarrow{\text{glue}}&\ \Sk_\kk(S^1\times S^2,\emptyset)\xrightarrow{\text{abs. eval.}} \kk
	\end{align*}
    where we applied the functional $\sktft_\kk(S^1\times B^3)$ from Example~\ref{exm:s2s1eval} using abstract evaluation of skeins in $\Sk_\kk(S^1\times S^2)$.
	The skein module of the solid torus is $\Sk_\kk(S^1\times B^2, (0,0)) \cong \kk\{\emptyset\}$ and pairing the empty skein with itself yields 
	$\emptyset \otimes \emptyset \mapsto 1$. For the boundary condition $c=(0,1)$, the meridional disk $\mathsf{M}\in \Sk_\kk(S^1\times B^2,(0,1))$ spans the skein module. The pairing is the map
    \begin{align*}
        \Sk_\kk(S^1\times B^2,(0,1)) \otimes \Sk_\kk(S^1\times B^2,(0,1))
		\xrightarrow{\text{glue}}&\ \Sk_\kk(S^1\times S^2,\emptyset)\xrightarrow{\text{abs. eval.}} \kk
    \end{align*}
    given by $\mathsf{M}\otimes \mathsf{M}\mapsto u$. The copairing assembles into the the bimodule map
	\begin{align*}
		c_2 \colon {}_{\skcat_\kk(S^1\times S^1)}\Sk_\kk(I\times S^1\times S^1)_{\skcat_\kk(S^1\times S^1)} &\to {}_{\skcat_\kk(S^1\times S^1)}\Sk_\kk(S^1\times B^2)\otimes_\kk \Sk_\kk(S^1\times B^2)_{\skcat_\kk(S^1\times S^1)}\\
		\id_{(0,0)} &\mapsto \emptyset \otimes \emptyset\\
        \id_{(0,1)} &\mapsto u^{-1} \mathsf{M}\otimes \mathsf{M}\\
        \id_{(1,0)} &\mapsto 0\\
        \id_{(1,1)} &\mapsto 0
	\end{align*}
    We obtain the handle map for $c=(0,0)=\emptyset$ 
	\begin{align*}
		m_2\colon \Sk_\kk(B^2\times S^1,\emptyset) &\to
		\Sk_\kk(S^1\times B^2,\emptyset)\\
		\emptyset &\mapsto \emptyset,
	\end{align*}
    and on the boundary conditions $c=(0,1),(1,0)$ and $(1,1)$, the handle map $m_2$ is zero since source or target are zero.

    \smallskip
	\noindent\textbf{3-handles.}
	The attaching region for the 3-handle $B^3\times B^1$ is $\alpha_3=S^2\times B^1$. For the skein module of $S^2\times B^1$, we again only have to consider the (admissible) boundary condition $\emptyset$. The skein module is
	\begin{align*}
		\Sk_\kk(S^2\times B^1,\emptyset)&\cong \kk\{\emptyset, \mathsf{S}\}
	\end{align*}
	where $\mathsf{S}$ denotes the class of a 2-sphere $S^2\times \{*\}\subseteq S^2\times B^1$. 
	The pairing $p_3^{(\emptyset)}$ is the composition
	\begin{align*}
		\Sk_\kk(S^2\times B^1,\emptyset) \otimes \Sk_\kk(S^2\times B^1,\emptyset)\xrightarrow{\text{glue}}&\ \Sk_\kk(S^2\times S^1)\\
		\cong&\ \Sk_\kk(B^2\times S^1) \otimes_{\skcat_\kk(S^1\times S^1)} \Sk_\kk(B^2\times S^1)\\
		\xrightarrow{\id\otimes m_2}&\ \Sk_\kk(B^2\times S^1)\otimes_{\skcat_\kk(S^1\times S^1)} \Sk_\kk(S^1\times B^2) \\
		\cong&\ \Sk_\kk(S^3)\\
        \xrightarrow{m_0}&\ \kk.
	\end{align*}
	We show that pairing is given by 
	\begin{align*}
		p_3^{(\emptyset)}\colon \Sk_\kk(S^2\times B^1,\emptyset) \otimes \Sk_\kk(S^2\times B^1,\emptyset)& \to \kk\\
		\emptyset \otimes \emptyset &\mapsto 1\\
		\emptyset \otimes \mathsf{S}&\mapsto 0\\
		\mathsf{S} \otimes \emptyset &\mapsto 0\\
		\mathsf{S} \otimes \mathsf{S}&\mapsto u^2.
	\end{align*}
	The empty skein is mapped to 1 since we chose $\ev(\emptyset)=1$. The skeins $\emptyset \otimes \mathsf{S}$ and $\mathsf{S} \otimes \emptyset$ each glue to an essential sphere in $\Sk_\kk(S^2\times S^1)$ which is cut into two meridional disks 
	\begin{align*}
	\mathsf{M}\otimes  \mathsf{M}\in \Sk_\kk(B^2\times S^1, (1,0))\otimes \Sk_\kk(B^2\times S^1,(1,0))
	\end{align*} 
	and mapped to zero by $\id \otimes m_2$. Two parallel essential spheres in $S^2\times S^1$ can be joined by a neck yielding a factor of $u$
    and become a compressible sphere which is equal to $u$. We obtain that $\mathsf{S}\otimes \mathsf{S}\mapsto u^2$. 
    The pairing is perfect with copairing
	\begin{align*}
		c_3\colon {}_{\skcat_\kk(S^2\times S^0)}\Sk_\kk(I\times S^2\times S^0)_{\Sk_\kk(S^2\times S^0)} \to {}_{\skcat_\kk(S^2\times S^0)}\Sk_\kk(S^2\times B^1) \otimes_\kk \Sk_\kk(S^2\times B^1)_{\skcat_\kk(S^2\times S^0)}
	\end{align*}
	given for the boundary condition $\emptyset$ as
	\begin{align*}
		\id_{\emptyset}&\mapsto \emptyset\otimes \emptyset + u^{-2}\mathsf{S}\otimes \mathsf{S}.
	\end{align*}
	We construct the handle map $m_3$
	\begin{align*}
		m_3\colon {}_{\skcat_\kk(S^2\times S^0)}\Sk_\kk(B^3\times S^0) &\to {}_{\skcat_\kk(S^2\times S^0)}\Sk_\kk(S^2\times B^1)
	\end{align*}
	using step (3) in Construction \ref{con:extension}. We again only have the boundary condition $\emptyset\sqcup \emptyset\in \skcat_\kk(S^2\times S^0)$. Applying the copairing $c_2$ followed by $\id \otimes m_0$, the sphere $\mathsf{S}$ in the second factor becomes compressible in $\Sk_\kk(S^3)$ and is replaced by a factor of $u$. 
    We obtain
	\begin{align*}
		m_3\colon \Sk_{\kk}(B^3\times S^0, \emptyset\sqcup \emptyset) &\to \Sk_{\kk}(S^2\times B^1,\emptyset\sqcup\emptyset)\\
		\emptyset &\mapsto \emptyset + \tfrac{1}{u}\mathsf{S}.
	\end{align*}
\smallskip

	\noindent\textbf{4-handles.}
	The 4-handle has attaching region $S^3$. Its skein module has no non-empty boundary conditions, and $\Sk_\kk(S^3)\cong \kk\{\emptyset\}$. Using $m_3$ and $m_0$, we construct the pairing $p_4$ as the composition
	\begin{align*}
		\Sk_\kk(S^3) \otimes \Sk_\kk(S^3)\xrightarrow{\text{glue}}&\ \Sk_\kk(S^3\times S^0)\\
		\cong&\ \Sk_\kk(B^3\times S^0) \otimes_{\skcat_\kk(S^2\times S^0)} \Sk_\kk(B^3\times S^0)\\
		\xrightarrow{\id\otimes m_3}&\ \Sk_\kk(B^3\times S^0)\otimes_{\skcat_\kk(S^2\times S^0)} \Sk_\kk(S^2\times B^1) \\
        \cong&\ \Sk_\kk(S^3)\\\xrightarrow{m_0}&\ \kk.
	\end{align*}
	After applying $m_3$ from above and gluing to $S^3$, the essential sphere $\mathsf{S}$ in the term $\emptyset+\tfrac{1}{u}\mathsf{S}\in \Sk_\kk(S^2\times B^1)$ becomes compressible and evaluates to $u$. 
    We obtain for the pairing
	\begin{align*}
		p_4\colon \Sk_\kk(S^3)\otimes_\kk \Sk_\kk(S^3)&\to \kk\\
		\emptyset\otimes \emptyset&\mapsto (1+ \tfrac{u}{u})=2.
	\end{align*}		
	It follows that $p_4$ is perfect by the assumption that $2\in \kk^\times$. We obtain as copairing	
	\begin{align*}
		c_4 \colon {}_{\skcat_\kk(\emptyset)}\Sk_\kk(I\times \emptyset)_{\Sk(\emptyset)}&\to  {}_{\skcat_\kk(\emptyset)}\Sk_\kk(S^3\times B^0) \otimes_\kk \Sk_\kk(S^3\times B^0)_{\skcat_\kk(\emptyset)}
	\end{align*}
	sending $\id_\emptyset \mapsto\tfrac{1}{2} \emptyset\otimes \emptyset$. Finally, the handle map is	
	\begin{align*}
		m_4\colon {}_{\Sk(\emptyset)}\Sk(\emptyset) &\to {}_{\Sk(\emptyset)}\Sk(S^3)\\
		\emptyset &\mapsto \tfrac{1}{2} \emptyset.
	\end{align*}
\end{con}
\begin{thm}\label{thm:rkOneTheory}
	Construction \ref{con:extendingRkOne} provides an extension of the surface skein theory $\sktft_{\kk^u}$ to a symmetric monoidal functor of bicategories
    \begin{align*}
        \sktft_{\kk^u}\colon \cob_{2,3,4}^{\mathrm{vop}}\to \mor(\kcat).
    \end{align*}
    Let $W$ be a closed oriented 4-manifold. The invariant associated to the Frobenius algebra $\kk$ assigns to $W$
	\begin{align*}
		\sktft_{\kk^u}(W)=\tfrac{1}{2}|H_3(W;\ZZ/2)|.
	\end{align*}
\end{thm}
\begin{proof}
    Let $W$ be a closed oriented 4-manifold. Assume without loss of generality that $W$ is connected. If $W$ is not connected, the same argument applies to every connected component and the value is additive. Then, for a handle decomposition of $W$, we can assume that there is only one 0-handle and one 4-handle. The map $m_4$ contributes the factor $\tfrac{1}{2}$. Consider the presence of a 3-handle. By assumption, we have that the 3-handle is not in canceling configuration with the unique 4-handle. By the proof of Theorem \ref{thm:WalkerExtension}, we may assume that the 3-handle is not in canceling configuration with a 2-handle; otherwise this pair would contribute the identity. Every non-canceling 3-handle adds a term $\tfrac{1}{u}\mathsf{S}$ which is evaluated to 1 by subsequent handle maps (via abstract evaluation or the skein relations). All other evaluations by handle maps do not contribute any factors or terms. We obtain that $\sktft_{\kk^u}(W)$ counts the number of third homology classes of $W$.
\end{proof}
We now briefly make connection between the surface skein theory $\sktft_{\kk^u}$ and 4-dimensional Dijkgraaf-Witten theory \cite{DijkgraafWitten} for the finite group $G=\ZZ/2$. For this, we first recall the following. Let $X$ be a (essentially)\footnote{An \emph{essentially finite} groupoid is a groupoid equivalent to a finite one.} finite groupoid. The \emph{groupoid cardinality} is the rational number
\begin{align*}
	|X| =  \sum_{[x]\in \pi_0(X)} \frac{1}{|\mathrm{Aut}_\mathcal{G}(x)|}
\end{align*}
where the sum is taken over isomorphism classes of objects in $X$.

Let $G$ be a finite group. To a closed, oriented 4-manifold $W$, Dijkgraaf-Witten theory assigns the groupoid cardinality $\mathcal{DW}_G(W)=|\mathsf{Bun}_G(W)|$ where $\mathsf{Bun}_G(W)$ is the groupoid of principal $G$-bundles on $W$. If $W$ is connected, then
\begin{align*}
	\mathcal{DW}_{G}(W) = \frac{|\Hom(\pi_1(W),G)|}{|G|}.
\end{align*}
\begin{prop}\label{prop:rkOneAndDW}
	Let $W$ be a connected, closed, oriented 4-manifold. Then, $\mathcal{DW}_{\ZZ/2}(W)=\sktft_{\kk^u}(W)$.
\end{prop}
\begin{proof}
	Using the Hurewicz theorem, the universal coefficient theorem and Poincaré duality, we compute
	\begin{align*}
		\mathcal{DW}_{\ZZ/2}(W) &= \tfrac{1}{2} |\Hom(\pi_1(W),\ZZ/2)| = \tfrac{1}{2} |\Hom(\pi_1(W)^{\mathrm{ab}},\ZZ/2)| = \tfrac{1}{2}|\Hom(H_1(W;\ZZ),\ZZ/2)|\\
		&= \tfrac{1}{2}|H^1(W;\ZZ/2)| = \tfrac{1}{2}|H_3(W;\ZZ/2)| = \sktft_{\kk^u}(W).\qedhere
	\end{align*}
\end{proof}
It could be interesting to explore this connection further, but we do not pursue this here.

\subsection{The skein module of the solid torus}
\label{sec:solidtorus}
We now turn our attention back to surface skein theories for general commutative Frobenius algebras $A$, and study the extendability of $\sktft_A$ to 4-dimensional 2-handles. For this, we describe the skein module $\Sk_A(S^1\times B^2,c)$ of the solid torus, the attaching region of the 2-handle, for boundary conditions $c$.

Isotopy classes of oriented 1-manifolds embedded in the torus $S^1\times S^1$ without inessential components are classified by their first homology class in $H_1(S^1\times S^1; \ZZ)$. Consider the map 
\begin{align*}
	 \ZZ\oplus \ZZ\cong H_1(S^1\times S^1; \ZZ) \xrightarrow{\iota} H_1(S^1\times B^2; \ZZ) \cong \ZZ
\end{align*}
induced by the embedding $S^1\times S^1 \subset S^1\times B^2$ of the boundary torus. Denote by $\lambda$ and $\mu$ the curves $S^1\times \{*\}$ and $\{*\}\times S^1$ in $S^1\times S^1$ (for some choice of points $*$), respectively. We call $\lambda$ the longitude and $\mu$ the meridian. Their homology classes form a basis of $H_1(S^1\times S^1; \ZZ)$ which we also denote by $\{\lambda,\mu\}$. Note that $\mu$ is filled by the embedding, that is, $\iota(\mu)=0$. We can write a homology class in $H_1(S^1\times S^1;\ZZ)$ as $(l,m) = l\cdot \lambda+ m\cdot\mu$, where $l$ denotes the number of intersections with $\mu$, and $m$ the number of intersections with $\lambda$. Boundary conditions without inessential components and considered without orientation are thus indexed by pairs of integers $c=(l,m)\in (\ZZ\oplus\ZZ)/\{\pm1\}$ up to global sign. 

We will see in Construction \ref{con:TheInvariant} that for the extension to 2-handles, it suffices restrict to the case $m=0$. Then, by the homological condition in Proposition \ref{prop:HomologDecompOfSkMod}, for $l$ odd, the skein module of the solid torus for $c=(l,0)$ is zero. Hence, we focus on the boundary conditions $c=(2n,0)$ for $n\geq 0$. 

In the following, we use the notation $\mathbb{T}=S^1\times S^1$ for the torus and $\mathbb{H}=S^1\times B^2$ for the solid torus.

\begin{defn}\label{def:dTL}
	Let $A$ be a commutative Frobenius algebra. Define the \emph{$A$-decorated Temperley--Lieb category} $\dTL_A$ to be the $\kk$-linear category with:
	\begin{itemize}
		\item \textbf{Objects:} $n\in \NN_0$.
		\item \textbf{Morphisms:} Let $m,n\in \NN_0$ and consider the unit square $[0,1]\times[0,1]$ with $m$ points equidistantly embedded in $[0,1]\times\{0\}$ and $n$ points equidistantly embedded in $[0,1]\times\{1\}$. An $A$-decoration for a properly embedded 1-manifold $c\subset [0,1]\times[0,1]$ is a labeling function $\ell\colon \pi_0(c)\to A$. The $\kk$-module of morphisms $\Hom_{\dTL_A}(m,n)$ is spanned by isotopy classes (rel.\ boundary) of properly embedded 1-manifolds in $[0,1]\times [0,1]$ bounding $m\times\{0\}\sqcup n\times \{1\}$ with an $A$-decoration---we call them \emph{$\dTL_A$-diagrams from $m$ to $n$}. These are considered up to the following relations. 
        \begin{itemize}
            \item ($\kk$-multilinearity) The decorations are $\kk$-linear in each component.
            \item (Circle relations) If a connected component of an $\dTL_A$-diagram is a closed loop with decoration $a\in A$, then it can be removed at the expense of scaling the remaining $\dTL_A$-diagram by the element $\varepsilon\circ m\circ \Delta(a)\in \kk$. Pictorially, this means: 
            \begin{align*}
                \dTLCirclea = \varepsilon\circ m\circ\Delta(a)\; \dTLemptyCircle
            \end{align*}
            \item (Saddle relations) Suppose a disk $D\subset (0,1)\times (0,1)$ intersects a given $\dTL_A$-diagram $c$ transversely in exactly two arcs. Let $c'$ denote the $\dTL_A$-diagram obtained from $c$ by surgering the two arcs inside $D$. 
            For any $a\in A$, writing $\Delta(a)=\sum_i x_i\otimes y_i$, we impose the following relation between decorated $c$ and $c'$:
            \begin{align*}
                \sum_i\LocaldTLLeftRight{x_i}{y_i} = \sum_i\LocaldTLUpDown{x_i}{y_i}\;. 
            \end{align*}
        \end{itemize}
        \item \textbf{Composition:} The composition map $\dTL_A(m,n)\otimes\dTL_A(n,p) \to \dTL_A(m,p)$ is induced by gluing the two unit squares along the vertical direction and appropriate rescaling. The decorations of composed $\dTL_A$-diagrams are multiplied using the commutative multiplication of $A$. The identity morphism on $m$ is provided by the isotopy class of $m$ parallel vertical strands.
	\end{itemize}
    By using the circle relations, every $\dTL_A$-diagram can be represented (although not necessarily uniquely) by an embedded planar matching with decorations on the arcs. Note that $\Hom_{\dTL_A}(m,n)$ is zero if $m+n$ is odd.
\end{defn}
\begin{rmk}\label{rmk:TL-in-dTL}
    Consider the usual $\kk$-linear Temperley--Lieb category $\mathsf{TL}^\delta$ with circle value $\delta\in \kk$. For a free commutative Frobenius algebra $A$ with rank $\delta$, there is a functor $\mathsf{TL}^{\delta}\to \dTL_A$ which is the identity on objects and sends every $\mathsf{TL}$-diagram to the $\dTL_A$-diagram that has the same underlying embedded 1-manifold and has decoration $1\in A$ on every connected component. In particular, via this functor, we can view the Jones--Wenzl projectors as idempotents in $\dTL_A$.
\end{rmk}
\begin{rmk}
    The $A$-decorated Temperley--Lieb category $\dTL_A$ carries a monoidal structure. On objects it is given by addition, and the  monoidal unit is $0\in \NN_0$. On morphisms the monoidal structure is induced by stacking two unit squares horizontally and rescaling appropriately.
\end{rmk}

\begin{rmk}
	Let $m,n\geq 1$ with $m+n$ even, and let $0\leq k\leq n$. Nesting $k$ cups provides an isomorphism
	\begin{align*}
		\dTL_A(m,n)\cong\dTL_A(m+k,n-k).
	\end{align*}
\end{rmk}
\begin{prop}\label{prop:SolidtorusAnddTL}
	Write $I=[0,1]$ for the unit interval. Let $o_{2n}\subset S^1\times I \times \{1\}$ be $2n$ concentric circles in the annulus. Rotating $\dTL_A$-diagrams induces an isomorphism
	\begin{align*}
		\dTL_A(0,2n)\xrightarrow{\cong} \Sk_A(S^1\times I\times I, o_{2n}) \cong\Sk_A(\mathbb{H}, (2n,0)).
	\end{align*}
\end{prop}
\begin{proof}
    Rotation of a $\dTL_A$-diagram yields a decorated surface in $S^1\times I\times I$. A closed component in a $\dTL_A$-diagram decorated with $a\in A$ can be replaced by $\varepsilon\circ m\circ \Delta(a)\in \kk$. This is the same value assigned to neck-cutting the $a$-decorated torus in $S^1\times I\times I$ obtained by rotation of the closed component. Let $D$ be a $\dTL_A$-diagram from 0 to $2n$ that does not contain any closed components. By rotating $D$ we obtain (the isotopy class of) an $A$-decorated incompressible surface $S^1\times D$ in $(S^1\times I\times I, o_{2n})$. These are exactly all classes of $A$-decorated incompressible surfaces in $(S^1\times I\times I, o_{2n})$. By \cite[Theorem 9.2]{KaiserFrobAlg}, the skein module $\Sk_A(S^1\times I\times I, o_{2n})$ is the quotient of the $\kk$-linear span of all such decorated surfaces by the submodule generated by tubing between them. We extend the rotation map linearly and note that this is compatible with the linearity in the decorations on each component. Lastly, note that the saddle relation between two $\dTL_A$-diagrams corresponds to the tubing relation between two incompressible surfaces since a tube between two incompressible surfaces can be neck-cut in two different ways. This is proved in detail in \cite[Theorem 2.1]{RussellBNSolidTorus} for the case $A_{\mathrm{BN}}$ and generalizes to general commutative Frobenius algebras $A$.
\end{proof}

\begin{exm}\label{exm:dTLEnd1} Consider the map that sends an element $a\in A$ to the decorated Temperley--Lieb diagram consisting of a single $a$-decorated strand. This and the isomorphisms from Proposition~\ref{prop:SolidtorusAnddTL} give compatible isomorphisms:
	\begin{align*}
		A \cong \dTL_A(0,2)\cong \Sk_A(S^1\times I\times I, o_{2}) \cong\Sk_A(\mathbb{H}, (2,0)).
	\end{align*}
\end{exm}
\begin{proof}
    We show that an isomorphism is given by the map $A \to \Sk_A(\mathbb{H}, (2,0))$ that sends $a\in A$ to the $a$-decorated annulus in $\Sk_A(\mathbb{H}, (2,0))$, which is clearly surjective. To verify the injectivity, we further embed $(\mathbb{H},(2,0))$ into $(B^3,S^1\sqcup S^1)$ and use Example~\ref{exm:FirstSkeinModules} to identify its skein module with $A^{\otimes 2}$. The composite is $\Delta\colon A\to A^{\otimes 2}$, and is injective since it has left inverse $\id_A\otimes \varepsilon$, so the original map also had to be injective.
\end{proof}

\begin{rmk}
	For the Frobenius algebra $A_\mathrm{BN}$, the definition of $\dTL_{\mathrm{BN}}$ recovers the \emph{dotted Temperley--Lieb} category $\mathrm{dTL}$ used in \cite{HRWKirbyColorForKh} (up to a $\ZZ$-grading). The relations defining $\dTL_{\mathrm{BN}}$ also appear in \cite{AsaedaFrohman} as the \emph{SBN relations} for a general description of the vertical part of the skein module of a Seifert fibered 3-manifold. See \cite[Proposition 5.20]{AsaedaFrohman}. The endomorphism algebras of $\dTL_{\mathrm{BN}}$ are also known as \emph{nil-blob algebras}. 
\end{rmk}

\subsection{Strong separability and perfectness}
\begin{defn}
	Let $A$ be a symmetric Frobenius algebra over $\kk$.
	\begin{enumerate}[(i)]
		\item The \emph{handle element} or \emph{window element} of $A$ is the element $H=m\circ \Delta(1)\in A$. 
		\item $A$ is \emph{strongly separable} if $H$ is invertible (under the multiplication $m$). In this case the element $\Delta(H^{-1})\in A\otimes A$ is called the  \emph{separability idempotent}. The $(A,A)$-bimodule map
        \begin{align*}
            \Delta\circ m(H^{-1},-)\colon A\to A\otimes A, \quad 1\mapsto \Delta(H^{-1})
        \end{align*}
        provides a section for the multiplication. We call a strongly separable Frobenius algebra \emph{$\Delta$-separable} if the comultiplication itself provides the section $\Delta\circ m= \id$, that is, if $H=1=H^{-1}$. Note that some authors call such Frobenius algebras \emph{special}.
	\end{enumerate}
\end{defn}
\begin{rmk}[Related notions of separability]
	Recall that classically, an algebra $A$ is called separable if there exists a separability idempotent $e$ in the enveloping algebra $A\otimes A^\op$. An equivalent condition for the separability of an algebra $A$ is the existence of a section of the multiplication as a bimodule map. Our definition for strongly separable Frobenius algebras agrees with the classical definition, see \cite[Theorems 2.6 and 2.14]{LaudaPfeiffer2DTFTopenclosed}, and readily generalizes to symmetric Frobenius algebra objects in a (locally small) symmetric monoidal category with monoidal unit $\kk$. The handle element is then the morphism $\kk\to A$ given by $H=m\circ \Delta \circ \eta$. For a more detailed discussion of different definitions of (strong) separability of algebras in a categorical setting and a comparison to the classical definition, we refer to \cite[Sections 2.4 and 2.5]{LaudaPfeiffer2DTFTopenclosed}.	Also compare \cite[Definition 4.1]{CzenkyKestenQuinonezWaltonExtended} and the classical reference \cite{AguiarSeparable}.
\end{rmk}
\begin{lem}\label{lem:pairingsOnA}
	Let $A$ be a commutative Frobenius algebra and $u\in A$ be an element. Then,
	\begin{align*}
		p_u\colon A\otimes A \to \kk,\quad  p_u(a,b):= \varepsilon(m(u,m(a,b)))
	\end{align*}
	defines a pairing which is perfect if and only if $u$ is invertible in $A$. The Frobenius pairing is $p_1=\varepsilon\circ m$.
\end{lem}
\begin{proof}
	If $u$ is invertible, $\Delta(u^{-1})$ defines a copairing. Indeed, the string straightening relations follow from the relations of the Frobenius (co)pairing and from $m(u,u^{-1})=1$. Suppose there exists a copairing $q_u$. Then the string straightening relations for $p_u$ and $q_u$ imply that $(\id\otimes \varepsilon)\circ q_u(1)$ provides an inverse for $u$. More details in a more general setting can be found in \cite[Theorem 2.14]{LaudaPfeiffer2DTFTopenclosed}.
\end{proof}
\begin{prop}\label{prop:Ext2handlesSeparable}
	Let $A$ be a commutative Frobenius algebra. The pairing on the skein module $\Sk_A(\mathbb{H}, (2,0))$ is perfect if and only if $A$ is strongly separable.
\end{prop}
\begin{proof}
	The skein module $\Sk_A(\mathbb{H}, (2,0))$ is isomorphic to $A$ by Example~\ref{exm:dTLEnd1} and under this identification, the pairing is given by $p_H$ which by Lemma \ref{lem:pairingsOnA} is perfect if and only if $H$ is invertible.
\end{proof}
\begin{exm}[Bar-Natan theory does not extend]\label{exm:BNnoExtension}
	Consider the Frobenius algebra $A_\mathrm{BN}$. Since $x^2=0$, $H=2x$ is not invertible, and hence, the associated skein TQFT $\sktft_{\mathrm{BN}}$ does not extend to 2-handles. The structure of the skein module of the solid torus for $A_{\mathrm{BN}}$ has been studied in \cite{RussellBNSolidTorus} and \cite{HeymanPhDThesis}.
\end{exm}

\subsection{The Kirby color for 2-handles}
If the skein TQFT $\sktft_A$ extends to 4-dimensional 2-handlebodies, then the attachment of a single 4-dimensional 2-handle is modeled by a family of elements $\omega_{2n} \in \Sk_A(\mathbb{H},(2n,0))$ that we call the Kirby color for $\sktft_A$\footnote{Note that there are related, but different notions of Kirby colors in the literature. For example, in \cite{HRWKirbyColorForKh}, a Kirby color to model the attachment of 4-dimensional 2-handles for skein lasagna modules \cite{PaulInvariantsOf4Man, MNSkeinLasagna2Handle}. Their Kirby color is an object in a completion of the category $\dTL_{\mathrm{BN}}$ satisfying a handle slide move.}, as follows. 
\smallskip

Let $A$ be a strongly separable commutative Frobenius algebra. Let $W$ be a $(4,2)$-handlebody with boundary $M=\partial W$. If it exists, the invariant $\sktft_A$ associated to $W$ defines a linear map 
\begin{align*}
	\sktft_A(W)\colon \Sk_A(M)\to \kk.
\end{align*}
Let $S\in \Sk_A(M)$ be (the isotopy class of) a decorated surface. We can assume that $S$ intersects the core of the belt region $\beta_2=B^2\times S^1$ of a 2-handle only in meridional disks. All other surface configurations can be pushed out of $\beta_2$ through the 2-handle. The boundary curve bounded by $S$ is therefore $c=(m,0)$. The 2-handle map on $c$ is $m_2(c)\colon \Sk_A(B^2\times S^1,(m,0))\to  \Sk_A(S^1\times B^2, (m,0))$. When $m$ is odd, $\Sk_A(S^1\times B^2, (m,0))= 0$ since there are no surfaces in the solid torus bounding an odd number of longitudes. In the case that $m=2n$, we make the following definition.
\begin{defn}\label{def:capAndKirby}
	Let $n\geq 0$ and let $A$ be a commutative Frobenius algebra such that the pairing $p_2$ on  $\Sk_A(\mathbb{H},(2n,0))$ is perfect, and let $c_2$ be the copairing. Then, we define the following.
	\begin{enumerate}[(i)]
		\item The decomposition $S^3 = S^1\times B^2 \cup_{S^1\times S^1} B^2\times S^1$ defines a map
		\begin{align*}
			\Sk_A(S^1\times B^2, (2n,0)) \otimes \Sk_A(B^2\times S^1, (2n,0)) \xrightarrow{\mathrm{glue}} \Sk_A(S^3,\emptyset)\xrightarrow{m_0}\kk.
		\end{align*}
		Let $\mathsf{M}\in \Sk_A(B^2\times S^1,(2n,0))$ be represented by $2n$ parallel undecorated meridional disks. We define
		\begin{align*}
			\mathrm{cap}\colon \Sk_A(S^1\times B^2, (2n,0))\to \kk, \quad 
			\mathrm{cap}(-)=m_0\circ \mathrm{glue}(-\otimes \mathsf{M}).
		\end{align*}
		We call $\mathrm{cap}(S)\in \kk$ the \emph{cap value} of an element $S\in \Sk_A(S^1\times B^2,(2n,0))$. 
		\item The \emph{Kirby color} $\omega_{2n}$ is the element 
		\begin{align*}
			\omega_{2n}=(\id\otimes\mathrm{cap})(c_2(\id_{(2n,0)})) \in \Sk_A(\mathbb{H},(2n,0)).
		\end{align*}
	\end{enumerate}
\end{defn} 

\begin{prop}\label{prop:fillAnnulus}
	Let $A$ be a commutative Frobenius algebra such that the pairing $p_2$ on $\Sk_A(\mathbb{H},(2n,0))$ with boundary condition $c:=(2n,0)$ is perfect. Then the map 
	\begin{align*}
		\Sk_A(\mathbb{H},(2n,0))&\xrightarrow{\phi} \mathcal{F}_A(c) = A^{\otimes \pi_0(c)} \cong A^{\otimes 2n}\\
        S &\mapsto \mathcal{F}_A(S)(1)
	\end{align*}
    induced by abstract evaluation from Construction \ref{con:abstractEval} (see Lemma~\ref{lem:absevalskein} and \ref{exm:FirstSkeinModules}.(v)) is injective. 
	 In particular, we can think of the Kirby color as an element $\omega_{2n}\in A^{\otimes 2n}$.
\end{prop}
\begin{proof} By Example~\ref{exm:s2s1eval}, the pairing $p_2$ factors through the abstract evaluation, providing a commutative diagram:
\[\begin{tikzcd}
            \Sk_A(\mathbb{H},(2n,0)) \times \Sk_A(\mathbb{H},(2n,0)) &
            \Sk_A(S^2 \times S^1) 
            \\
            \mathrm{im}(\phi)\times \mathrm{im}(\phi) &
            \kk
            \arrow["\phi \times \phi",from=1-1, to=2-1]
            \arrow["\text{glue}",from=1-1, to=1-2]
            \arrow["\sktft_A(B^3\times S^1)",from=1-2, to=2-2]
            \arrow["p_2",from=1-1, to=2-2]
            \arrow[dashed,from=2-1, to=2-2]
\end{tikzcd}\]
Since $\phi$ preserves a perfect pairing (images of dual basis elements remain dual in $\mathrm{im}(\phi)$), it is an isomorphism onto its image, i.e.\ injective.
\end{proof}

\begin{rmk}
	We expect the above to be true for commutative Frobenius algebras $A$ in greater generality. In particular, the strong separability is not a necessary condition. For example, in the case of $A_{\mathrm{BN}}$, this is shown in \cite[Proposition 3.26]{HRWKirbyColorForKh} using the polynomial representation $\mathbf{Pol}$.
\end{rmk}

\begin{con}[The invariant of $(4,2)$-handlebodies]\label{con:TheInvariant}	
	Let $A$ be a commutative Frobenius algebra such that the pairing $p_2^{(c)}$ on the skein module of the solid torus is perfect for all boundary conditions $c=(2n,0)$. 
    Consider the surface skein TQFT $\sktft_A$ and assume it can be extended to $\cob_{2,3,3+2/4}$ with $m_0$ determined by the evaluation $\ev(\emptyset)=1$. For more general evaluations of the empty skein in $\Sk_A(S^3)$ to $\ev(\emptyset)\in \kk^\times$, we apply Proposition \ref{prop:evEulerChar}. Let $W$ be a connected oriented $(4,2)$-handlebody with boundary a closed oriented 3-manifold $\partial W$. In the following, we explicitly describe the linear functional
	\begin{align*}
		\sktft_A(W) \colon \Sk_A(\partial W) \to \kk
	\end{align*}
	assigned to $W$ by the partial extension of $\sktft_A$. 
	
	We can assume that $W$ has only one 0-handle. Each 1-handle is attached along some $S^0\times B^3$. After attachment of the 1-handles, the 4-manifold has boundary $\#_l(S^1\times S^2)$ where $l$ is the number of 1-handles. If there are no 2-handles, by abuse of notation we write $\#_0(S^1\times S^2)=S^3$. Each 2-handle is attached along the attaching region $S^1\times B^2$ embedded in $\#_l(S^1\times S^2)$. The embeddings are determined by a framed link $L$ in $\#_l(S^1\times S^2)$. Consider the dual link $L^\vee$ in $\partial W$, i.e.\ the link obtained by taking the cores of the solid tori $B^2\times S^1$ that are the belt regions of the 2-handles. 

    Let $S\hookrightarrow \partial W$ be an $A$-decorated closed surface in $\partial W$ transverse to $L^\vee$ and record the finite set of intersections $P=S\cap L^\vee$ and its partition $P=\sqcup_jP_j$ with $P_j :=S\cap L^\vee_j$, according to the connected components $L^\vee_j$ of $L^\vee$.
    Construction~\ref{con:abstractEvalext} with \ref{exm:FirstSkeinModules}.(iv) provides a $\kk$-linear abstract evaluation map
	\begin{align*}
		\eval_A(S,P)\colon \bigotimes_j A^{\otimes P_j}\to \kk
	\end{align*} 
     For the punctures corresponding to the link component $L_j^\vee$, we take the Kirby color $\omega_{|P_j|}\in A^{\otimes P_j}$ to arrive at the description:
	\begin{align*}
		\sktft_A(W) \colon \Sk_A(\partial W)\to \kk\;, \quad
        S\mapsto \eval_A(S,P)(\otimes_j \omega_{|P_j|}).
	\end{align*}
    By Proposition \ref{prop:incompGens}, the $\kk$-linear map $\sktft_A(W)$ is determined by its values on $A$-decorated incompressible surfaces in $\partial W$. 
\end{con}

\begin{rmk}\label{rmk:kirbyRecursiveCapping}
	The Kirby color $\omega_{2n}$ satisfies a recursive property. Let $(a_1,\dots, a_n)$ and $(b_1,\dots, b_n)$ be a pair of bases of $A$ that are dual under the Frobenius pairing. That is, $\beta(a_i,b_j)= \varepsilon\circ m(a_i,b_j)=\delta_{ij}$. Set $a_1=1$ and write $1^\vee=b_1$ for the element dual to $1$. 
    Then, $\varepsilon(1^\vee)=1$ and $\varepsilon(b_i)=0$ for $2\leq i\leq N$. Consider the Kirby color $\omega_{2n}\in \Sk_A(S^1\times I\times I,o_{2n})$ as above and choose two adjacent concentric circles in $o_{2n}$ at position $(k,k+1)$ and glue on an annulus decorated with $a\in A$. This defines an element 
	\begin{align*}
		\mathrm{Ann}^{(k,k+1)}_{a}(\omega_{2n})\in \Sk_A(S^1\times I \times I, o_{2n-2})
	\end{align*}
	satisfying
	\begin{align*}
		\mathrm{Ann}^{(k,k+1)}_{1^\vee}(\omega_{2n})=\omega_{2n-2}\quad \quad \text{and}\quad \quad  \mathrm{Ann}^{(k,k+1)}_{b_i}(\omega_{2n}) = 0,\quad \text{for } 2\leq i\leq N.
	\end{align*}
	To see this, note that the basis is chosen such that a compressible sphere decorated with $1^\vee$ evaluates to 1 and a compressible sphere decorated with $b_i$ evaluates to 0. In the situation of Construction \ref{con:TheInvariant}, consider a surface $S$ in the skein module of $\partial W$ for a $(4,2)$-handlebody $W$ such that $S$ intersects a 2-handle of $W$ in $2n-2$ points. By the sphere relation, we can create a compressible sphere decorated with $1^\vee$ in the belt region, a tubular neighborhood for the dual link $L^\vee$, such that $L^\vee$ intersects the sphere component twice in a row. We obtain the relation $\mathrm{Ann}^{k,k+1}_{1^\vee}(\omega_{2n})=\omega_{2n-2}$.	If the sphere component has a decoration by $b_i$, by the same arguments, we obtain the second relation for $\omega_{2n}$.
	
	Note that there is a cyclic symmetry in the Kirby color by rotation of the boundary curves. This introduces another annulus capping property
	\begin{align*}
		\mathrm{Ann}^{(1,2n)}_{1^\vee}(\omega_{2n})=\omega_{2n-2}\quad \quad \text{and}\quad \quad  \mathrm{Ann}^{1,2n}_{b_i}(\omega_{2n}) = 0,\quad \text{for } 2\leq i\leq N.
	\end{align*}
		
	We expect that for a given strongly separable, commutative Frobenius algebra $A$ such that the surface skein theory extends to 2-handles, one can determine the Kirby color recursively using the above properties after having determined the structure of idempotents in $\dTL_A(2n,2n)$. We do not pursue this approach here. Instead, we will combinatorially compute the Kirby color for surface skein theory associated to the Frobenius algebra $A_\alpha$ in the next section. We do, however, establish these properties for the Kirby color $\omega_{2n}$ in the case of $A_\alpha$ in Corollary \ref{cor:annulusCapping} explicitly.
\end{rmk}

\section{Computation of the Kirby color}\label{sec:KirbyCol}
\subsection{Idempotents in the disk category}
We recall the definition of the Bar-Natan category $BN_A(\Sigma,\mathbf{p})$ of a surface $\Sigma$ with boundary points $\mathbf{p}\subset \Sigma$ from \cite[Definition 3.6]{HRWKirbyColorForKh} and \cite[Definition 2.4]{HRWBorderedInvariantsKh} which relies on the original construction in \cite{BarNKhoHomTangles}. Our version here is not graded and takes a general commutative Frobenius algebra $A$ as an input. Note that before, we have only allowed closed surfaces in the definition of the surface skein category in Definition \ref{def:skeincat}. 
\begin{defn}
	Let $\Sigma$ be an oriented surface with boundary $\partial\Sigma$, and let $\mathbf{p}\subset \partial\Sigma$ be a finite set of points in the boundary. Define the \emph{Bar-Natan category} $BN_A(\Sigma,\mathbf{p})$ as the $\kk$-linear category that has 
    \begin{itemize}
        \item \textbf{Objects:} 1-dimensional properly embedded 1-manifolds bounding $\mathbf{p}$.
        \item \textbf{Morphisms:} The $\kk$-module of morphisms $\Hom_{\BN_A(\Sigma,\mathbf{p})}(c,d)$ is spanned by $A$-decorated cobordisms with corners\footnote{The definition of decorated surfaces in Definition \ref{def:decSurface} readily extends to the case with corners.} properly embedded in $\Sigma\times [0,1]$ with boundary 
    \begin{align*}
        \mathbf{p}\times [0,1]\cup_{\mathbf{p}\times\{0,1\}} (c\times \{0\} \sqcup d\times \{1\}) \subset \partial \Sigma\times [0,1] \cup_{\partial \Sigma\times \{0,1\}} \Sigma \times \{0,1\}.
    \end{align*}
    The decorated cobordisms are considered up to skein relations as in Definition \ref{def:skeinmod}: the decorations are $\kk$-linear in each connected component, and the sphere and neck-cutting relations are imposed.
        \item \textbf{Composition:} As in Definition \ref{def:skeincat}, composition is induced by gluing along common boundary, multiplying the decorations and rescaling in the interval direction.
    \end{itemize}
\end{defn}

In the following, we write $BN_A(B^2,2n)$ for the Bar-Natan category of the disk with $\mathbf{p}$ given by $2n$ points embedded in the boundary of the disk. Denote by $\BN_A(B^2,2n)$ the full subcategory of $BN_A(B^2,2n)$ with objects given by \emph{planar matchings} embedded in $B^2$. There are $C_n=\tfrac{1}{n+1}\binom{2n}{n}$ isomorphism classes of objects in $\BN_A(B^2,2n)$.

We also write $\overline{\mathsf{BN}}_A(B^2, 2n) = \Kar(\BN_A(B^2,2n))$ for the Karoubi completion of $\BN_A(B^2,2n)$. When $2n$ is clear from context, we sometimes abbreviate $\BN_A=\BN_A(B^2,2n)$ and $\overline{\BN}_A = \overline{\BN}_A(B^2,2n)$. 

\begin{prop}[Properties of the disk category]\label{prop:diskCat}
	Let $M,N$ be objects in $\BN_A(B^2,2n)$, that is, planar matchings in $B^2$ with $2n$ endpoints. Then the following holds.
	\begin{enumerate}[(i)]
		\item The endomorphism algebra is $\End_{\BN_A}(M)\cong A^{\otimes\pi_0(M)}$. In particular, it is commutative.
		\item Every morphism $\varphi\in \Hom_{\BN_A}(M,N)$ is of the form \begin{align*}
			\varphi = g\circ s_k\circ \dots \circ s_1 \circ f
		\end{align*} 
		with $f\in \End_{\BN_A}(M)$ and $g\in \End_{\BN_A}(N)$ and saddles $s_i$, i.e.\ a cobordism of Morse index 1, locally surgering two arcs.
	\end{enumerate}
\end{prop}
\begin{proof}
	For \textit{(i)}, we identify $B^2\times I \cong B^3$ and note that the boundary conditions determined by $M$ consist of closed curves enumerated by $\pi_0(M)$. Analogous to Example \ref{exm:FirstSkeinModules}.(i), we obtain the desired isomorphism.
    For \textit{(ii)}, we use that $\Hom_{\BN_A}(M,N)$ is isomorphic to $\Sk_A(B^3, c)$ for a closed 1-manifold $c$ obtained from gluing $M$ and $N$. By Example \ref{exm:FirstSkeinModules}.(i), any morphisms in $\Hom_{\BN_A}(M,N)$ can be represented by a configuration of decorated disks. These can be written as a sequence $s_1,\dots, s_k$ of saddles, with decorations moved to the source or target and recorded as endomorphisms $f\in \End(M)$ or $g\in \End(N)$.
\end{proof}
As a direct consequence from \textit{(ii)}, we obtain the following.
\begin{cor}\label{cor:diskCatKar}
	Let $M,N$ be objects in $\BN_A(B^2,2n)$. Let $e_M\in \End_{\BN_A}(M)$ and $e_N\in\End_{\BN_A}(N)$ be idempotents. Every morphism $\psi\in \Hom_{\overline{\BN}_A}(e_M,e_N)$ is of the form \begin{align*}
		\psi = g\circ s_k\circ \dots \circ s_1 \circ f
	\end{align*} 
	with $f\in \End_{\overline{\BN}_A}(e_M)$ and $g\in \End_{\overline{\BN}_A}(e_N)$ and saddles $s_i$.
\end{cor}
Recall the notion of the trace of a $\kk$-linear category.
\begin{defn}
	Let $\C$ be a $\kk$-linear category. The \emph{trace} of $\C$ is the $\kk$-module
	\begin{align*}
		\Tr(\C) = \left(\bigoplus_{x\in \mathsf{Ob}(\C)} \End(x)\right)\big/ \mathrm{span}_\kk \{f\circ g - g\circ f \mid f\colon x\to y,g\colon y\to x\}.
	\end{align*}
\end{defn}
\begin{lem}\label{lem:TrOfKar}
	Let $\C$ be a $\kk$-linear category. Then the inclusion $\C\to \Kar(\C)$ of $\C$ into its Karoubi completion $\Kar(\C)$ induces an isomorphism of $\kk$-modules
	\begin{align*}
		\Tr(\C) \cong \Tr(\Kar(\C)).
	\end{align*}
\end{lem}
\begin{proof}
	The proof can be found in \cite[Proposition 3.2]{BGHATraceAltDecat}.
\end{proof}
Tracing the disk category, we obtain the skein module of the solid torus.
\begin{prop}\label{prop:skmodIdempotents}
	Let $A$ be a commutative Frobenius algebra that is free over a commutative ring $\kk$. Then, gluing provides an isomorphism of $\kk$-modules
	\begin{align*}
		\Sk_A(\mathbb{H},(2n,0))\cong \Tr(BN_A(B^2,2n))\cong\Tr(\BN_A(B^2,2n))\cong \Tr(\overline{\mathsf{BN}}_A(B^2,2n)).
	\end{align*}
\end{prop}
\begin{proof}
	The first isomorphism follows from gluing formulae for general skein theories. See \cite[Section 4.1]{WalkerUnivState} and \cite[Section 1.2]{WalkerTQFT06}. The second isomorphism follows from the fact that we only need to consider gluing over boundary conditions that do not contain any closed circle components. This follows from considerations analogous to those in Remark \ref{rmk:gluingMorita}. The third isomorphism follows from Lemma \ref{lem:TrOfKar}.
\end{proof}
\begin{rmk}\label{rmk:skeletal}
	It suffices to restrict to a skeletal version of $\overline{\BN}_A$ when computing the trace. 
	In view of Corollary \ref{cor:diskCatKar} and  Proposition \ref{prop:skmodIdempotents}, we can describe the skein module of the solid torus via isomorphism classes of idempotents of the disk category. For the Frobenius algebra $A_\alpha$ which we consider in Subsection~\ref{subsec:alphaThy}, we construct idempotents $e\in\End(M)$ and $f\in\End(N)$ for planar matchings $M$ and $N$, such that 
	\begin{align*}\Hom_{\overline{\BN}_{\alpha}}(e,f) = 0 \quad \text{ for } e\not\cong f.
	\end{align*}
	In this case the skein module of the solid torus decomposes as the direct sum of $\End(e)$ over the isomorphism classes of those idempotents in the disk category. See Corollary \ref{cor:skmodViaIsoClasses}.
\end{rmk}
\begin{rmk}[Arc algebras]
	Since additive completion also does not change the trace, the above also computes the trace of  Khovanov arc algebras generalized to commutative Frobenius algebras $A$.
\end{rmk}

\subsection{Modified Bar-Natan skein theory}\label{subsec:alphaThy}
We now consider the $\alpha$-modified Bar-Natan Frobenius algebra $A_\alpha=\kk[x]/(x^2-\alpha)$ over $\kk=K[\alpha^{\pm1}]$ with $\Delta\colon 1\mapsto 1\otimes x+x\otimes 1$. We assume $2\in \kk^\times$. As usual, we denote the decoration by $x\in A_\alpha$ as a dot. The skein relations are 
\begin{align*}
	\ddsheet\, = \alpha\; \sheet && \neck\,= \cutdup + \cutddown&& 
	\sphere\, = 0 && \dsphere\, = 1.
\end{align*}
We compute the skein module $\Sk_\alpha(\mathbb{H},(2n,0))$ from the idempotents of the disk category, and the Kirby color $\omega_{2n}$ to model the 2-handle attachment. Under the isomorphism $\dTL_\alpha(0,2n)\cong \Sk_\alpha(\mathbb{H},(2n,0))$, we sometimes think of elements in the skein module of the solid torus as $\dTL_{\alpha}$-diagrams. We have the following local relations.
\begin{align*}
	\dTLCircle = \varepsilon(H)=\varepsilon(2x)=2 \quad &\text{and} \quad\dTLCircleDot=\varepsilon(2x^2)= \varepsilon(2\alpha)=0\\
	\LocaldTLdud\, +\, \LocaldTLudd\, &=\, \LocaldTLdupperudlower\, +\, \LocaldTLudupperdlower\\
	\alpha\;\LocaldTLudud\, + \, \LocaldTLdd\, &= \alpha \; \LocaldTLudud + \LocaldTLdupperdlower
\end{align*}

Let $M$ be an object in $\BN_{\alpha}(B^2,2n)$, i.e.\ a planar matching on $2n$ points in the disk. Write $\arcs(M)$ for the set of $n$ arcs of $M$. A partition $P$ of $\arcs(M)$ induces a partition of the set of boundary points by taking the union over all arcs in the same block. We refer to this as the \emph{boundary partition} and denote it by $\partial P$. The orientation on $B^2$ induces a cyclic order on the boundary points in $\partial B^2$. A choice of a distinguished point $1\subset \partial B^2$ induces a total order on the set of boundary points which we identify with $\{1,\dots, 2n\}$. With this identification, the arcs of a planar matching are encoded by a partition of the set $\{1,\dots, 2n\}$ into blocks $\{i,j\}$ of size 2. In the following, we often make the choice of a point $1\subset \partial B^2$, but our statements will be independent of this choice.

For a subset $D\subseteq \arcs(M)$, we write $x_D$ for the element in $\End_{\BN_\alpha}(M)$ that has dotted identity sheets on arcs in $D$ and identity sheets on other arcs. We prove the following statements after constructing and establishing several properties of the idempotents.
\begin{prop}[Structure of the idempotents]\label{prop:structureOfIdem}
	For a planar matching $M$ in $\BN_{\alpha}(B^2,2n)$, each partition $P=\{B,C\}$ of $\arcs(M)$ into blocks $B,C\subseteq \arcs(M)$, with $C=\emptyset$ allowed, defines an idempotent
	\begin{align*}
		e_P = \frac{1}{2^{n-1}}\sum_{\substack{D\subseteq \arcs(M)\\|D|\text{ even}}}  (-1)^{|D\cap B|}\alpha^{-|D|/2}  x_D
	\end{align*}
	in $R=\End_{\BN_{\alpha}}(M)$ satisfying the following properties.
	\begin{enumerate}[(i)]
		\item (Orthogonality) $e_PRe_Q=0=e_QRe_P$ for partitions $P\neq Q$.
		\item (Completeness) The set of idempotents for partitions $P$ satisfies $\sum_{P} e_P = \id_M$.
		\item ($A_\alpha$-Primitivity) $\End_{\overline{\BN}_\alpha}(e_P) \cong A_\alpha$.
	\end{enumerate}
\end{prop}

\begin{prop}[Isomorphism classes of idempotents]\label{prop:isoClassesOfIdem}
	Let $M$ and $N$ be planar matchings on $2n$ points. Two idempotents $e_P \in \End(M)$ and $e_Q\in \End(N)$ associated to partitions $P$ of $\arcs(M)$ and $Q$ of $\arcs(N)$ are isomorphic in the Karoubi completion if and only if the induced boundary partitions are the same. The isomorphism classes of primitive idempotents are in one-to-one correspondence with returning walks on $\ZZ$ that start and end at 0, have length $2n$, and start with a step in positive direction.
\end{prop}
\begin{cor}\label{cor:isoClassesWalks}
	The set of returning walks on $\ZZ$ of length $2n$ starting and ending at 0 indexes a basis for $\Sk_\alpha(\mathbb{H}, (2n,0))\cong \Tr(\overline{\BN}_\alpha(2n))$. In particular, the skein module $\Sk_\alpha(\mathbb{H}, (2n,0))$ has rank $\binom{2n}{n}$ over $\kk=K[\alpha^{\pm1}]$.
\end{cor}

In the following, we prove Proposition \ref{prop:structureOfIdem}. The separability idempotent $\Delta(H^{-1})=\frac{1}{2}(1\otimes 1+\frac{1}{\alpha}x\otimes x)\in A_\alpha^{\otimes 2}$ will play a central role in constructing the idempotents of $\End(M)\cong A_\alpha^{\otimes n}$.
\begin{con}[Join and disjoin]\label{con:joindisjoin}
	First consider the case $n=1$. There is a unique planar matching $M$ on $2$ points. The endomorphism algebra $\End(M)\cong A_\alpha$ has the idempotent $1\in A_\alpha$ which corresponds to the identity sheet on the unique arc. The identity does not decompose further\footnote{We assume that $\sqrt{\alpha}\notin K$. Evaluating $\alpha\mapsto 1$, we recover Lee theory where the identity decomposes into root projectors. See also Remark \ref{rmk:genSemisimpleFA} below.}. In the following, let $n\geq 2$. Let $M$ be a planar matching on $2n$ points and let $b,c\in \arcs(M)$. We use the following notation for elements in $\End(M)\cong A_\alpha^{\otimes n}$. We write $x_b$ for the element in $\End(M)$ given by the dotted sheet on $b$ and the identity on the other arcs. Under the identification $\End(M)\cong A_\alpha^{\otimes n}$, we have that $x_b$ is the elementary tensor with $x\in A_\alpha$ in the factor corresponding to the arc $b$ and $1\in A_\alpha$ on the other factors. For a product, we use the multiplication in the algebra $A_\alpha^{\otimes n}$.
    Define the \emph{join} and the \emph{disjoin} of arcs $b$ and $c$ as the endomorphisms
    \begin{align*}
        j(b,c) =  \tfrac{1}{2} \left( \id + \tfrac{1}{\alpha} x_b x_c\right)\in \End(M) \quad\quad \text{and} \quad\quad d(b,c) = \tfrac{1}{2} \left( \id - \tfrac{1}{\alpha} x_b x_c\right)\in \End(M)
    \end{align*}
	respectively. Indeed, the join and disjoin are idempotents since for the doubly dotted sheet on an arc $b$, we have $x_b^2=\alpha \cdot \id \in \End(M)$. The join $j(a,b)$ is the separability idempotent in the factor $A_\alpha^{\otimes2}$ that corresponds to the arcs $a$ and $b$ and the identity on the other arcs. Observe that the join and the disjoin are homogeneous, if we endow $A_\alpha$ with a grading with $\deg(1)=0$, $\deg(x)=2$ and $\deg(\alpha)=4$.
\end{con}

\begin{lem}\label{lem:join}
	Let $n\geq 2$ and let $M$ be a planar matching on $2n$ points.
	\begin{enumerate}[(i)]
		\item Let $b$ and $c$ be arcs of $M$. Then, $j(b,c)d(b,c)=0$.
		\item Let $B=\{b_1,\dots, b_k\}\subset \arcs(M)$ be a set of $k$ arcs. Then, 
		\begin{align*}
			j(B) := \prod_{i=1}^{k-1} j(b_i,b_{i+1}) = \frac{1}{2^{k-1}}\sum_{\substack{D\subseteq B\\|D|\text{ even}}} \alpha^{-|D|/2}  x_D
		\end{align*}		
        is independent of the enumeration of the arcs in $B$ and $j(B)j(b,b')=j(B)$ for $b,b'\in B$. 
		\item Let $B$ be a set of arcs with $b,b'\in B$ and let $c\in \arcs(M)\setminus B$. Then $j(B)d(b,c)=j(B)d(b',c)$.
	\end{enumerate}
\end{lem}
\begin{proof}
	We obtain \textit{(i)} directly from the definition. For \textit{(ii)}, observe that a product of $k-1$ joins is a sum of $2^{k-1}$ terms each of which has an even number of dots, and each of the even dottings occurs exactly once.  The factors of $1/\alpha$ follow from the homogeneity of the join, the factor $1/2^{k-1}$ is clear. The expression is indeed independent of the enumeration of arcs in $B$. In particular we can choose $b_{k-1}=b$ and $b_{k}=b'$. This shows \textit{(ii)}. Now, let $b,b'\in B$ and $c\in \arcs(M)\setminus B$. Then $j(b,b')d(b,c)=j(b,b')d(b',c)$. By \textit{(ii)}, we have $j(B)=j(B)j(b,b')$ and obtain \textit{(iii)}.
\end{proof}
With Lemma \ref{lem:join}, we can think of joined arcs as belonging to the same block and disjoined arcs belonging to different blocks. To join a set of $k$ arcs, we need precisely $k-1$ joins. Next, we show that we can at most choose two blocks of arcs. For this, we observe the following properties.
\begin{lem}\label{lem:disj}
	Let $n\geq 2$ and let $M$ be a planar matching on $2n$ points. 
	\begin{enumerate}[(i)]
		\item Let $a,b,c$ be arcs of $M$. Then, $d(a,b)d(b,c)d(c,a)=0$. \label{sub:tripledisj}
		\item Let $e$ be an idempotent for $M$ given by partition into three blocks. That is,
		\begin{align*}
			e = j(A)j(B)j(C)d(a,b)d(b',c)d(c',a')
		\end{align*}
		for disjoint sets of arcs $A,B,C$ with $a,a'\in A$, $b,b'\in B$ and $c,c'\in C$. Then, $e=0$.
		\item Let $\{B,C\}$ be a partition of $\arcs(M)$ with $C\neq \emptyset$ allowed. Fix arcs $b\in B$ and $c\in C$. Then,
		\begin{align*}
			j(B)j(C)d(b,c) = \frac{1}{2^{n-1}}\sum_{\substack{D\subseteq \arcs(M)\\|D|\text{ even}}} (-1)^{|B\cap D|}\alpha^{-|D|/2} x_D.
		\end{align*}
	\end{enumerate}
\end{lem}
\begin{proof}
	For \textit{(i)} we compute
    \begin{align*}
        d(a,b)d(b,c)d(c,a) =&\ \frac{1}{8} \left( \id - \tfrac{1}{\alpha} x_ax_b \right) \left( \id - \tfrac{1}{\alpha} x_bx_c \right)  \left( \id - \tfrac{1}{\alpha} x_ax_c \right) \\
        =&\  \frac{1}{8} \left( \id - \tfrac{1}{\alpha} x_ax_b - \tfrac{1}{\alpha} x_b x_c +\tfrac{1}{\alpha} x_a x_c - \tfrac{1}{\alpha} x_a x_c + \tfrac{1}{\alpha} x_b x_c + \tfrac{1}{\alpha} x_a x_b - \tfrac{1}{\alpha^2} x_a^2 x_c^2\right)=0.
    \end{align*}
	For \textit{(ii)} we apply Lemma \ref{lem:join}.\textit{(iii)} twice and obtain by \textit{(i)}
	\begin{align*}
		e = j(A)j(B)j(C)d(a,b)d(b',c)d(c',a') = j(A)j(B)j(C)d(a,b)d(b,c)d(c,a) = 0.
	\end{align*}
	Next, we show \textit{(iii)}. We proceed similarly to the case $C=\emptyset$ in Lemma \ref{lem:join}.\textit{(ii)}. Joins and disjoins only differ by a sign. Hence, for $C\neq \emptyset$, we obtain the same expression up to sign. Exactly half of the terms carry a minus sign. Those are the terms which have an odd number of dots on $B$ and on $C$. We obtain that the sign is determined by $|B\cap D|\equiv |C\cap D| \mod 2$.
\end{proof}
As a direct consequence, we obtain the following.
\begin{cor}\label{cor:orthog}
    Let $n\geq 2$ and let $M$ be a planar matching on $2n$ points. Consider two different partitions $P=\{B_1,C_1\}$ and $Q=\{B_2,C_2\}$ of the set $\arcs(M)$ with $C_i$ allowed to be empty. Then the associated idempotents $e_P,e_Q\in \End(M)$ satisfy
    \begin{align*}
        e_P R e_Q=0=e_Q R e_P
    \end{align*}
    with $R=\End(M)$.
\end{cor}
\begin{proof}
    By a similar computation as in Lemma \ref{lem:disj}.\textit{(ii)}, it follows that $e_Pe_Q=0=e_Qe_P$. By commutativity of $\End(M)$ we obtain the statement.
\end{proof}

\begin{lem} \label{lem:decompId}
	Let $n\geq 2$ and let $M$ be a planar matching on $2n$ points. Then, $\id_M = \sum_P e_P$ in $\End_{\BN_{\alpha}}(M)$, where the sum is taken over partitions of $\arcs(M)$ into exactly one or exactly two blocks.
\end{lem}
\begin{proof}
	Write $\arcs(M)=\{b_1,\dots, b_n\}$. Then we can expand
	$\id_M = \prod_{i,j} \bigl[j(b_i,b_j)+ d(b_i,b_j)\bigr]$ since every factor is $\id_M$. Evaluation of the product yields the sum of all possible combinations of joins and disjoins. By Lemma~ \ref{lem:disj}, all terms that do not correspond to a partition of the arcs into one or two blocks are zero. The non-zero terms reduce to products of $n-1$ joins and disjoins by idempotency.
\end{proof}

For $b$ an arc of a planar matching $M$ and an idempotent $e\in \End(M)$, write 
\begin{align*}
    \mathrm{dot}_b(e) = x_b e = e x_b \in \End(M).
\end{align*}
\begin{lem}\label{lem:primitive}
	Let $M$ be a planar matching and let $e$ be the idempotent associated to a partition of the arcs of $M$ into blocks $B,C$ with $C=\emptyset$ allowed. Then
	\begin{enumerate}[(i)]
		\item Let $b$ and $b'$ be arcs in the same block. Then, $\mathrm{dot}_b(e) = \mathrm{dot}_{b'}(e)$.
		\item If $b\in B$ and $c\in C$, then $\mathrm{dot}_b(e)=-\mathrm{dot}_c(e)$.
		\item In the Karoubi completion, $\End_{\overline{\BN}_{\alpha}}(e)\cong A_\alpha$. That is, the idempotent $e$ is $A_\alpha$-primitive.
	\end{enumerate}
\end{lem}
\begin{proof}
	Write $e=j(B)j(C)d(b,c)$. For \textit{(i)} we assume without loss of generality that $C\neq \emptyset$ and $b,b'\in B$. Observe that
	\begin{align*}
        \mathrm{dot}_b(e) &= x_b e = x_b j(b,b') e = \tfrac{1}{2} (x_b +x_{b'}) e
	\end{align*}
	is symmetric in $b$ and $b'$. Similarly for $b\in B$ and $c\in C$,
	\begin{align*}
        \mathrm{dot}_b(e) &= x_b e = x_b d(b,c) e = \tfrac{1}{2} (x_b -x_{c}) e
	\end{align*}
	is antisymmetric in $b$ and $c$. This shows \textit{(ii)}. For \textit{(iii)}, we note that an endomorphism $f\in \End_{\overline{\BN}_{\alpha}}(e)$ is an endomorphism $f\in \End(M)$ with $ef=f$. It is a $K[\alpha^{\pm1}]$-linear combination of products of $\mathrm{dot}_b(e)$. By \textit{(i)} and \textit{(ii)}, this reduces to $f=\mu\cdot e+\nu\cdot\mathrm{dot}_b(e)$ for some $\mu,\nu\in K[\alpha^{\pm1}]$ and some $b\in B$.
\end{proof}
We are now ready to prove the following.
\begin{proof}[Proof of Proposition \ref{prop:structureOfIdem}]
	First consider $n\geq 2$. The general form has been shown in Lemma \ref{lem:disj}. Orthogonality is Corollary \ref{cor:orthog}, completeness is shown in Lemma \ref{lem:decompId}, and $A_\alpha$-primitivity is Lemma \ref{lem:primitive}.\textit{(iii)}. For $n=1$, we identify the identity as the idempotent associated to the partition with the unique arc in block $B$ and $C=\emptyset$. The general form, orthogonality and completeness are clear. $A_\alpha$-primitivity follows as for $n\geq2$.
\end{proof}
\begin{rmk}
	Consider $2n$ boundary points with a choice of boundary point $1 \subset \partial B^2$. For a planar matching $M$, let $b_0$ be the arc containing the boundary point 1. By Lemma \ref{lem:primitive}, for every arc $b\in \arcs(M)$, we have $\mathrm{dot}_b(e) = \pm \mathrm{dot}_{b_0}(e)$. In the following, we write $\dot{e} = \mathrm{dot}_{b_0}(e)$ and always assume that partitions $P=\{B,C\}$ of $\arcs(M)$ (with $C=\emptyset$ allowed) satisfy $b_0\in B$.
\end{rmk}
Following Remark \ref{rmk:skeletal}, we pass to a skeletonization of $\overline{\BN}_\alpha(B^2,2n)$ and denote the isomorphism class of an idempotent $e$ by $\mathbf{e}$. With a fixed choice of boundary point $1\subset \partial B^2$, we write $\{\mathbf{e},\mathbf{\dot{e}}\}$ for the basis of $\End(\mathbf{e})$ induced by $\{e,\dot{e}\}$. In the following, we prove Proposition \ref{prop:isoClassesOfIdem} and Corollary \ref{cor:isoClassesWalks}.
\begin{lem}\label{lem:saddlesjoindisjoin}
	Consider planar matchings $M$ and $N$ differing by exactly two arcs  $a_1,a_2\in \arcs(M)$ and $b_1,b_2\in \arcs(N)$. Write $s$ for the saddle from $M$ to $N$ changing $(a_1,a_2)$ to $(b_1,b_2)$. Let $c,c'\in \arcs(M)\cap\arcs(N)$ be common arcs of $M$ and $N$. Then saddles, joins and disjoins satisfy the following.
	\begin{enumerate}[(i)]
		\item $j(b_1,b_2)\circ s=s=s \circ j(a_1,a_2)$. 
		\item $d(b_1,b_2)\circ s=0=s \circ d(a_1,a_2)$.
		\item $j(c,c')\circ s=s\circ j(c,c')$ and $d(c,c')\circ s=s\circ d(c,c')$
		\item $j(b_1,c) \circ s = j(b_2,c)\circ s = s\circ j(a_1,c) = s\circ j(a_2,c)$ 
		\item $d(b_1,c) \circ s = d(b_2,c)\circ s = s\circ d(a_1,c) = s\circ d(a_2,c)$
		\item Denote by $\overline{s}$ the reverse saddle decorated by $H^{-1}=\frac{1}{2\alpha}x$. Then, $\overline{s}\circ s=j(a_1,a_2)$ and $s\circ \overline{s}=j(b_1,b_2)$.
	\end{enumerate}	
\end{lem}
\begin{proof}
	We will see that saddles behave similarly to joins. Indeed, since dots can move through saddles and saddles are connected, the dot on $a_1$ in $j(a_1,a_2)$ can be moved to $a_2$ and canceled for a factor of $\alpha$. This yields $s\circ j(a_1,a_2)=s$ and similarly $j(b_1,b_2)\circ s = s$. Now, using \textit{(i)}, we obtain \textit{(ii)} directly from Lemma \ref{lem:join}.\textit{(i)}. Since $s$ is the identity on $c$ and $c'$, we immediately obtain \textit{(iii)}. Again by applying \textit{(i)}, we obtain \textit{(iv)} and \textit{(v)} from Lemma \ref{lem:join}. For \textit{(vi)}, observe that the composition of a saddle and its reverse create a tube that can be neck-cut. With decoration by $H^{-1}$, this precisely yields the separability idempotent in the corresponding $A_\alpha^{\otimes 2}$ and the identity everywhere else. This is equal to the join.
\end{proof}

Let $M$ and $N$ be planar matchings as in Lemma \ref{lem:saddlesjoindisjoin} with a saddle $s$ from $M$ to $N$. Let $P$ and $Q$ be partitions of $\arcs(M)$ and $\arcs(N)$ respectively, into one or two blocks. We say that $s$ \emph{respects the partitions} $P$ and $Q$ if the partitions differ precisely by $(a_1,a_2)$ and $(b_1,b_2)$. That is, for $P=\{A_1,A_2\}$ and $Q=\{B_1,B_2\}$, $A_1\setminus \{a_1,a_2\} = B_1\setminus \{b_1,b_2\}$ and $A_2=B_2$. A sequence of saddles respects the partitions if each saddle does.

\begin{lem}\label{lem:isoseqofsaddles}
	Let $n\geq 2$ and let $M$ and $N$ be planar matchings on $2n$ points. Let $P$ and $Q$ be partitions of $\arcs(M)$ and $\arcs(N)$ respectively, into one or two blocks. Then the idempotents $e_P$ and $e_Q$ satisfy the following properties.
	\begin{enumerate}[(i)]
		\item There is an isomorphism $e_P \cong e_Q$ if and only if there exists a sequence of saddles $s=s_m\circ \dots \circ s_1$ from $M$ to $N$ respecting the partitions $P$ and $Q$. In this case, an isomorphism is given by $\varphi= e_Q \circ s\circ e_P$ with inverse $\varphi^{-1} = e_P \circ \overline{s} \circ e_Q$ where $\overline{s}$ is the reversed saddle sequence, each decorated by $H^{-1}$.
		\item There is an isomorphism $e_P \cong e_Q$ if and only if their boundary partitions are the same $\partial P=\partial Q$.
		\item $\Hom_{\overline{\BN}_{\alpha}}(e_P, e_Q) \cong\begin{cases}
			A_\alpha & \text{if }e_P\cong e_Q\\
			0&\text{else.}
		\end{cases} $ 
	\end{enumerate}
\end{lem}
\begin{proof}
	Using Corollary \ref{cor:diskCatKar}, by moving all decorations, we can assume that any morphism $\varphi\in \Hom_{\overline{\BN}_{\alpha}}(e_P,e_Q)$ is up to decorations of the form $e_Q\circ s \circ e_P$ for a sequence $s$ of saddles. If $s$ respects the partitions, by applying Lemma \ref{lem:saddlesjoindisjoin}.\textit{(vi)} iteratively and moving the decorations to $\End(e_P)$ and $\End(e_Q)$, we obtain that $\varphi=e_Q\circ s \circ e_P$ is an isomorphism. If $s$ does not respect the partitions, the morphism is zero by Lemma \ref{lem:saddlesjoindisjoin}.\textit{(ii)}. This shows \textit{(i)}. For \textit{(ii)}, we note that by considering the thickened disk as a 3-ball that has closed circle components in $\partial B^3$, it follows that there exists a sequence of saddles between $e_P$ and $e_Q$ respects the partition if and only if the boundary partitions $\partial P$ and $\partial Q$ are the same.
    To prove \textit{(iii)}, we again use that by Corollary \ref{cor:diskCatKar}, every morphism can be represented by a sequence of saddles and decorations. If the saddles do not respect the boundary partitions, the morphism is zero as above. If they do respect the boundary partitions, $e_P\cong e_Q$ by \textit{(ii)}. Hence, we only need to show that for $e_P\cong e_Q$, we have $\Hom_{\overline{\BN}_\alpha}(e_P,e_Q)\cong A_\alpha$. Indeed, by Lemma \ref{lem:primitive}, $\End(e_P)\cong A_\alpha$ with basis $\{e_P, \dot{e}_P\}$ (after a choice of boundary point 1) and we can move all decorations either to $e_P$ or $e_Q$. 
\end{proof}
Using Lemma \ref{lem:isoseqofsaddles}.\textit{(iii)} and passing to a skeletonization $\widehat{\BN}_\alpha(B^2,2n)$ of the category $\overline{\BN}_\alpha(B^2,2n)$, we obtain the following consequence.
\begin{cor}\label{cor:skmodViaIsoClasses}
	The isomorphism from Proposition \ref{prop:skmodIdempotents} becomes
	\begin{align*}
		\Sk_\alpha(\mathbb{H},(2n,0))\cong \Tr(\widehat{\BN}_\alpha(B^2,2n)) = \bigoplus_{\mathbf{e}}\End(\mathbf{e})
	\end{align*}
	where $\End(\mathbf{e})\cong\End(e)\cong A_\alpha$ for a representative $e$ of the class $\mathbf{e}$.
\end{cor}

Fix a boundary point $1\subset \partial B^2$ and consider the totally ordered set of boundary points $\{1,\dots, 2n\}$. Let $M$ be a planar matching on these $2n$ points, $P$ a partition of $\arcs(M)$ into exactly one or exactly two blocks, and let $e_P\in \End(M)$ be the idempotent associated to $P$. Then, by Lemma~\ref{lem:isoseqofsaddles}.\textit{(ii)}, the isomorphism class $\mathbf{e}_P$ is determined by the induced boundary partition $\partial P$. Such a boundary partition can be encoded as a binary sequence $b=(0,b_2,\dots, b_{2n})\in (\ZZ/2)^{2n}$ with $b_i\in\ZZ/2$ labeling whether the $i$-th boundary point is in block 0 or 1. Recall that by convention, we always assume that the boundary point 1 is in block 0. We call such $b$ \emph{boundary partition sequences}.
We denote the set of all boundary partition sequences by $\mathcal{B}_{2n}$.

Consider walks on $\ZZ$ of length $2n$ starting in 0. If a walk is ending in 0, we call it \emph{returning}. We encode a returning walk on $\ZZ$ starting and ending in $0$ by a binary sequence that we call \emph{walk sequence}. Denote by $\mathcal{W}_{2n}$ the set of all walk sequences of length $2n$. Steps in positive direction are encoded by $0$, steps in negative direction by $1$. The set $\mathcal{W}_{2n}$ has $\binom{2n}{n}$ elements. Denote by $\mathcal{W}^+_{2n}$ the subset of those walks starting with the step $0\mapsto 1$, encoded by sequences starting with $0$. 

Denote by $s=(0,1,0,1,\dots, 0,1)\in (\ZZ/2)^{2n}$ the \emph{alternating sequence} of length $2n$.

\begin{lem}\label{lem:walksOnZ}
	Adding the alternating sequence $s$ defines a bijection $\mathcal{W}^+_{2n}\cong \mathcal{B}_{2n}$.
\end{lem}
\begin{proof}
	Let $b\in \mathcal{B}_{2n}$ be a boundary partition sequence associated to an isomorphism class $\mathbf{e}$ represented by an idempotent $e_P\in \End_{\BN_{\alpha}}(M)$ for some planar matching $M$ and arc partition $P$. By adding the alternating sequence $s$ to $b$, this defines the walk sequence $w$ of a walk starting with a step in positive direction since by convention $b_1=0=w_1$. We need to show that this walk is returning. Every planar matching on $2n$ points can be reduced to a planar matching on $2(n-1)$ points by removing an arc joining two adjacent points. Consequently, removing subsequences $(0,0)$ and $(1,1)$ from $b$ corresponds under alternation to removing $(0,1)$ or $(1,0)$ from $w$. After removing $n$ arcs from $b$, this procedure terminates with the empty sequence. For $w$, we obtain that the walk must return since we removed an equal number of ones and zeros. 
	
	Conversely, let $w\in \mathcal{W}^+_{2n}$ be walk sequence of a returning walk of length $2n$ starting with a step in positive direction, and let $\overline{w}=w+s$ be its alternation. We construct a planar matching $M$ and an idempotent $e_P\in\End_{\BN_{\alpha}}(M)$ with boundary partition sequence $\overline{w}$. Pick subsequences $(0,1)$ or $(1,0)$ to remove from $w$. This creates a returning walk of length $2(n-1)$ (not necessarily starting with a step in positive direction). This corresponds to removing $(0,0)$ or $(1,1)$ in $\overline{w}$ as above. Since $w$ has the same number of 0s and 1s, we can remove $n$ such subsequences and this process terminates with the empty walk sequence. For $\overline{w}$ this removes $n$ subsequences $(0,0)$ or $(1,1)$ and $\overline{w}$ can be realized as a boundary partition sequence. Now construct $M$ and the partition $P$ as follows. Draw an arc $\{i,j\}$ if the subsequence $(\overline{w}(i),\overline{w}(j))$ in $(\overline{w}(1),\dots, \overline{w}(i),\dots, \overline{w}(j), \dots, \overline{w}(2n))$
	was removed in the process. The arc is partitioned in the same block as the boundary point 1 if and only if $(\overline{w}(i),\overline{w}(j))=(0,0)$. This yields an idempotent $e_P\in \End(M)$ given by the partition $P$ and has boundary partition sequence $\overline{w}$. 

	The two constructions above are clearly mutually inverse. While the choice of planar matching is not unique, the boundary partition sequence and the walk sequence are unique.
\end{proof}
\begin{rmk}\label{rmk:altWalkSeq}
   One model for walks on $\ZZ$ is provided by considering $\ZZ$ as a graph with vertices $m\in\ZZ$ and edges $\{m,m+1\}$, making it into a 2-regular\footnote{A graph is called $r$-regular if every vertex has degree $r$. An $r$-regular tree with $r\geq2$ is infinite.} tree. A walk on a graph is a sequence of edges such that for each pair of consecutive edges, the edges are adjacent. In the case of $\ZZ$, for every step there are exactly two choices: a step in negative or a step in positive direction as above. Given the starting vertex $0\in \ZZ$, a walk of length $2n$ is encoded in the walk sequence $w\in \mathcal{W}_{2n}$.
   
   There is an alternative way to determine a walk. For this, label the edges alternately by 0 and 1 as follows. For $k\in \ZZ$, the edge $\{2k,2k+1\}$ is labeled by 0 and the edge $\{2k-1,2k\}$ is labeled by 1. Then, at every vertex there is one edge with label 0 and one edge with label 1. Hence, starting at the vertex $0\in \ZZ$, a binary sequence of labels fully determines a walk on the 2-regular tree $\ZZ$. A sequence of labels $l\in (\ZZ/2)^{2n}$ and a walk sequence $w\in \mathcal{W}_{2n}$ define the same walk of length $2n$ starting at $0\in \ZZ$ if and only if $l=w+s$ where $s$ is the alternating sequence. This provides a more direct interpretation of a boundary partition sequence $b\in \mathcal{B}_{2n}$ as a walk on $\ZZ$ and in light of the generalization in Remark~\ref{rmk:genRkN} below, this is the correct encoding of walks. See also Example \ref{exm:walkIsoClass} for an illustration of a walk in the case $n=6$.
\end{rmk}

\begin{exm}\label{exm:walkIsoLowN}
	For $n=1$, we have $\Sk_\alpha(\mathbb{H},(2,0))\cong A_\alpha$ by Example~\ref{exm:dTLEnd1}. There is a unique planar matching $M$ with $\arcs(M)=\{b\}$. The idempotent associated to the partition $B=\{b\}$, $C=\emptyset$ is the identity $\id\in \End(M)$ which becomes $\udcup\in \dTL_\alpha(0,2)$ under the isomorphism $\Sk_\alpha(\mathbb{H},(2,0))\cong \dTL_{\alpha}(0,2)$ from Proposition~\ref{prop:SolidtorusAnddTL}. Its dotted version is $\dcup\in \dTL_\alpha(0,2)$.
    
	For $n=2$ and a choice of boundary point $1\subset \partial B^2$, we have the following correspondence.
	\begin{center}
	\renewcommand{\arraystretch}{1.3}
		\begin{tabular}{c|c|c|c|c}\
			$w\in \mathcal{W}^+_{4}$&$b\in\mathcal{B}_{4}$&partition of $\arcs(M)$&$\dTL_\alpha$-diagram of $\mathbf{e}$& $\dTL_\alpha$-diagram of $\mathbf{\dot{e}}$\\
			\hline
			$(0,1,0,1)$&$(0,0,0,0)$& \begin{tabular}{@{}c@{}}
				$B=\{\{1,2\},\{3,4\}\}, C=\emptyset$ \\
				$B=\{\{1,4\},\{2,3\}\}, C=\emptyset$
			\end{tabular}  & \begin{tabular}{@{}c@{}}
			$\tfrac{1}{2}\left(\udcup\,\udcup + \tfrac{1}{\alpha}\dcup\,\dcup\right)$\\
			$=\tfrac{1}{2}\left(\nestedudcups + \tfrac{1}{\alpha}\nesteddcups \,\right)$ \end{tabular} & \begin{tabular}{@{}c@{}}$\tfrac{1}{2}(\dcup\, \udcup + \udcup\, \dcup)$ \\ $=\tfrac{1}{2}(\nesteddcupudcup+ \nestedudcupdcup$ \end{tabular} \\
			\hline
			$(0,0,1,1)$&$(0,1,1,0)$& $B=\{\{1,2\}\}, C=\{\{3,4\}\}$ &$\tfrac{1}{2}\left(\udcup\, \udcup - \tfrac{1}{\alpha}\,\dcup\, \dcup\right)$ & $\tfrac{1}{2}(\dcup\, \udcup - \udcup\, \dcup)$ \\
			\hline 
			$(0,1,1,0)$&$(0,0,1,1)$& $B=\{\{1,4\}\}, C=\{\{2,3\}\}$ & $\tfrac{1}{2}\left(\,\nestedudcups - \tfrac{1}{\alpha}\nesteddcups\, \right)$ & $\tfrac{1}{2}\left(\,\nesteddcupudcup - \nestedudcupdcup\, \right)$
		\end{tabular}
	\end{center}
	Note that the first boundary partition sequence can be represented using two different planar matchings. The idempotents are the respective joins in the endomorphism algebra and are isomorphic. On the level of $\dTL_\alpha$, this is witnessed by the saddle relations.
\end{exm}

\begin{exm}\label{exm:walkIsoClass}
	Let $n=6$ and consider the following example of a walk sequence $w\in \mathcal{W}^+_{12}$ and its associated isomorphism class of an idempotent with boundary partition sequence $b\in\mathcal{B}_{12}$ for $b=w+s$. We choose a planar matching on $12$ points with arc partition $\{B,C\}$ that realizes $b$. We color the arcs in $B$ in black, and the arcs in $C$ in red.
	\begin{align*}
		\longwalk && \idempForLongwalk\\[0.2cm]
		w=(0,1,1,0,0,0,1,0,1,1,0,1)  &&  
		b=(0,0,1,1,0,1,1,1,1,0,0,0)
	\end{align*}	
	To highlight the sequence of labels as described in Remark~\ref{rmk:altWalkSeq}, we also colored the steps of the walk in black for label 0 and red for label 1.
\end{exm}

\begin{proof}[Proof of Proposition \ref{prop:isoClassesOfIdem}]
	Lemma \ref{lem:isoseqofsaddles}.\textit{(ii)} provides the first and Lemma \ref{lem:walksOnZ} the second statement. Note that the explicit indexing of basis elements depends on the choice of $1\subset \partial B^2$.
\end{proof}
\begin{proof}[Proof of Corollary \ref{cor:isoClassesWalks}]
	Choose a boundary point $1\subset \partial B^2$. Using that $\{\mathbf{e},\mathbf{\dot{e}}\}$ forms a $\kk$-basis for $\End(\mathbf{e})$ and identifying $\mathbf{\dot{e}}$ with the reflection of the walk associated to $\mathbf{e}$, we obtain the statement from Corollary \ref{cor:skmodViaIsoClasses} and Lemma \ref{lem:walksOnZ}. We again remark that the explicit isomorphism depends on the choice $1\subset \partial B^2$.
\end{proof}

Before we consider the pairing on $\Sk_\alpha(\mathbb{H},(2n,0))$ to compute the Kirby color for $A_\alpha$, we comment below on a generalization of $A_\alpha$ to rank $N$ and a root projector version of the above results.
\begin{rmk}[Generalization to rank $N$]\label{rmk:genRkN}
	There are similar statements for $\beta$-modified Bar-Natan theory of rank $N$ based on the Frobenius algebra $A_\beta$ from Example \ref{exm:FAExmLinkHom}. Let $N\geq 2$, $\kk=K[\beta^{\pm1}]$ for a field $K$ with $N\in K^\times$ and a primitive $N$-th root of unity $\zeta_N\in K$. Consider $A_\beta=\kk[x]/(x^N-\beta)$ with $\varepsilon\colon x^{k}\mapsto \delta_{k,N-1}$ and $\Delta\colon 1\mapsto \sum_{i=0}^{N-1} x^{N-1-i}\otimes x^{i}$. The handle element $H=Nx^{N-1}$ is invertible with $H^{-1}=x/(N\beta)$.
	
	Consider the category $\BN_\beta(B^2,2n)$ with a distinguished boundary point $1\subset \partial B^2$. The idempotents in $\BN_\beta(B^2,2n)$ have a similar structure. Proposition \ref{prop:structureOfIdem} generalizes as follows. Let $M$ be a planar matching on $2n$ points. We sketch how partitions of $\arcs(M)$ into $N$ blocks $B_1,\dots, B_N$ with the arc containing the boundary point 1 mapped to $B_1$, and other blocks allowed to be empty, define a set of complete, $A_\beta$-primitive, orthogonal idempotents in $A_\beta^{\otimes n}\cong\End_{\BN_\beta}(M)$. The separability idempotent 
	\begin{align*}
		\Delta(H^{-1})= \frac{1}{N\beta} \sum_{i=0}^{N-1} x^{N-i} \otimes x^{i} \in A^{\otimes 2} 
	\end{align*}
	on two arcs $b,c$ and the identity on the other arcs defines the \emph{join} idempotent $j(b,c)\in \End(M)\cong A_\beta^{\otimes n}$ of $b$ and $c$. Consider the algebra endomorphism $\Phi\colon A_\beta\to A_\beta$ determined by $\Phi(x)=\zeta_Nx$. The idempotent $1\otimes 1-\Delta(H^{-1})\in A_\beta^{\otimes 2}$ is orthogonal to the join and splits into $N-1$ idempotents in $A_\beta^{\otimes n}$ constructed from $\Phi^k$ for $1\leq k\leq N-1$.
	\begin{align*}
		\distk\; = \mathrm{dist}_k=  (\id\otimes \Phi^k)\circ \Delta(H^{-1}) \in A_\beta^{\otimes 2} 
	\end{align*}
	Here we depicted the algebra morphism $\Phi^k$ in colored dashed lines. Using graphical calculus, we verify the idempotency by computing
	\begin{align*}
		\left(\;\distk\; \right)^2 \;=\; \distksq\; =\; \distksqmod\; =\; \distk
	\end{align*}
	where we applied in the second step that $\Phi^k$ is an algebra morphism. The last equality follows from the $H^{-1}$ canceling the handle. For $1\leq i\leq N-1$ fixed, we have
	\begin{align*}
		\sum_{k=0}^{N-1}({\zeta_N}^{i})^k = 0.
	\end{align*}
	As a consequence
	\begin{align*}
		 \Delta(H^{-1})+\sum_{k=1}^{N-1} \mathrm{dist}_k &= \sum_{k=0}^{N-1}  (\id\otimes \Phi^k)\circ \Delta(H^{-1}) = \frac{1}{N\beta} \sum_{k=0}^{N-1} \sum_{i=0}^{N-1} x^{N-i} \otimes \Phi^k(x^{i})\\
		 &= \frac{1}{N\beta} \sum_{k=0}^{N-1} \sum_{i=0}^{N-1} x^{N-i} \otimes {\zeta_N}^{k\cdot i} x^i  = \frac{1}{N\beta} \sum_{k=0}^{N-1} x^N\otimes 1  = 1\otimes 1 \in A_\beta^{\otimes 2}.
	\end{align*}
	For $b,c\in \arcs(M)$ and choose $\mathrm{dist}_k\in A_\beta^{\otimes 2}$ on $b$ and $c$, and the identity on the other arcs. This defines the \emph{$k$-distance} idempotent $d_k(b,c)\in \End(M)\cong A_\beta^{\otimes n}$. For a partition $P=\{B_1,\dots, B_N\}$ as above, we let $e_P\in \End(M)$ be the idempotent constructed from products of join and distance idempotents as follows: If two arcs $b,b'$ are in the same block $B_i$, we have a factor $j(b,b')$. If $b\in B_i$ and $c\in B_{i+k}$ with indices considered mod $N$, we have a factor $d_k(b,c)$ recording the distance $k$ between the indices of the blocks in which the two arcs are partitioned. Note that for $N=2$, this recovers the disjoin for the algebra morphism with $\Phi(x)=-x$. One can show that this construction indeed yields a set of complete, $A_\beta$-primitive, orthogonal idempotents.
    
    We have the following statements analogous to Lemma~\ref{lem:isoseqofsaddles}. A partition of $\arcs(M)$ induces a partition on the set of boundary points. For partitions $P$ and $Q$ of the arc sets of planar matchings $M_1$ and $M_2$ into $N$ blocks, the corresponding idempotents $e_P$ and $e_Q$ are isomorphic if and only if their boundary partitions are the same. Otherwise $\Hom_{\overline{\BN}_{\beta}}(e_P,e_Q)=0$ in the Karoubi completion. As a result, $\Sk_\beta(\mathbb{H},(2n,0))\cong \bigoplus_\textbf{e} \End(\textbf{e})$, generalizing Corollary~\ref{cor:skmodViaIsoClasses}. Isomorphism classes of these idempotents in $\BN_\beta(B^2,2n)$ correspond to returning walks of length $2n$ on an $N$-regular tree starting and ending at a fixed vertex and with a fixed first step as follows. Consider the edges of the $N$-regular tree to be labeled by $\{1,\dots, N\}$ such that every vertex has exactly one edge for every label. At each step, the boundary partition sequence determines a walk by taking the corresponding edge. Since the boundary partitions can be realized by a planar matching, similar to Lemma~\ref{lem:walksOnZ}, the walk is returning. The fixed first step is a result of the choice that the boundary point 1 is in block $B_1$. Since $\End(e_P)\cong A_\beta$, the rank of $\Sk_\beta(\mathbb{H},(2n,0))$ over $\kk$ is given by the number of \emph{all} returning walks of length $2n$ on an $N$-regular tree starting and ending at a fixed vertex. This generalizes Proposition~\ref{prop:isoClassesOfIdem} and Corollary~\ref{cor:isoClassesWalks}. 
	
\end{rmk}
\begin{rmk}[Root projector version]\label{rmk:genSemisimpleFA}
	Another related theory is the following. Let $N\geq 1$ and let $K$ be a field with $N\in K^\times$. Recall from Example~\ref{exm:FAExmLinkHom} the Frobenius algebra
	\begin{align*}
		A_\mu=K[x,\mu_1,\dots, \mu_N, \mathrm{disc}(\mu_1,\dots,\mu_N)^{-1}]/\Pi_i(x-\mu_i)
	\end{align*}
	with counit $\varepsilon\colon x^{k}\mapsto \delta_{k,N-1}$. The comultiplication is determined by $\varepsilon$, and the handle element given by the formal derivative $H=\mathrm{d}p/\mathrm{d}x$ of the polynomial $p=\prod_i(x-\mu_i)$. Here, $\mathrm{disc}(\mu_1,\dots \mu_N)$ is the discriminant as a function of the roots. Its invertibility guarantees that the $\mu_i$ are pairwise distinct and thereby the strong separability of $A_\mu$. See \cite{BoernerDrube} for a version of this skein theory without the separability assumption.
	
	The Lagrange interpolators
	\begin{align*}
		\pi_i = \prod_{\substack{1\leq j \leq N\\ j\neq i}} \frac{x-\mu_j}{\mu_i-\mu_j}\in A_\mu, && \pi_i^2 = \pi_i + \frac{\prod_{j\neq i}(x-\mu_j)-\prod_{j\neq i}(\mu_i-\mu_j)}{(\prod_{j\neq i} (\mu_j-\mu_i)^2)(x-\mu_i)}p = \pi_i\in A_\mu.
	\end{align*}
	are idempotents for every $i$ called \emph{root projectors}. Note that the fraction on the right-hand side is a polynomial since the numerator has a root at $x=\mu_i$. For $i\neq j$, we clearly have 
	\begin{align*}
		\pi_i\pi_j=0 \in A_\mu.
	\end{align*}
	By evaluating the polynomial $\sum_i \pi_i$ and the constant polynomial $1$ at $x=\mu_k$ for all $1\leq k\leq N$, we obtain that the Lagrange interpolators satisfy
	\begin{align*}
		\sum_{i=1}^N \pi_i = 1 \in A_\mu.
	\end{align*}
	The idempotents $\pi_i\in A_\mu$ form a $\kk$-basis can be used to construct idempotents in $A_\mu^{\otimes n}\cong \End(M)$. For each arc $b\in \arcs(M)$, pick a root $\mu_i$ for $1\leq i \leq N$ and decorate the sheet on $b$ with $\pi_i$. This defines an idempotent $e\in \End(M)$. The set of such idempotents forms a $\kk$-basis of $\End(M)$. The map $\arcs(M)\to \{\mu_1,\dots ,\mu_N \}$ can be thought of as partitioning $\arcs(M)$ into blocks corresponding to the roots $\mu_i$ (with blocks allowed to be empty). Summing over idempotents associated to partitions into root blocks, we obtain
	\begin{align*}
		\sum_P e_P = \id \in \End(M).
	\end{align*}
	Another consequence of the above is $\End_{\overline{\BN}_{\mu}}(e_P)=\kk$. The separability idempotent $\Delta(H^{-1})\in A_\mu^{\otimes 2}$ acts as a join. This follows from the bimodule property of $\Delta$ or, more topologically, by decorations moving along the neck. In fact, this shows that the separability idempotent of a general strongly separable commutative Frobenius algebra always has a join property. Expanding the separability idempotent in $A_\mu^{\otimes 2}$ in terms of the basis $\pi_i\otimes \pi_j$ and using orthogonality of $\pi_i$ and idempotency, we find that it decomposes as
	\begin{align*}
		\Delta(H^{-1})=\sum_{i=1}^N \pi_i\otimes \pi_i \in A_\mu^{\otimes 2}.
	\end{align*}
	This means that saddles have an inverse given by the reversed saddle decorated with $H^{-1}$. We obtain statements similar to Lemma \ref{lem:saddlesjoindisjoin}, but with more general idempotents orthogonal to $\Delta(H^{-1})$ in place of the disjoin. As a consequence, we also obtain the analogous statement to Lemma \ref{lem:isoseqofsaddles}. Two idempotents associated to partitions into root blocks as above are isomorphic in the Karoubi completion $\overline{\BN}_\mu$ if and only if the induced maps $\{1,\dots, 2n\}\to \{\mu_1,\dots, \mu_N\}$, partitioning the boundary points into root blocks, agree. A similar argument to Lemma \ref{lem:walksOnZ} shows that isomorphism classes of these idempotents are classified by walks of length $2n$ on $N$-regular trees, starting and ending at a fixed vertex, but without a fixed choice of first step. This is a consequence of the fact that here we do not restrict to those partitions mapping the arc on the boundary point $1\subset \partial B^2$ into a fixed block. The rank of $\Sk_\mu(\mathbb{H},(2n,0))$ is the number of all such walks on $N$-regular trees.
    
	Recall that for $A_\beta$, considered in Remark \ref{rmk:genRkN}, the $k$-distance idempotent recording the distance between the blocks of two arcs. The analogous idempotent in $A_\mu$ is 
    \begin{align*}
		\mathrm{dist}_k = \sum_{i=1}^N \pi_i\otimes \pi_{i+k} \in A_\mu^{\otimes 2}
	\end{align*}
    where the index $i+k$ is to be treated mod $N$. Note however, that there are many more idempotents in $A_\mu^{\otimes 2}$. Indeed, one could write down a similar expression for any permutation of roots $\mu_i$.
\end{rmk}

\subsection{The Kirby color for $A_\alpha$}
In this section, we compute the Kirby color for $A_\alpha$ as defined in Definition~\ref{def:capAndKirby}. For this, we first compute the pairing $p_2$ and show that it is perfect. Using the copairing and cap values, we then derive two expressions for the Kirby color.

In the following, we will assume that the evaluation of the empty skein in $S^3$ is $\ev(\emptyset)=1\in \kk^\times$. By Proposition \ref{prop:evEulerChar}, a more general choice of $\ev(\emptyset)\in \kk^\times$ can be recovered in the computation of the invariant using the Euler characteristic. We will speak of the cap value and Kirby color always for $\ev(\emptyset)=1$. Note that the pairing $p_2$ and copairing $c_2$ do not carry a factor of $\ev(\emptyset)$.

\begin{prop}\label{prop:pairing}
	The pairing $p_2$ on the skein module of the solid torus is given on  isomorphism classes of primitive idempotents as	
	\begin{align*}
		\langle -,-\rangle \colon \bigoplus_{\mathbf{e}} \End(\mathbf{e})& \otimes \bigoplus_{\mathbf{e}} \End(\mathbf{e}) \to \kk\\
		\langle \mathbf{e}, \mathbf{f}\rangle = 2\cdot\delta_{\mathbf{e},\mathbf{f}}, \quad \quad
		\langle \mathbf{\dot{e}}, \mathbf{\dot{f}} \rangle &= 2\alpha \cdot\delta_{\mathbf{e},\mathbf{f}}, \quad \quad
		\langle \mathbf{\dot{e}}, \mathbf{f} \rangle =0= \langle \mathbf{e}, \mathbf{\dot{f}} \rangle.
	\end{align*}
	By our assumptions $2,\alpha\in \kk^\times$ it is, thus, perfect.
\end{prop}
\begin{proof}
	In the following, choosing a representative idempotent $e\in \End(M)$ for $\mathbf{e}$ amounts to picking a planar matching $M$ realizing the boundary partition. By Proposition \ref{prop:SolidtorusAnddTL}, we obtain an element in $\dTL_A(0,2n)$. Two different choices agree by the relations in $\dTL_{\alpha}$.
    
	We first show $\langle \mathbf{e}, \mathbf{e}\rangle = 2$. Choose the same representing idempotent $e\in \End_{\BN_{\alpha}}(M)$ for a planar matching $M$ with arc partition $\{B,C\}$ of $\arcs(M)$. By Proposition \ref{prop:structureOfIdem}, the idempotent $e$ is a sum over $\dTL_\alpha$ diagrams with an even number of dots per diagram. The pairing creates a closed circle component out of each arc and evaluates them. By the relations in $\dTL_\alpha$, a circle is zero if and only if it is dotted. Hence, the only non-zero terms in $\langle e,e\rangle$ are the diagonal terms. For each non-zero circle component, we get a factor of 2 since two dots are precisely canceled by a factor of alpha. There are $2^{n-1}$ terms, one for each even subset $D\subseteq \arcs(M)$.  The signs cancel out in the diagonal terms. Hence, we obtain
	\begin{align*}
		\langle \mathbf{e},\mathbf{e}\rangle = \frac{1}{2^{n-1}2^{n-1}} 2^{|\arcs(M)|} 2^{n-1}=2.
	\end{align*}
	Now, choose representatives $e\in\End_{\BN_{\alpha}}(M)$ and $f\in\End_{\BN_{\alpha}}(N)$ for planar matchings $M$ and $N$. The $\dTL_\alpha$-diagram for $\dot{e}$ consists of terms each of which has an odd number of dots. This means any such term paired with any term of $f$ yields $n$ circle components where at least one circle has a single dot. Hence, 
	\begin{align*}
		\langle \mathbf{\dot{e}}, \mathbf{f}\rangle= 0 =\langle  \mathbf{e},\mathbf{\dot{f}}\rangle.
	\end{align*} 
	We show that $\langle\mathbf{e}, \mathbf{f}\rangle=0$ for $\mathbf{e}\neq \mathbf{f}$. Let $b,b'\in \mathcal{B}_{2n}$ the boundary partition sequences for $\mathbf{e}$ and $\mathbf{f}$ respectively. Since $b\neq b'$, there exists a pair of indices $(k,l)$ with $1\leq k,l\leq 2n$ such that
	\begin{align*}
		b(k)=b(l)=0\quad\quad \text{and}\quad\quad 0=b'(k)\neq b'(l)=1.
	\end{align*}
 	We can assume that the arcs containing the boundary points $k$ and $l$ are distinct for both $e$ and $f$. Denote by $j(k,l)$ and $d(k,l)$ the join and disjoin respectively of the two arcs containing the boundary points $k$ and $l$. We have $e\cdot j(k,l)=e$ and $f\cdot d(k,l)=f$. We compute the pairing of $e\cdot j(k,l)$ and $f\cdot d(k,l)$. After gluing the $\dTL_\alpha$-diagrams for $e$ and $f$, for the closed circle components corresponding to $k$ and $l$, we have locally 
	\begin{align*}
		\frac{1}{4}\left[\; \LocaldTLududKL\, + \frac{1}{\alpha}\,\LocaldTLddKL\,- \frac{1}{\alpha}\,\LocaldTLddKL\,-\frac{1}{\alpha^2}\, \LocaldTLDoubleddKL \;\right] = 0.
	\end{align*}
    Note that this essentially follows from the orthogonality $d(k,l)j(k,l)=0$. Lastly, for $\langle \mathbf{\dot{e}}, \mathbf{\dot{f}}\rangle$, we have
	\begin{align*}
		\langle \mathbf{\dot{e}}, \mathbf{\dot{f}}\rangle = \alpha \langle \mathbf{e}, \mathbf{f}\rangle = 2\alpha\cdot \delta_{\mathbf{e},\mathbf{f}}
	\end{align*}
	since we assume that each of the dot is on the arc with boundary point 1 and cancels out for an $\alpha$ upon pairing.
\end{proof}
As a consequence, we obtain the copairing.
\begin{cor}\label{cor:copairing}
	The copairing bimodule map
	\begin{align*}
		c_2\colon {}_{\skcat_\alpha(\mathbb{T})}\Sk_\alpha(I\times \mathbb{T} )_{\skcat_\alpha(\mathbb{T})} \to {}_{\skcat_\alpha(\mathbb{T})}\Sk_\alpha(\mathbb{H}) \otimes_\kk \Sk_\alpha(\mathbb{H})_{\skcat_\alpha(\mathbb{T})}
	\end{align*}
	for the boundary condition $c=(2n,0)$ is given by
	\begin{align*}
		\id_c &\mapsto \sum_{\mathbf{e}} \tfrac{1}{2}\mathbf{e}\otimes \mathbf{e} + \tfrac{1}{2\alpha} \mathbf{\dot{e}}\otimes\mathbf{\dot{e}}.
	\end{align*}
\end{cor}
\begin{exm}\label{exm:copairingLowN}
	For $n=1$, we have $\id_{(2,0)}\mapsto \tfrac{1}{2}\udcup\otimes \udcup + \tfrac{1}{2\alpha}\dcup\otimes \dcup$. For $n=2$, we have three walk sequences and correspondingly three isomorphism classes of idempotents $\mathbf{e}$ and their dotted versions $\mathbf{\dot{e}}$ from Example~\ref{exm:walkIsoLowN}. The copairing sends 
	\begin{align*}
		\id_{(4,0)}\mapsto &\ \tfrac{1}{8}(\udcup\, \udcup + \tfrac{1}{\alpha}\,\dcup\, \dcup) \otimes (\udcup\, \udcup + \tfrac{1}{\alpha}\,\dcup\, \dcup) + \tfrac{1}{8\alpha}(\dcup\, \udcup + \udcup\, \dcup)\otimes (\dcup\, \udcup + \udcup\, \dcup)\\
		 +&\ \tfrac{1}{8}(\udcup\, \udcup - \tfrac{1}{\alpha}\,\dcup\, \dcup)\otimes (\udcup\, \udcup - \tfrac{1}{\alpha}\,\dcup\, \dcup) + \tfrac{1}{8\alpha}(\dcup\, \udcup - \udcup\, \dcup)\otimes (\dcup\, \udcup - \udcup\, \dcup)\\
		+&\ \tfrac{1}{8}\left(\,\nestedudcups - \tfrac{1}{\alpha}\nesteddcups\, \right)\otimes \left(\,\nestedudcups - \tfrac{1}{\alpha}\nesteddcups\, \right) + \tfrac{1}{8\alpha}\left(\,\nesteddcupudcup - \nestedudcupdcup\, \right)\otimes \left(\,\nesteddcupudcup - \nestedudcupdcup\, \right).
	\end{align*}
\end{exm}

\begin{rmk}[The pairing and copairing for $A_\beta$]\label{rmk:copRkN}
Consider surface skein theory for the Frobenius algebra $A_\beta$ as in Remark \ref{rmk:genRkN}. Consider the category $\BN_{\beta}(B^2,2n)$ for $A_\beta$, and choose a distinguished boundary point $1\subset \partial B^2$ inducing a total order on the boundary points in $\partial B^2$. Let $M$ be a planar matching on $2n$ points and $e\in \End_{\BN_\beta}(M)$ be an idempotent associated to a partition of $\arcs(M)$ into $N$ blocks with all blocks but the first block allowed to be empty. Denote by $e^{\bullet (k)}\in \End_{\overline{\BN}_\beta}(e)$ the endomorphism that has $k$ dots on the arc containing the boundary point 1. Consider a skeletonization of $\overline{\BN}_\beta$. A basis for $\End_{\overline{\BN}_\beta}(\mathbf{e})$ is given by $\{\mathbf{e}, \mathbf{e}^{\bullet(1)}, \dots, \mathbf{e}^{\bullet(N-1)}\}$. Here $\mathbf{e}^{\bullet (k)}$ is the endomorphism of $\mathbf{e}$ induced by $e^{\bullet (k)}$. The pairing $p_2$ on the $A_\beta$-skein module of the solid torus with $2n$ longitudes becomes
\begin{align*}
	\langle\mathbf{e},\mathbf{f}\rangle=N\cdot \delta_{\mathbf{e},\mathbf{f}}, \text{ for } 1\leq k,l,\leq N-1, \quad \langle \mathbf{e}^{\bullet (k)},\mathbf{f}^{\bullet (l)}\rangle = N \beta \cdot \delta_{N,k+l}\delta_{\mathbf{e},\mathbf{f}}
\end{align*}
and is perfect as by assumption, $\beta,N\in \kk^\times$. As a result, for $c=(2n,0)$, the copairing sends 
 \begin{align*}
    \id_c \mapsto \sum_{\mathbf{e}} \sum_{i=1}^N  \frac{1}{N\beta}\mathbf{e}^{\bullet(i)}\otimes \mathbf{e}^{\bullet(N-i)}.  
 \end{align*}
Note that for, $H^{-1}= \tfrac{1}{N\beta} x$, we have
\begin{align*}
    \Delta(H^{-1}) = \sum_{i=1}^{N} \frac{1}{N\beta}x^i\otimes x^{N-i} \in A_\beta^{\otimes 2}.
\end{align*}
This means, we can think of the copairing as $\Delta(H^{-1})$ acting on each summand $\mathbf{e}\otimes \mathbf{e}$.
\end{rmk}
Next, we compute the cap value of an element in the skein module of the solid torus. By Corollary \ref{cor:skmodViaIsoClasses}, this amounts to determining the cap value of isomorphism classes $\mathbf{e}$ of primitive idempotents $e$. Considering idempotents $e$ as elements in $\dTL_\alpha(0,2n)$, by the isomorphisms from Proposition \ref{prop:SolidtorusAnddTL} and Corollary \ref{cor:skmodViaIsoClasses}, we have that two isomorphic idempotents $e$ and $e'$ are related by $\dTL_\alpha$ relations and therefore define the same element and $\mathrm{cap}(\textbf{e})=\mathrm{cap}(e)$. The meridional disks cap off the annuli obtained from rotating the $\dTL_\alpha$-diagrams. Hence, by the sphere and dotted sphere relation for $A_\alpha$, the cap value of $e$ is equal to the coefficient of the fully dotted term in $e$.

\begin{prop}[Cap value]\label{prop:capval}
	Let $M$ be a planar matching on $2n$ points, and $e\in \End_{\BN_{\alpha}}(M)$ an idempotent associated to the partition $\{B, C\}$ of $\arcs(M)$ with $C=\emptyset$ allowed. Let $b\in B$ be an arc with $\dot{e}=\mathrm{dot}_b(e)$. Then, 
	\begin{align*}
		\mathrm{cap}(e) = 
		\begin{cases}
			\frac{(-1)^{|B|}}{2^{(n-1)}\alpha^{n/2}}&\quad \text{if $n$ even,}\\
			0&\quad \text{if $n$ odd}
		\end{cases}
	&&\text{ and } &&
		\mathrm{cap}(\dot{e}) =
		\begin{cases}
			0&\quad \text{if $n$ even},\\
			\frac{(-1)^{|C|}}{2^{(n-1)}\alpha^{(n-1)/2}} &\quad \text{if $n$ odd.}
		\end{cases}
	\end{align*}
\end{prop}
\begin{proof}
	By Proposition \ref{prop:structureOfIdem}, the idempotent $e$ is a linear combination of terms with an even number of dotted sheets with all coefficients non-zero. If $n=|\arcs(M)|$ is even, there is a term with dots on all sheets $D=\arcs(M)$, the fully dotted term. This term has coefficient $(-1)^{|B|}2^{-(n-1)}\alpha^{-n/2}$. Note that since $n$ even, we have $(-1)^{|B|}=(-1)^{|C|}$. If $n$ is odd, there is no such term and $\mathrm{cap}(e)=0$. To determine $\mathrm{cap}(\dot{e})$, note that $\dot{e}$ consists of terms with an odd number of dotted sheets since either a dot on $b$ is added or canceled for an $\alpha$. As a result, $\mathrm{cap}(\dot{e})=0$ if $n$ is even. If $n$ is odd, then the fully dotted term of $\dot{e}$ has the same coefficient as the term of $e$ with all dotted sheets except the sheet on $b$. The coefficient is $(-1)^{|D\cap B|}2^{n-1}\alpha^{-(n-1)/2}$ for $D=\arcs(M)\setminus\{b\}$. But $|C| = |D\cap C| \equiv |D\cap B| \mod 2$. 
\end{proof}

A vector $z=(z_1, \dots, z_{2n})\in (\ZZ/2)^{2n}$ defines a pure tensor $T_z = x^{z_1}\otimes \dots \otimes x^{z_{2n}} \in A_\alpha^{\otimes 2n}$. The collection of all $T_z$ forms a basis of $A_\alpha^{\otimes2n}$. Write $|z|=|\{i\mid z_i=1\}|$ for the number of factors of $x$ in $T_z$. 
\begin{thm}\label{thm:KirbyAlpha}
	The Kirby color for $A_\alpha$ is given by
	\begin{align*}
		\omega_{2n} = \frac{1}{2^{2n}}\sum_{k=0}^{n}\frac{1}{\alpha^k}S(n-k,k)\sum_{\substack{z\in (\ZZ/2)^{2n} \\|z|=2k}} \mathrm{sign}(z,k)T_z \;\in A_\alpha^{\otimes 2n}
	\end{align*}
	with \begin{align*}
		\mathrm{sign}(z,k) = \begin{cases}
			+1 &\text{if } \sum\limits_{i\text{ even}} z_i\equiv k \mod 2,\\
			-1 &\text{else.}			
		\end{cases}
	\end{align*} and the super Catalan numbers\footnote{Setting $b=1$,  $S(a,1)/2$ recovers the Catalan number $C_a$. Setting $b=0$, one obtains the middle binomial coefficient $\tbinom{2a}{a}$. Moreover, the expression is symmetric in $a$ and $b$. For more details on the super Catalan numbers, see \cite{GesselSuperBallot}.}
	\begin{align*}
		S(a,b) = \frac{(2a)!(2b)!}{a!b!(a+b)!} \quad \text{ for } a,b\in \NN_0.
	\end{align*}
\end{thm}
\begin{exm}\label{exm:kirbySmalln} 
	For $n=1$, we have $S(1,0)=S(0,1)=2$ and therefore $\omega_{2}=\tfrac{1}{4}(2\cdot 1\otimes 1 +  \tfrac{2}{\alpha} x\otimes x)\in A_\alpha^{\otimes 2}$. Indeed, applying $\mathrm{cap}$ to the image of $\id_{(2,0)}$ in Example \ref{exm:copairingLowN} under the copairing $c_2$ yields 
	\begin{align*}
		\omega_2 = \tfrac{1}{2}\mathrm{cap}(\udcup)\, \udcup + \tfrac{1}{2\alpha} \mathrm{cap}(\dcup)\, \dcup = \tfrac{1}{2\alpha}\dcup = \tfrac{1}{2}(1\otimes 1 + \tfrac{1}{\alpha} x\otimes x).
	\end{align*}
	Let $n=2$. Applying $\mathrm{cap}$ to the expression in Example \ref{exm:copairingLowN}, we obtain the Kirby color 
	\begin{align*}
		\omega_4 &= \sum_{\mathbf{e}}\tfrac{1}{2}\mathrm{cap}(\mathbf{e}) \mathbf{e} + \tfrac{1}{2\alpha} \mathrm{cap}(\mathbf{\dot{e}})\mathbf{\dot{e}} = \sum_\mathbf{e}\tfrac{1}{2}\mathrm{cap}(\mathbf{e}) \mathbf{e}
		 = \tfrac{1}{8\alpha}\left( \tfrac{1}{\alpha}\dcup\,\dcup + \udcup\,\udcup + \tfrac{1}{\alpha}\dcup\,\dcup - \udcup\,\udcup + \tfrac{1}{\alpha} \nesteddcups - \nestedudcups\,\right)\\
		 &= \tfrac{1}{16} \left(\tfrac{4}{\alpha^2}\dcup\,\dcup + \tfrac{2}{\alpha^2}\nesteddcups- \tfrac{2}{\alpha}\nestedudcups \,\right)\\
		 &= \tfrac{1}{16}\left(6\cdot 1\otimes 1\otimes 1\otimes 1
		 + \tfrac{2}{\alpha} \cdot 1\otimes 1\otimes x\otimes x
		 -\tfrac{2}{\alpha}\cdot 1\otimes x\otimes 1 \otimes x
		 + \tfrac{2}{\alpha} \cdot 1 \otimes x \otimes x\otimes 1\right.\\
		 &\left.+ \tfrac{2}{\alpha} \cdot x\otimes 1 \otimes 1\otimes x
		 -\tfrac{2}{\alpha}\cdot x\otimes 1\otimes x \otimes 1
		 + \tfrac{2}{\alpha} \cdot x\otimes x\otimes 1\otimes 1
		 + \tfrac{6}{\alpha^2} \cdot x\otimes x\otimes x\otimes x\right).
	\end{align*}
	 Indeed, for terms with $|z|=0$ or $|z|=4$ we have $S(0,2)=S(2,0)=6$, and for terms with $|z|=2$ we have a coefficient $S(1,1)=2$.
\end{exm}
For the proof of Theorem \ref{thm:KirbyAlpha}, we need the following Lemma.
\begin{lem}[von Szily's identity]
	Let $a,b$ be non-negative integers. The super Catalan numbers satisfy
	\begin{align*}
		S(a,b) = \sum_{k\in \ZZ} (-1)^k \binom{2a}{a+k}\binom{2b}{b-k}
	\end{align*}
\end{lem}
\begin{proof}
	See \cite{LarcombeFrenchGesselvonSzilyIdentity} and references therein.
\end{proof}
\begin{proof}[Proof of Theorem \ref{thm:KirbyAlpha}]
	First, we prove that the only terms $T_z$ occurring in the expression have $|z|=2k$ for some $k=0,\dots,n$. Using the embedding $\dTL_\alpha(0,2n)\subset A_\alpha^{\otimes 2n}$, an undotted cup $\udcup$ is evaluated to $1\otimes x+ x\otimes 1$, an odd number of $x$'s; a dotted cup $\dcup$ to $\alpha 1\otimes 1 + x\otimes x$, an even number of $x$'s. Let $e$ be a primitive idempotent representing the isomorphism class $\mathbf{e}$ and its dotted version $\mathbf{\dot{e}}$. If $n$ is even, only the undotted isomorphism class $\mathbf{e}$ yields a non-zero cap value. Every term in $e$ has an even number of dotted arcs and thereby also an even number of undotted arcs. Hence, considered as elements in $A^{\otimes 2n}$, every term has an even number of factors of $x$. If $n$ is odd, only dotted isomorphism classes have non-zero cap value. In this case there is an odd number of dotted arcs in each term and there is again an even number of undotted arcs. Hence, only $T_z$ with $|z|=2k$ occur. Next, we determine the coefficients. The $2^{-2n}$ and $\alpha^{-k}$ follow from cap values and coefficients of the idempotents: Every (dotted) idempotent has an overall factor of $2^{-(n-1)}$. The cap values provide another factor of $2^{-(n-1)}$. The copairing has a factor of $\frac12$ and the remaining factor of $\frac12$ comes from the combinatorial consideration below. For the factors of $\alpha$, observe that the expression for the idempotents, copairing, and cap values are homogeneous when endowing $A_\alpha$ with a grading as in Construction \ref{con:joindisjoin}. Hence, we can restore the factor of alpha by counting the number of $x$'s in the term $T_z$ which yields $\alpha^{-k}$.
	
	We are now left with sums over (pure) $\dTL_\alpha$-diagrams with coefficients $\pm1$ (and can subsequently ignore the factors of $\alpha$). Fix an idempotent $e$ representing $\mathbf{e}$, that is, fix a planar matching $M$ and partition $\{B,C\}$ of $\arcs(M)$. Note that if $n$ is even, $|B|\equiv |C| \mod 2$. Hence, the cap value contributes a global sign $(-1)^{|C|}$ and we do not need to distinguish $n$ even and odd. By Proposition~\ref{prop:structureOfIdem}, the $\dTL_\alpha$-diagram with dotting $D$ contributes the sign $(-1)^{|D\cap C|}$ to the expression of $e$. Every isomorphism class of these idempotents yields every term $T_z$ exactly once. For fixed $M$, we choose an arc $\{i,j\}\in \arcs(M)$ to be dotted if $z_i = z_j$, and undotted if $z_i\neq z_j$. This determines $D$ uniquely. The overall sign of $x_D$ is
	\begin{align*}
		(-1)^{|C|}(-1)^{|D\cap C|}=(-1)^{|C\setminus (D\cap C)|}.
	\end{align*}
	Let $b$ be the boundary partition sequence of $\mathbf{e}$, and fix the representing planar matching $M$ and the dotting $D$ in $e$ producing the term $T_z$. Recall that $b_i=1$ if and only if the arc at $i$ is in $C$ for the partition $\{B,C\}$ of $\arcs(M)$. We make the following computation. 
	\begin{align*}
		z\cdot b &= \sum_i z_ib_i = \sum_{(i,j)\in \arcs(M)} z_ib_i+z_jb_j = \sum_{(i,j)\in C} z_i+z_j = \sum_{(i,j)\in C\cap D} z_i+z_j + \sum_{(i,j)\in C\setminus(D\cap C)}  z_i+z_j \\&= |C\setminus (D\cap C)|  + \sum_{(i,j)\in C\setminus(D\cap C)}  z_i+z_j \equiv |C\setminus (D\cap C)| \mod 2.
	\end{align*}
	In the last two steps we used that for $(i,j)\in C\setminus(D\cap C)$, $z_i\neq z_j$, we have $z_i+z_j=1$, and for $(i,j)\in C\cap D$, we have $z_i=z_j$. Hence, by enumerating the isomorphism classes by their boundary partition sequences, the coefficient of $T_z$ with $|z|=2k$ that is still unaccounted for is equal to 
	\begin{align*}
		\sum_{b\in \mathcal{B}_{2n}} (-1)^{z\cdot b}.
	\end{align*}
	In the following we show that this is precisely $\tfrac{1}{2}S(n-k,k)\mathrm{sign}(z,k)$. First, let $z=(0,\dots,0, 1,\dots, 1)\in (\ZZ/2)^{2n}$ with $|z|=2k$ and consider $T_z$. Let $\mathcal{W}^{\text{even}}_{2n,2k}$ and $\mathcal{W}^{\text{odd}}_{2n,2k}$ be the subsets of $\mathcal{W}^{+}_{2n}$ with an even resp.\ odd number of 1's in the last $2k$ digits. We have		
		\begin{align*}
			|\mathcal{W}^{\text{even}}_{2n,2k}| = \tfrac{1}{2} \sum_{\substack{m\in \ZZ\\m \text{ even}}} \tbinom{2(n-k)}{n-m}\tbinom{2k}{m} \quad \quad \text{and}\quad \quad
			|\mathcal{W}^{\text{odd}}_{2n,2k}| =\tfrac{1}{2} \sum_{\substack{m\in \ZZ\\m \text{ odd}}} \tbinom{2(n-k)}{n-m}\tbinom{2k}{m}
		\end{align*}
		since we can divide the sequences into subsequences of lengths $2(n-k)$ and $2k$ and sum over $m$, the number of ones in the subsequence of length $2k$. The $\frac12$ comes from considering walk sequences $\mathcal{W}^+_{2n}$ instead of  $\mathcal{W}_{2n}$. Now, for a boundary partition sequence $b$ and its associated walk sequence $w$, we have $b=w+(0,1,0,1,\dots)$ and 
		\begin{align*}
			\sum_{b\in \mathcal{B}_{2n}} (-1)^{z\cdot b} &= (-1)^k \sum_{w\in \mathcal{W}_{2n}^+}(-1)^{z\cdot w} = (-1)^k \left( |\mathcal{W}^{\text{even}}_{2n,2k}| - |\mathcal{W}^{\text{odd}}_{2n,2k}|\right) = \tfrac{(-1)^k}{2}  \sum_{m\in \ZZ} (-1)^m\tbinom{2(n-k)}{n-m}\tbinom{2k}{m}\\
			&= \tfrac{1}{2}  \sum_{m\in \ZZ} (-1)^{m+k}\tbinom{2(n-k)}{n-m}\tbinom{2k}{m} =  \tfrac{1}{2}  \sum_{m'\in \ZZ} (-1)^{m'}\tbinom{2(n-k)}{n-k+m'}\tbinom{2k}{k-m'} = \tfrac{1}{2}S(n-k,k)
			\end{align*}
		where we substituted $m'=k-m$ in the penultimate step, and applied von Szily's identity in the last step. For general terms $T_z$ with $|z|=2k$, up to a global sign global sign accounted for by $(-1)^{z\cdot b}$ in the first equality, a very similar calculation applies. Indeed, for the alternating sequence $s=(0,1,0,\dots)\in (\ZZ/2)^{2n}$ and the walk sequence $w=b+s$ associated to $b$, we have
		\begin{align*}
			(-1)^{z\cdot b} = (-1)^{z \cdot s}(-1)^{z \cdot w} = \mathrm{sign}(z,k)(-1)^k(-1)^{z\cdot w}.
		\end{align*}
		We then replace $\mathcal{W}^{\text{odd}}_{2n,2k}$ and $\mathcal{W}^{\text{even}}_{2n,2k}$ by sets of walks that have an odd resp.\ even number of 1s at those $2k$ fixed positions with $z_i=1$. These sets clearly have the same cardinality as $\mathcal{W}^{\text{odd}}_{2n,2k}$ and $\mathcal{W}^{\text{even}}_{2n,2k}$, respectively. Therefore, we have determined the coefficients of all $T_z$ in the expression of $\omega_{2n}$.
		
\end{proof}
\begin{rmk}
	Capping off the Kirby color $\omega_{2n}$ with $2k$ undotted and $2n-2k$ dotted disks, extracts the coefficient of terms of the form $T_z$ for $|z|=2k$. 
\end{rmk}

In the following, we derive another expression for the Kirby color for $A_\alpha$. For this, we will work with the additional assumption that the underlying field $K$ in $\kk=K[\alpha^{\pm 1}]$ satisfies $\mathrm{char}(K)=0$. 

The symmetric group $S_{m}$ acts on $\dTL_\alpha(0,m)$ as follows. Define the crossing morphism
\begin{align*}
	\cross := \udid - \cupcap\in \dTL_\alpha(2,2) \quad \text{satisfying} \quad {\cross}^{\,2}=\udid,\quad  \dldcross = -\urdcross \quad \text{and}\quad \drdcross = - \uldcross\, .
\end{align*}
Let $1\leq i\leq m-1$. Define for the simple transposition $s_i:=(i,i+1)\in S_m$ the morphism
\begin{align*}
     P_{i}:= \;\crossSimpleTransp\; \in \dTL_\alpha(m,m)
\end{align*}
that is the identity on the first $i-1$ strands, the crossing morphism on $i$ and $i+1$, and the identity on the remaining strands. As composites of these morphisms satisfy the braid relations, the assignment $s_i\mapsto P_i$ extends to a group homomorphism 
\begin{align*}
    S_m \to \dTL_\alpha(m,m),\quad \sigma \mapsto P_\sigma.
\end{align*}
We let $\sigma\in S_m$ act on $\dTL_\alpha(0,m)$ by composing with the morphism $P_\sigma\in \dTL_\alpha(m,m)$.
However, this does not define a symmetric braiding on the category $\dTL_\alpha$ since naturality fails as the dot slides through a crossing only up to sign. 

Define the $m$-th symmetrizer to be the morphism
\begin{align*}
	\Symmet{m}\; :=\; \frac{1}{m!}\sum_{\sigma\in S_m} P_\sigma \in \dTL_\alpha(m,m).
\end{align*}
This satisfies the usual recursive relation
\begin{align*}
	\Symmet{m}\; = \;\frac{1}{m}\;\SymStack{m-1}{{m-1}}\; +\; \frac{m-1}{m}\;\SymStackCross{m-1}{m-1}\;, && \SymOne \;=\; \SymLine\;.
\end{align*}
Using the definition of the crossing, we obtain the recursion relation for the $m$-th Jones--Wenzl projector in $\dTL_\alpha$ and hence, it is equal to the $m$-th symmetrizer. Recall from Remark \ref{rmk:TL-in-dTL} that we can consider the Jones--Wenzl projectors as idempotents in $\dTL_\alpha$. See \cite[Section 3.2]{HRWKirbyColorForKh} for more details in the very similar case of $A_{\mathrm{BN}}$. 

The separability idempotent $\Delta(H^{-1})\in A_\alpha^{\otimes 2}$ can be considered as an element in $\dTL_\alpha(0,2)$, a cup with decoration given by $H^{-1}=\frac{1}{2\alpha}x$. We denote this element by
\begin{align*}
	 \SepidemCup\; :=\; \tfrac{1}{2\alpha}\;\dcupbig\;\in \dTL_\alpha(0,2).
\end{align*}
We also use graphical notation to depict the Kirby color by 
\begin{align*}
	\kirbycol{\omega_{2n}}\; \in \dTL_\alpha(0,2n).
\end{align*}

\begin{thm}\label{thm:KirbySymmetrizer}
	The Kirby color can be expressed as scaled symmetrized power of the separability idempotent.
	\begin{align*}
		\kirbycol{\omega_{2n}}\;= \tfrac{1}{2^n} \tbinom{2n}{n}\; \SymmetKirby{2n} \; \in \dTL_\alpha(0,2n).
	\end{align*}
\end{thm}
\begin{proof}
	We show that expressions in Theorem \ref{thm:KirbyAlpha} and \ref{thm:KirbySymmetrizer} are equal. Since the separability idempotent $\Delta(H^{-1})=\frac{1}{2}( 1\otimes 1 +\alpha^{-1}x\otimes x)$ has an even number of $x$'s, so do its symmetrized powers. By the symmetrizer property it suffices to determine the coefficient of any term $T_z$ with $|z|=2k$. Indeed, swapping $0$ and $1$ in $z$ only results in an overall sign change. The terms $T_z$ have a global factor of $2^{-n}\alpha^{-k}$ from powers of the separability idempotent. Without loss of generality, pick $z=z_0=(0,\dots, 0, 1,\dots,1)\in (\ZZ/2)^{2n}$ with $|z|=2k$. This choice for $z$ satisfies $\mathrm{sign}(z,k)=+1$ and the dot slide relations implement the swap of $0$ and $1$. We compute the coefficient of $T_z$ in the symmetrizer expression as	
	\begin{align*}
		\frac{1}{2^n\alpha^k}\frac{1}{2^n}\binom{2n}{n}\frac{1}{|S_{2n}|} \binom{n}{k} |\mathrm{Stab}(T_z)|=  \frac{1}{2^n\alpha^k}\frac{1}{2^n}\frac{(2n)!}{n!n!}\frac{1}{(2n)!} \frac{n!}{(n-k)!k!} (2k)!(2(n-k))!=\frac{1}{2^{2n}\alpha^k} S(n-k,k).
	\end{align*}
	The factor $\binom{n}{k}$ comes from the other terms in $\Delta(H^{-1})^{\otimes n}$ with $|z|=2k$. A straightforward counting argument shows that the sign $\mathrm{sign}(z,k)$ for a general term $T_z\in A^{\otimes 2n}$ and the sign obtained from swapping 0s and 1s from $z_0$ to $z$ are equal.
\end{proof} 
\begin{exm}
	For $n=1$, we have the separability idempotent
	\begin{align*}
			\kirbyTwo{\omega_{2}} = \tfrac{1}{2}\tbinom{2}{1}\; \SymKirbyTwo{2} =  \tfrac{1}{2\alpha} \left(\tfrac{1}{2}\,\dcupid\, +\tfrac{1}{2}\, \dcupcross\right) = \tfrac{1}{2\alpha}\left(\tfrac{1}{2}\,\dcupid\, +\tfrac{1}{2}\, \dcupid - \tfrac{1}{2}\,\dcupcapcup \right)= \tfrac{1}{2\alpha}\dcupbig = \SepidemCup
	\end{align*}
	where we used that the dotted circle is zero in $\dTL_\alpha$.
	\smallskip 
    
	Let $n=2$. The expression for $\omega_4$ in Example \ref{exm:kirbySmalln} can also be obtained from the symmetrizer of
	\begin{align*}
		\tfrac{1}{4}\tbinom{4}{2}\;\SepidemCup\,\SepidemCup = \tfrac{6}{16\alpha^2}\; \dcupbig\,\dcupbig = \tfrac{1}{16} \left( 6\cdot 1\otimes 1\otimes 1\otimes 1 + \tfrac{6}{\alpha} \cdot 1\otimes 1\otimes x\otimes x + \tfrac{6}{\alpha}\cdot  x\otimes x\otimes 1\otimes 1 + \tfrac{6}{\alpha^2}\cdot x\otimes x\otimes x\otimes x\right).
	\end{align*}
	The terms with $|z|=4$ and $|z|=0$ are invariant. For the other terms, the computation in the proof of Theorem~\ref{thm:KirbySymmetrizer} applies.
\end{exm}
We can now explicitly verify the properties stated Remark \ref{rmk:kirbyRecursiveCapping} for $A_\alpha$.
\begin{cor}\label{cor:annulusCapping}
	The Kirby color satisfies the following properties
	\begin{align*}
		\kirbyCap{\omega_{2n}} = 0\in \dTL_\alpha(0,2n) \quad \quad \text{and} \quad \quad \kirbydCap{\omega_{2n}}=\kirbywide{\omega_{2n-2}}\;\in \dTL_\alpha(0,2n-2).
	\end{align*}
\end{cor}
\begin{proof}
	For the first property, we note that the symmetrizer is annihilated by caps since the symmetrizer is the same as the Jones--Wenzl projector. For the second property, we first establish the following relations in $\dTL_\alpha$ using the definition of the crossing and the saddling relation in $\dTL_\alpha$
	\begin{align*}
		\DottedCapNaturality
        \;=\; 
        \DottedCapNaturalityOne - \DottedCapNaturalityTwo - \DottedCapNaturalityThree + \DottedCapNaturalityFour 
        \;=\; 
        \DottedCapId + \CapDottedId -\DottedDiagCap 
        \;=\; 
        \DiagDottedCap
	\end{align*}
	and similarly for the dotted cup. By this relation and the symmetrizer property of absorbing crossings, we can assume that the dotted cap is on the two rightmost strands. Applying the recursion for the symmetrizer twice, we obtain the following.
	\begin{align*}
		\Symmet{2n}&= \frac{1}{2n(2n-1)} \left[\,\SymmetDoubleId{2n-2}\, + (2n-2)\, \SymStackCrossId{2n-2}{2n-2}\;\right]\\ 
		+ \frac{1}{2n(2n-1)} &\left[ \; \SymStackIdCross{2n-2}{2n-2} + {2n-2}\; \SymStacksTwoCrossDown{2n-2}{2n-2}{2n-2} +{2n-2}\; \SymStacksTwoCrossUp{2n-2}{2n-2}{2n-2} + {(2n-2)^2}\;\SymFourStackCrosses{2n-2}{2n-2}{2n-2}{2n-2}\;\right] 
	\end{align*}
	Note that some of the expressions can be simplified by the idempotency of the symmetrizer. We prove inductively that 
	\begin{align*}
		\SymmetDottedCupsCapped{2n}	= t_{2n}\;\SymmetDottedCups{2n-2}, \quad \text{with } t_{2n} = \frac{2n\alpha}{2n-1}.
	\end{align*}
	For $n=1$, we have 
	\begin{align*}
			 \SymTwoCappedCupped{2}= \DoublyDottedCircle= \alpha\; \UndottedCircle= 2\alpha
	\end{align*}
	since the empty diagram on zero strands is one. Now let $n\geq2$ and assume that
	\begin{align*}
		\SymmetDottedCupsCapped{2n-2} = t_{2n-2}\; \SymmetDottedCups{2n-4}
	\end{align*}
	for some\footnote{Note that we do not assume $t_{2n-2}=\tfrac{(2n-2)\alpha}{2n-3}$ yet to keep generality for Remark \ref{rmk:determineKirbyRecur}.} $t_{2n-2}\in \kk$. We use $n$ dotted cups and a dotted cap on the symmetrizer recursion above to obtain the following expression.
		\begin{align*}
			\SymmetDottedCupsCapped{2n} &= \frac{1}{2n(2n-1)} \left( 2\alpha\;\SymmetDottedCups{2n-2} + (2n-2) \alpha \;\SymmetDottedCups{2n-2}\right) \\
			&+ \frac{1}{2n(2n-1)}\left( 2\alpha \;\SymmetDottedCups{2n-2} + (2n-2)\alpha \;\SymmetDottedCups{2n-2}+ (2n-2) \alpha \;\SymmetDottedCups{2n-2} \right.\\
			&\quad\quad+\left. \left((2n-2)\alpha + (2n-2)(2n-3)t_{2n-2})\right)\,\SymmetDottedCups{2n-2}\right)
		\end{align*}
	Here we have used for the first five terms
	\begin{align*}
		\DoublyDottedCircle=2\alpha\,,\quad \RightTwist=\OneLine\,,\quad \CrossedDottedCircle = \DoublyDottedCircle = 2\alpha,\quad
		\CrossedDottedCap=\DottedCap\,,\quad \CrossedDottedCup=\DottedCup\,.
	\end{align*}
	To obtain the last term, we use another recursion on the symmetrizer in the middle of the last term.
	\begin{align*}
		(2n-2)^2\;\SymmetTripleCrossAndCapped{2n-2} =(2n-2)\; \SymmetNextToTripleCrossAndCapped{2n-3} + (2n-2)(2n-3)\; \SymmetDoubleQuadrupleCrossAndCapped{2n-3}{2n-3}
	\end{align*}
	The larger $(2n-2)$-symmetrizer absorbs the $(2n-3)$-symmetrizers. Then, the first term has coefficient $(2n-2)\alpha$ by naturality of dotted cup and the invariance of the dotted cup under crossing. Applying naturality of the dotted cup to the second term, we have
	\begin{align*}
		(2n-2)(2n-3)\;\SymmetDoubleDCappedDCupped{2n-2}{2n-2} = (2n-2)(2n-3)t_{2n-2}\; \SymmetDottedCups{2n-2}
	\end{align*}
	by the inductive assumption and again by absorption of smaller symmetrizers. From the above computation we obtain
	\begin{align*}
		\SymmetDottedCupsCapped{2n} &=\frac{1}{2n(2n-1)}\left( 4(2n-1)\alpha + (2n-2)(2n-3)t_{2n-2}\right)\;\SymmetDottedCups{2n-2}
	\end{align*}
	and can now compute the coefficients $t_{2n}$ recursively. First, let $\tilde{t}_{2n}= 2n(2n-1)t_{2n}$. Then, 
	\begin{align*}
		\tilde{t}_{2n} = 4(2n-1)\alpha+\tilde{t}_{2n-2},\quad \tilde{t}_2=4\alpha \quad\quad \text{and hence,} \quad\quad \tilde{t}_{2n}=4\alpha \sum_{i=1}^n (2i-1)= 4\alpha n^2.
	\end{align*}
	As a consequence, we obtain
	\begin{align*}
		t_{2n} = \frac{1}{2n(2n-1)}\tilde{t}_{2n}= \frac{2n\alpha}{2n-1}.
	\end{align*}
	Finally, we have by Theorem \ref{thm:KirbySymmetrizer}
	\begin{align*}
		\SymmetKirbyCappedRight{\omega_{2n}} 
		= \tfrac{1}{2^{2n}\alpha^n}\tbinom{2n}{n}\,\SymmetDottedCupsCapped{2n}
		= \tfrac{1}{2^{2n}\alpha^n}\tbinom{2n}{n}t_{2n}\;\SymmetDottedCups{2n-2}= \tfrac{1}{2^{2n-2}\alpha^{n-1}}\tbinom{2n-2}{n-1}\; \SymmetKirby{2n-2} =\kirbywide{\omega_{2n-2}}
	\end{align*}
	and this concludes the proof.
\end{proof}
\begin{rmk}\label{rmk:determineKirbyRecur}
	An alternative way to derive the Kirby color for $A_\alpha$-surface skein theory starts from the properties in Corollary \ref{cor:annulusCapping}. From the first property, it can be derived that the Kirby color admits an expression as the Jones--Wenzl projector $\mathrm{JW}_{2n}$ applied to an element $\dTL_\alpha(0,2n)$ (if we also assume that the idempotent does not have through-degree zero). The cup annihilation property of $\mathrm{JW}_{2n}$ then implies that the element that the Jones--Wenzl projector is composed with must be a linear combination of dotted cups. But the relations in $\dTL_\alpha$ imply that any such configuration have the same image after applying $\mathrm{JW}_{2n}$. The second property fixes the coefficient recursively. In the last step of the proof of Corollary \ref{cor:annulusCapping}, one can also use the coefficients $t_{2n}$ to recursively determine the coefficient in the expression in Theorem \ref{thm:KirbySymmetrizer}. This fully determines $\omega_{2n}$. 
	
	Analogues of these properties hold for more general commutative Frobenius algebras $A$ for which $\sktft_A$ extends to $(4,2)$-handlebodies, see Remark \ref{rmk:kirbyRecursiveCapping}. It is conceivable that these properties lend themselves to a derivation the Kirby color after the structure of idempotents in $\dTL_A(2n,2n)$ has been fully understood. Concretely, this requires characterizing idempotents by Jones--Wenzl-like properties of annihilating caps in $\dTL_A$. 
\end{rmk}

\begin{rmk}
	We expect that the Kirby color for the surface skein theory $\sktft_\beta$ associated to the commutative Frobenius algebra $A_\beta$ does \emph{not} admit a formulation as a symmetrization of a tensor power of the separability idempotent if $N>2$. While the idempotents in the relevant $\dTL_\beta$ require further study, the existence of the symmetrizer for $N=2$ is related to the symmetry w.r.t.\ orientation swap. 
	An oriented version of the constructions in this section would require the longitudinal boundary curves on the solid torus to be alternately oriented. Forgetting orientation would then relate back to our setting. The failure of naturality of the symmetric group action is reflected by the fact that swapping two adjacent boundary curves requires a flip of orientation. The unoriented graphical calculus of (decorated) Temperley--Lieb diagrams does not record this, but perhaps an oriented version would be more natural---just as Khovanov homology in its original formulation for $\sl_2$ is (almost)\footnote{While Khovanov homology is an invariant of \emph{oriented} links, up to grading shift it does not depend on the choice of orientation.} insensitive to orientation. In this sense, the existence of the symmetrizer is a happy accident in rank $N=2$. Skein theory for the Frobenius algebra $A_\beta$ of rank $N>2$ relates to a version with webs and foams for $\mathfrak{gl}_N$. There, orientations play an important role.
\end{rmk}

\subsection{Computation of the invariant in examples}
In the following we compute the invariant $\sktft_\alpha(W)$ for examples of $(4,2)$-handlebodies $W$. We begin with $W=S^2\times B^2$ presented by one 0-handle, no 1-handles, and one 2-handle attached along the 0-framed unknot in $S^3$. The invariant $\sktft_\alpha(S^2\times B^2)$ defines a linear functional on $\Sk_\alpha(S^2\times S^1)$. We obtain the skein module $\Sk_\alpha(S^2\times S^1)$ from the skein module of $S^2\times B^1$ by gluing, introducing further relations. Recall from Lemma \ref{lem:incompInThickSTwo} that the incompressible surfaces in $S^2\times B^1$ are parallel copies of essential spheres $S^2\times \{*\}$. We have a similar statement for incompressible surfaces in and $S^2\times S^1$.
\begin{lem}
    Isotopy classes of closed incompressible surfaces in $S^2\times S^1$ are represented by parallel copies of essential spheres $S^2\times \{*\}$.
\end{lem}
\begin{proof}
    Let $\mathsf{S}\hookrightarrow S^2\times S^1$ be a closed incompressible surface. We can assume that there exists a $t\in S^1$ such that $\mathsf{S}\cap (S^2\times \{t\}) = \emptyset$. Otherwise the intersection would generically be a disjoint union of closed curves on $S^2$ along which we could find a compression disk in $S^2\times \{t\}$. It follows that we can assume that $\mathsf{S}$ embeds in $S^2\times I$. By Lemma \ref{lem:incompInThickSTwo}, we obtain the statement.
\end{proof}
As a consequence of the above, $A_\alpha$-decorated parallel copies of essential spheres span the surface skein module $\Sk_\alpha(S^2\times S^1)$. Recall from the discussion before Proposition \ref{prop:TensorAlgebraSkModSphere} the tubing relations between parallel essential spheres, and the explicit relations \eqref{eq:tubingWith1} and \eqref{eq:tubingWithx} for $A_\alpha$ from the proof of Corollary \ref{cor:alphaSkmodThickSphere}.
Also recall from Corollary~\ref{cor:alphaSkmodThickSphere} that the set $\{\mathsf{S}^k,\mathsf{S}^k\mathsf{D} \mid k\geq 0\}$ forms a $\kk$-basis of the algebra $\Sk_\alpha(S^2\times B^1,\emptyset)$. Since we obtain $S^2\times S^1$ by gluing $S^2\times B^1$ to itself, on the level of skein modules, this computes the trace of the algebra $\Sk_\alpha(S^2\times B^1, \emptyset)$, i.e.\ the coinvariants
\begin{align*}
    \Sk_\alpha(S^2\times S^1) \cong \Tr (\Sk_\alpha(S^2\times B^1,\emptyset)) = \Tr \left(\kk\langle  \mathsf{S},\mathsf{D}\rangle /(\mathsf{S} \mathsf{D} + \mathsf{D} \mathsf{S},
        \mathsf{D} \mathsf{D} + \alpha \mathsf{S} \mathsf{S} - 1)\right).
\end{align*} 
This is the same as the trace of the $\kk$-linear category $\skcat_\alpha(S^2)$ since by Remark \ref{rmk:gluingMorita}, we only need to consider the empty boundary condition here. 
\begin{prop}\label{prop:skModSTwoSOne}
	Let $k\geq 1$. Denote by $\mathsf{S}^k\in \Sk_\alpha(S^2\times S^1)$ the isotopy class of the disjoint union of $k$ parallel copies of the essential sphere $S^2\times \{*\}$ without decoration, and by $\mathsf{D}\in  \Sk_\alpha(S^2\times S^1)$ the class of one dotted sphere $S^2\times \{*\}$. Then, the skein module 	$\Sk_\alpha(S^2\times S^1)$ is spanned by
	\begin{align*}
	 \{\emptyset, \mathsf{D},\mathsf{S}, \mathsf{S}^{2k} \text{ for } k\geq 1\}.
	\end{align*}
\end{prop}
\begin{proof}
	Gluing $S^2\times B^1$ to $S^2\times S^1$ introduces an additional isotopy cyclically permuting the positions of the essential spheres. This is implemented in the trace of $\Sk_\alpha(S^2\times B^1, \emptyset)$ by adding cyclic permutation relations of the words in $\mathsf{S}$ and $\mathsf{D}$. First, observe that for $k\geq 1$,
	\begin{align*}
		\mathsf{S}^{k-1}\mathsf{D}\mathsf{S} = \mathsf{S}^k\mathsf{D} = - \mathsf{S}^{k-1}\mathsf{D}\mathsf{S}
	\end{align*} 
	where the first equality follows from cyclic permutation and the second equality by \eqref{eq:tubingWith1} from the proof of Corollary~\ref{cor:alphaSkmodThickSphere}. Hence $\mathsf{S}^{k}\mathsf{D}=0$ for $k\geq 1$.
	Second, we show that $\mathsf{S}^{2k+1}$ and $\mathsf{S}^{2k-1}$ are linearly dependent. Fix $l\geq 2$. Again by cyclic permutation and applying \eqref{eq:tubingWith1}, we have
    \begin{align*}
       \mathsf{S}^{2k-1}\mathsf{D}^{l} = \mathsf{D}\mathsf{S}^{2k-1}\mathsf{D}^{l-1} = -  \mathsf{S}\mathsf{D}\mathsf{S}^{2k-2}\mathsf{D}^{l-1} = \dots = -\mathsf{S}^{2k-1}\mathsf{D}^{l}
    \end{align*}
    and hence, $\mathsf{S}^{2k-1}\mathsf{D}^{l}=0$.
    As a consequence, by applying \eqref{eq:tubingWithx} from the proof of Corollary~\ref{cor:alphaSkmodThickSphere}, we obtain
	\begin{align*}
		0=\mathsf{S}^{2k-1}\mathsf{D}^2=-\alpha\mathsf{S}^{2k+1} + \mathsf{S}^{2k-1}.
	\end{align*}

	The first argument appears in \cite[Theorem 3.1]{AsaedaFrohman} for the case of $A_{\mathrm{BN}}$. The second argument appears in the proof of \cite[Theorem 5.11]{BoernerDrube}.
\end{proof}
\begin{prop}\label{prop:InvSTwoBTwo}
	The invariant of the $(4,2)$-handlebody $W=S^2\times B^2$ is the linear functional
	\begin{align*}
		\sktft_\alpha(S^2\times B^2) \colon \Sk_\alpha(S^2\times S^1)&\to \kk\\
		\emptyset &\mapsto \ev(\emptyset)^2\\
		\mathsf{D} &\mapsto 0\\
		\mathsf{S}&\mapsto 0 \\
		\mathsf{S}^{2k}&\mapsto \tfrac{1}{2^{2k}\alpha^k} \tbinom{2k}{k}\ev(\emptyset)^2.
	\end{align*}
\end{prop}
\begin{proof}
	Each of the essential spheres is punctured once. For an odd number of parallel spheres, the invariant is zero. Capping off the Kirby color $\omega_{2k}$ with $2k$ undotted disks yields the coefficient of the term $T_z$ with $z=(1,\dots, 1)$. This is equal to $\tfrac{1}{2^{2k}\alpha^k} \tbinom{2k}{k}$. By Proposition \ref{prop:evEulerChar}, there is a scalar $\ev(\emptyset)^2\in \kk^\times$ since $\chi(S^2\times B^2)=2$.
\end{proof}
\begin{rmk}\label{rmk:InvThickenedCircle}
	Compare the linear functional $\sktft_\alpha(S^2\times B^2)$ above with $\sktft_\alpha(B^3\times S^1)$ from Example~\ref{exm:s2s1eval} which also defines a linear functional on $\Sk_\alpha(S^2\times S^1)$. The $(4,2)$-handlebody $W= S^1\times B^3$ is built from one 0-handle and one 1-handle. It defines the invariant
	\begin{align*}
		\sktft_\alpha(B^3\times S^1) \colon \Sk_\alpha(S^2\times S^1)\to \kk,	\quad 	\emptyset \mapsto 1,\quad \mathsf{S}^k\mapsto 0,\quad	\mathsf{D}\mapsto 1,
	\end{align*}
	by using abstract evaluation. Note that this is independent of the choice of $\ev(\emptyset)$ since the Euler characteristic is $\chi(B^3\times S^1)=0$. 
\end{rmk}

Next, we consider the invariant of the $(4,2)$-handlebody $W=S^1\times S^1\times B^2$ which consists of one 0-handle, two 1-handles and one 2-handle. See \cite[Example 5.3]{KirbyTopOf4Man}. First, we determine the skein module $\Sk_\alpha(S^1\times S^1\times S^1)$. 
In the following, we view $M=S^1\times S^1\times S^1$ as the Seifert-fibered manifold over the torus $F=S^1\times S^1$ with projection onto the first two factors $p\colon M\to F$.
\begin{lem}\label{lem:onesidedIncompIn3Torus}
    Any incompressible surface in $S^1\times S^1 \times S^1$ is orientable.
\end{lem}
\begin{proof}
    By \cite[Theorem 3.6]{FrohmanOneSidedIncompressibles}, we obtain that there are no 1-sided, i.e.\ non-orientable, surfaces in $S^1\times S^1\times S^1$, since the Euler class of the trivial bundle $p\colon M\to F$ is zero.
\end{proof}
\begin{lem}\label{lem:IncompIn3Torus}
    Isotopy classes of incompressible surfaces in $S^1\times S^1\times S^1$ are configurations of parallel tori, parameterized by their first homology class $(k,l,m)\in H_1(S^1\times S^1\times S^1;\ZZ)\cong \ZZ\oplus \ZZ\oplus \ZZ$ up to global sign.
\end{lem}
\begin{proof}
    By Lemma \ref{lem:onesidedIncompIn3Torus}, every incompressible surface in $S^1\times S^1\times S^1$ is orientable, i.e.\ 2-sided. It follows from a theorem of Waldhausen \cite[Satz 2.8]{WaldhausenSeifertFib} that isotopy classes of 2-sided incompressible surfaces in the Seifert-fibered 3-manifold $S^1\times S^1\times S^1$ are represented by surfaces $S$ such that either $S=p^{-1}(p(S))$, or $p|_S$ is a covering map\footnote{In \cite{FrohmanOneSidedIncompressibles} these are called \emph{vertical} and \emph{horizontal}, respectively. Note that the notion of horizontal is different in \cite[Definition 2.5]{WaldhausenSeifertFib}}. These are classified by their first homology class $(k,l,m)\in H_1(S^1\times S^1\times S^1;\ZZ)$ up to global sign, i.e $(k,l,m)\simeq (-k,-l,-m)$, and represented by $d=\gcd(k,l,m)$ parallel copies of incompressible tori of \emph{slope} $(\frac{k}{d},\frac{l}{d},\frac{m}{d})$.  
    The special case of the zero homology class $(0,0,0)\in H_1(S^1\times S^1\times S^1; \ZZ)$ corresponds to the empty surface. Also compare \cite[Section 3]{AsaedaFrohman}.
\end{proof}
\begin{prop}\label{prop:SkModThreeTorus}
	Let $(p,q,r)\in \ZZ^3/(\pm1)$ be a triple of integers up to global sign with $\gcd(p,q,r)=1$. Denote by $\mathsf{T}^d_{(p,q,r)}\in\Sk_\alpha(S^1\times S^1\times S^1)$ the isotopy class of $d$ parallel copies of incompressible tori without decoration. Denote by $\mathsf{D}^d_{(p,q,r)}$ the class of $d$ parallel copies of incompressible dotted tori. Then the skein module $\Sk_\alpha(S^1\times S^1\times S^1)$ is spanned by
	\begin{align*}
	\{ \emptyset, \mathsf{D}_{(p,q,r)}, \mathsf{T}_{(p,q,r)}, \mathsf{T}^{2k}_{(p,q,r)} \mid k\geq 1, (p,q,r)\in \ZZ^3/\{\pm1\}, \gcd(p,q,r)=1\}.
	\end{align*} 
\end{prop}
\begin{proof}
	The tubing relations \eqref{eq:tubingWith1} and \eqref{eq:tubingWithx} generalize to $\mathsf{T}_{(p,q,r)}$ and $\mathsf{D}_{(p,q,r)}$ and can be used to reduce the generators as in Corollary~\ref{cor:alphaSkmodThickSphere}. Using the analogous cyclic permutation relation from Proposition \ref{prop:skModSTwoSOne}, it can also be shown that $\mathsf{T}^{2k+1}_{(p,q,r)}$ and $\mathsf{T}_{(p,q,r)}$ are linearly dependent. 
\end{proof}

\begin{defn}
    Let $n\geq1$, $d\mid 2n$ and $r=2n/d$. We define the \emph{$(d,r)$-cyclic toric cap value} of the Kirby color $\omega_{2n}$ to be the abstract evaluation
    \begin{align*}
        \mathrm{cap}^{(1)}_{d,r}(\omega_{2n}):= \eval_\alpha(S,P)(\omega_{2n})\in \kk
    \end{align*}
    from Construction \ref{con:abstractEvalext}, where $S$ is the surface with $d$ connected components $C_i$ each of which is a torus, and $P$ with $|P|=dr=2n$ has $r$ points on each $C_i$. The tensor factors corresponding to $C_i$ in $A^{\otimes P}$ are at the positions
		\begin{align*}
			i,\quad i+d, \quad i+2d,\quad \dots, \quad i+(r-1)d \quad \text{for } 1\leq i \leq d.
		\end{align*}
\end{defn}
\begin{rmk}
    The superscript on $\mathrm{cap}_{d,r}^{(1)}(\omega_{2n})$ stands for the genus. One could generalize this to cyclic cap values of genus $g$ surfaces. 
\end{rmk}

\begin{prop}\label{prop:InvThickenedTorus}
	The invariant $\sktft_\alpha(S^1\times S^1\times B^2)$ is the linear functional
\begin{align*}
	\sktft_\alpha(S^1\times S^1\times B^2)\colon \Sk_\alpha (S^1\times S^1 \times S^1) &\to \kk\\
	\emptyset & \mapsto 1\\ 
	\mathsf{D}_{(p,q,r)}&\mapsto 0\\
	\mathsf{T}_{(p,q,r)}&\mapsto 0\\
	\mathsf{T}^{2k}_{(p,q,r)}&\mapsto \mathrm{cap}^{(1)}_{2k,|r|}(\omega_{2k|r|})
\end{align*}
\end{prop}
\begin{proof}
	Note that since $\chi(S^1\times S^1\times B^2)=0$, there is no dependence on $\ev(\emptyset)\in \kk^\times$. Hence, the empty skein evaluates to 1. The invariant is zero on $\mathsf{D}$ and $\mathsf{T}$ since the number of intersections is odd. For $2k$ parallel copies of undotted tori $(p,q,r)$, we have $2k|r|$ punctures in cyclic order of genus 1 surfaces.	
\end{proof}

\bibliographystyle{alpha}
\bibliography{allrefs}

\end{document}